\documentclass[11pt, a4paper]{article}

\usepackage{geometry}
\usepackage[english]{babel}
\usepackage[utf8]{inputenc}

\usepackage{amssymb}

\usepackage{bbm}
\usepackage{amsthm}
\usepackage{amsmath}

\theoremstyle{plain}
\newtheorem{thm}{Theorem}[section]
\newtheorem{lem}[thm]{Lemma}

\newtheorem{cor}[thm]{Corollary}

\theoremstyle{definition}

\newtheorem{remark}{Remark}[section]

\newtheorem{assumption}{Assumption}
\newenvironment{customass}[1]
{\renewcommand\theassumption{#1}\assumption}
{\endassumption}

\usepackage{subcaption}
\usepackage{graphicx}
\makeatletter
\newcommand*\bigcdot{\mathpalette\bigcdot@{.5}}
\newcommand*\bigcdot@[2]{\mathbin{\vcenter{\hbox{\scalebox{#2}{$\m@th#1\bullet$}}}}}
\makeatother

\usepackage{extarrows}
\usepackage{amsfonts}
\usepackage[colorinlistoftodos]{todonotes}
\usepackage[rightcaption]{sidecap}

\title{Bayesian Inference on Volatility in the Presence of Infinite Jump Activity and Microstructure Noise}

\author{Qi Wang, Jos\'{e} E. Figueroa-L\'{o}pez\footnote{Corresponding author}, Todd Kuffner}

\date{Department of Mathematics and Statistics,\\
Washington University in St. Louis}

\begin{document}
\maketitle

\begin{abstract}
	Volatility estimation based on high-frequency data is key to accurately measure and control the risk of financial assets. A L\'{e}vy process with infinite jump activity and microstructure noise is considered one of the simplest, yet accurate enough, models for financial data at high-frequency. Utilizing this model, we propose a ``purposely misspecified" posterior of the volatility obtained by ignoring the jump-component of the process.
	The misspecified posterior is further corrected by a simple estimate of the location shift and re-scaling of the log likelihood. Our main result establishes a Bernstein-von Mises (BvM) theorem, which states that the proposed adjusted posterior is asymptotically Gaussian, centered at a consistent estimator, and with variance equal to the inverse of the Fisher information. 
	In the absence of microstructure noise, our approach can be extended to inferences of the integrated variance of a general It\^o semimartingale. 
	Simulations are provided to demonstrate the accuracy of the resulting credible intervals, and the frequentist properties of the approximate Bayesian inference based on the adjusted posterior.
\end{abstract}

\textit{MSC 2010 subject classifications}: Primary 62M09; secondary 62F15.

\textit{Keywords and phrases}: Bernstein-von Mises theorem, Semiparametric inference, It\^o semimartingales, high-frequency based inference, Microstructure noise

\section{Introduction}\label{sec:intro}

In the past decade, jumps have played an increasingly important role in asset price modeling. The necessity of jumps is supported by both empirical and realistic considerations such as (i) sudden and relatively large changes observed in real stock prices; (ii) the implied volatility smile phenomenon, which is more pronounced for short maturity options; and (iii) the proper management of risk \cite{peterintro, Cont:2003}. Though early on (e.g., Merton's model) the attention was centered on finite-jump activity models (i.e., those exhibiting finite jumps in finite time intervals), infinite-activity models are considered more realistic as suggested by many studies based on real asset returns \cite{infjump, Li2008Bayesian, Szerszen2009Baye, Yu2011MCMC,Yang2017jump}. Here we consider a one-dimensional  L\'{e}vy process $X=\{X_{t}\}_{t\geq{}0}$ defined on some probability space $(\Omega, \mathcal{F}, (\mathcal{F}_t)_{t\ge 0}, P)$ over a fixed time horizon $t \in [0, T]$ , which is a fundamental and widely-used tool to model jump processes with infinite activity. Concretely,
\begin{equation}
X_t := \mu t+ \theta^{1/2} W_t +J_t, \label{model}
\end{equation}
where $\mu\in\mathbb{R}$ and $\theta\in[0,\infty)$ are the drift and the variance parameters, respectively, $W=\{W_t\}_{t\ge 0}$ is a Wiener process, and $J=\{J_t\}_{t\ge 0}$ is an independent pure-jump L\'evy process. In financial applications, $X_{t}$ typically represents the log-return or log-price process $\log(S_{t}/S_{0})$ of an asset with price process $\{S_{t}\}_{t\geq{}0}$. In that case, the parameter $\sigma=\theta^{1/2}$ is called the volatility of the process and contributes to the total ``variability" of the process $X$. Further details about the model and its components are given in \S~\ref{sec:setup}.

With improvements in computational power and the advent of electronic-based financial markets, high-frequency data (every minute, second, or even nanosecond) has become widely available. While exploiting the convenience of massive data, we also suffer from market microstructure frictions (e.g., serial autocorrelation, price discreteness, and temporary demand-supply imbalance) caused by the nature of trading at high frequency. In an attempt to explain the nature of tick-by-tick data, \cite{zhou1996noise} and \cite{TSRV} introduced the concept of microstructure noise, in which the observed transaction log-price $Y_{t}$ at time $t$ is a noisy measure of an underlying ``efficient" log-price $X_t$:
\begin{equation}
Y_{t}=X_{t}+\varepsilon_{t}=\mu t+ \theta^{1/2} W_t +J_t+\varepsilon_{t}. \label{modelWithNoiseb}
\end{equation}

Our purpose is to estimate the variance parameter $\theta$ based on high-frequency sampling observations $Y_{t_{0}},Y_{t_{1}},\dots,Y_{t_{n}}$ ($0=t_{0}<\dots<t_{n}=T$) of the process over a fixed period of time $[0,T]$.
From the perspective of frequentist point estimation, when there is no microstructure noise, \cite{mancini2009} proposed a consistent estimator by eliminating  those increments of the process,  $\Delta_{i}Y:=Y_{t_{i}}-Y_{t_{i-1}}$, which are larger in absolute value than a suitably defined threshold. The asymptotic efficiency of the estimator with the restriction of a bounded variation jump process $J$ is proved later in \cite{cont2011}. When the microstructure noise is taken into account but jumps are not present, several estimators have been proposed. The two-scale estimator in \cite{TSRV} considered two different estimation scales of the process to estimate and eliminate the effect of the noise. The preaveraging approach in \cite{JACOD2009preaveraging} replaced the increments $\Delta_{i}Y$ with a weighted summation over a small window. The kernel method in \cite{Barndorff-NielsenKernel} utilized the weighted realized autocovariances. When both noise and jumps are present, \cite{PODOLSKIJ2009IA, podolskij2009} introduced the modulated bipower variation estimator using the bipower variation of the weighted average of the increments. The estimator is consistent, but cannot achieve the efficient convergence rate $n^{-1/4}$, which represents the best rate that can be achieved for the estimation problem in presence of noise and jumps. \cite{CHRISTENSEN2010} proposed two quantile-based realized volatility estimators by employing empirical quantiles of the averaged returns. The estimators are both consistent and asymptotically efficient, but only applicable for processes with finite jumps. More recently, \cite{Jing2014jumpest} combined the preaveraging method of \cite{JACOD2009preaveraging} and the thresholding ideas of \cite{mancini2009} to construct a consistent estimator of the integrated variance that is robust to both noise and infinite jump activity. The details of this estimator are explained in \S~\ref{sec:correct}. 

Whereas there are numerous frequentist estimators available, the development of an explicit and efficient Bayesian approach which can accommodate high-frequency data remains a largely open problem.
For a fully Bayesian approach, the joint posterior of the parameters must be derived based on the full likelihood function and a joint prior distribution for all the parameters in the model. Then, integrating the joint posterior over the nuisance parameters (in our case the parameters related to the jump component $J$ and microstructure noise $\varepsilon$) yields a marginal posterior distribution for the parameter of interest. Since the posterior is often intractable, Markov Chain Monte Carlo (MCMC) methods are typically used to sample from the joint posterior, and then numerical integration over the nuisance parameters is achieved by simply ignoring the corresponding MCMC output for those parameters. MCMC-based Bayesian methods have been applied to the volatility estimation problem by several studies. \cite{Jones99bayesianestimation} and \cite{Eraker2003impact} used MCMC for a diffusion process augmented by a Poisson jump process. More recently, additional model complexity has been accommodated by taking infinite activity into consideration. \cite{Yu2011MCMC} proposed an MCMC estimation method using both spot and option prices. Their jumps are assumed to follow either a variance gamma process or an $\alpha$-stable process. \cite{jasra2011a} developed an automated sequential Monte Carlo algorithm by adding an additional re-sampling step for variance gamma jumps. \cite{he2014efficient} applied a slice sampling approach with a similar variance gamma assumption. \cite{Griffin2016} incorporated realized variation and realized power variation into a MCMC procedure, and analyzed a generalized variance gamma process. \cite{Yang2017jump} considered both returns and the Chicago Board of Options Exchange (CBOE) Volatility Index (VIX) to obtain the posterior for the jump part. The variance gamma process and normal inverse gamma process were considered. 

Although the papers mentioned above considered Bayesian inference derived from the joint posterior, they all require strong assumptions about the structure of the jumps, which severely limit the practical value of these methods. Without these simplifying assumptions, it is quite challenging to write down the full likelihood function under the semi-parametric setting \eqref{model}, which means that it is also difficult to obtain the full joint posterior without such assumptions. One of these assumptions is the choice of a particular specification of the jump process $J$, among many possible jump processes. However, empirical results in \cite{Li2008Bayesian}, \cite{Yu2011MCMC}, and \cite{Kou2017} suggested that different jump assumptions lead to different estimation results for volatility. The posterior depends heavily on the structure of the jumps. Thus, sticking to just one jump type increases the possibility of misspecification and, therefore, can lead to inaccurate estimation and inference. 

Moreover, specifying and calculating the distribution of the jump component may incur heavy computational costs, especially when working with high-frequency data. For this reason, nearly all of the aforementioned studies consider only daily returns data. Some literature like \cite{jasra2011a} and \cite{Griffin2016} did apply their methods to hourly data and 5-minute data, respectively. However, they both fixed one of the parameters of the jump process as constant, in order to reduce the computational load.

The difficulties of deriving the posterior and the associated heavy computational costs are mainly caused by the jumps,  which are only related to the nuisance parameters. Our target of estimation, the variance or volatility, is not affected by the jumps, and is modeled by a simple Gaussian process, for which Bayesian inference can be more easily obtained. Based on this observation, one plausible idea to tackle the problem is to ignore the nuisance parameters in the nonparametric part of the process, replace the nuisance parameters in the parametric part by their consistent estimators, and construct a posterior only for the parameter of interest. The advantages of such an approach are that one need not specify a prior on the jump process, and it is not necessary to obtain samples from the full joint posterior. By contrast, we will directly obtain an approximation to the marginal posterior for the volatility, which we will show can be used for accurate Bayesian inference. This approach was recently used by \cite{Martin}. They derived a `purposely misspecified' posterior for a jump-diffusion model with constant volatility, finite jump activity and without microstructure noise, which targets the parameter of interest, the volatility, directly. Using a misspecified model on purpose, the inherent difficulty of specifying the likelihood function in a nonparametric model is tackled by omitting the complicated nuisance component of the model. The bias and the inaccurate variance caused by the misspecification are later corrected by applying a location shift and rescaling the likelihood using a Gibbs posterior. They showed that the adjusted posterior possesses good asymptotic properties, as guaranteed by a Bernstein-von Mises theorem.

In this paper, we study a `purposely misspecified’ posterior for the variance $\theta$ of the model \eqref{modelWithNoiseb} either with or without microstructure noise, which is a considerably more difficult and realistic setting in comparison to the finite jump activity model without microstructure noise that was studied by \cite{Martin}. Our main result is a Bernstein-von Mises Theorem for the adjusted posterior for the volatility parameter, which shows that the proposed posterior is asymptotically normal and centered at a consistent estimator, and with variance shrinking at rates $n^{-1/2}$ and $n^{-1}$, respectively, depending on whether a microstructure noise is incorporated or not in the model.

The novel contributions of this paper can be summarized as follows. First, we allow the jump process to be any L\'{e}vy process with bounded variation, i.e. there is no parametric assumption about the nuisance component, and no assumption of finite jump activity. We also allow for an additive microstructure noise in the data. These relaxations of the stronger assumptions made in the existing literature help to alleviate inaccuracies introduced by model misinterpretation, and also avoid expensive computational costs. In fact, we also show that in the situations when the microstructure noise can be ignored (e.g., when working with medium range frequencies), our approach can be extended to the estimation of the integrated variance of a general It\^o semimartingale $X$. In particular, we allow stochastic volatility and a general pure-jump semimartingale component $J$.

It is important to remark that our proposed inference procedure is among the first Bayesian approaches that can accommodate truly high-frequency data; due to high computational costs and lack of theoretical performance guarantees, most of the existing literature involves methods which are only applicable to low frequency data, such as daily observations.  Finally, our results suggest that, under certain circumstances, misspecification on purpose can serve as a vehicle for accurate approximate Bayesian inference about low-dimensional interest parameters in complex, possibly infinite-dimensional models.

The paper is organized as follows. A detailed description of the setting and model are provided in \S~\ref{sec:setup}. Differences between finite and infinite activity when deriving the `purposely misspecified' posterior are highlighted in \S~\ref{sec:comparefa}. This analysis reveals the importance of proposing a modified version of the Bernstein-von Mises theorem, which is stated in \S~\ref{sec:bvm}. The misspecified model is presented  in \S~\ref{sec:misspecified}, and further extended in \S~\ref{Sect:Extensions}. The main results are stated in \S~\ref{sec:theorem} and \S~\ref{sec:correct}. Simulation results given in \S~\ref{sec:simu}  illustrate the performance of our procedures. Discussion and concluding remarks are in \S~\ref{sec:conclusion}. The proofs and further technical details appear in the Appendix.

\section{Model setup}\label{sec:setup}

As mentioned in \S~\ref{sec:intro}, we consider a one-dimensional continuous-time process defined on some probability space $(\Omega, \mathcal{F}, (\mathcal{F}_t)_{t\ge 0}, P)$ over a fixed and finite time horizon, $X=\{X_t; t\in [0,T]\}$, which is assumed to follow model \eqref{model}.
It consists of a drift part with constant coefficient $\mu\in \mathbb{R}$, a diffusion part with constant coefficient $\theta\in \mathbb{R}_{+}$, which represents the volatility or variance, and a pure jump part $J = \{J_t\}_{t\ge 0}$. The parameter space for $\theta$, denoted as $\Theta$, is assumed to be a bounded and open subset of $(0,+\infty)$ such that $0\notin \bar{\Theta}$.

The jump component $J$ is assumed to be a pure jump L\'{e}vy process, which is used in many fields of science. In mathematical finance, a L\'{e}vy process is widely recognized to provide a better fit to intraday returns than plain Brownian motion. A comprehensive overview of the applications of L\'{e}vy processes can be found in \cite{levyprocessbook} and \cite{Cont:2003}. A L\'evy processes is defined as a c\`{a}dl\`{a}g, real valued stochastic process, which has independent and stationary increments, and is stochastically continuous. It is known that a L\'evy process $X$ takes the general form (\ref{model}) with $J$ defined as 
\begin{equation}\label{LIDcmp}
\begin{aligned}
J_{t} = J_{1t} + \tilde{J}_{2t} , \quad J_{1t} = \int_{0}^{t}\int_{|x|>1} x\,\mu(dx, ds), \\
\quad \tilde{J}_{2t} = \int_{0}^{t}\int_{0<|x|\le 1} x\,(\mu(dx, ds)-v(dx)ds),
\end{aligned}
\end{equation}
where $\mu$ is a Poisson random measure on $\mathbb{R}_{+}\times \mathbb{R}\backslash\{0\}$ with mean measure $\nu(dx)dt$ such that $\int_{\mathbb{R}\backslash\{0\}}(|x|^{2}\wedge{}1)\nu(dx)<\infty$. This is the so-called L\'{e}vy-It\^{o} decomposition of $X$ and $\nu$ is called the L\'evy measure of $X$. 

The contaminated process, which equals to $X_t$ plus a noise component $\varepsilon_t$, is  observed at equally-spaced discrete times $\{0=t_0< t_1<\ldots<t_n=T\}$, $t_j - t_{j-1} = \Delta_n = T/n$. More specifically, we observe
\begin{equation}\label{model-XY}
Y_{t_j} = X_{t_j} + \varepsilon_{t_j}, \quad t_j=j\Delta_{n}, \; j=0,\dots,n.
\end{equation}
The data is assumed to be generated by the model (\ref{modelWithNoiseb})-(\ref{model-XY}) with true volatility value $\theta^{*}$, which is the target to be estimated. The L\'{e}vy model with microstructure noise (\ref{modelWithNoiseb}) is considered one of the simplest, yet accurate enough, models for financial data at high-frequency. For an assessment of its empirical accuracy, we refer to \cite{FLKiseop2015}.

The process $Y$ satisfies the following assumptions: 
\begin{customass}{(N)}\label{ass:noise}\hfill
	\begin{enumerate}
		\item The microstructure noise components,
		$\varepsilon = \{\varepsilon_{t_j}\}_{j=1}^n$, are independent and identically distributed (i.i.d.), and follow a  $\mathcal{N}(0,\sigma_{\varepsilon}^2)$ distribution. In Bayesian framework, we assume that the i.i.d. holds true conditionally on the unknown parameter $\sigma_{\varepsilon}$.
		\item The processes $\varepsilon$ and $X$ are independent. 
	\end{enumerate}	
\end{customass}

\begin{customass}{(JD)}\label{ass:bddvar}
	The Blumenthal-Getoor index $\alpha$ of $J$ is  less than $1$:
	\begin{equation}\label{DgnBG1}
	\alpha = \inf\left\{p>0:\int_{|x|\le 1} |x|^p\nu(dx)<\infty\right\} < 1.
	\end{equation}
	In particular, this implies that the paths of the process $J$ are of bounded variation, almost surely.
\end{customass}

\begin{customass}{(JF)}\label{ass:finite16}
	The process $J$ has a finite $16$th moment. Combined with \ref{ass:bddvar}-1, this assumption is equivalent to $$\int_{|x|\ge 1} x^{16}\,\nu(dx)<\infty.$$
\end{customass}

\begin{remark}\hfill
	\begin{enumerate}
		\item \cite{hansen2006moderatefreq} suggested  that the independence assumption for $\varepsilon$ and $X$ is reasonable for moderate intraday frequency (e.g. 1 minute).
		\item For a L\'evy process, 
		the Blumenthal-Getoor index $\alpha$ controls the small jump activity of the process: it becomes larger as the small jumps are more persistent. The assumption of $\alpha<1$ is inspired by \cite{cont2011} and \cite{JacodSum}, and used later in \S~\ref{sec:correct} to apply a central limit theorem (CLT) for a threshold estimator of the volatility. \cite{JacodSum} concluded that when $\alpha\ge 1$, there is no CLT in general for a realized quadratic threshold estimator of the integrated variance. Its rate of convergence to the integrated variance is much slower than $n^{-1/2}$. A detailed proof of both the CLT and lack-of-CLT can be found in \cite{cont2011}.  A similar bounded variation assumption also appear in previous studies, such as \cite{CHRISTENSEN2010}, \cite{cont2011}, \cite{JacodSum}, and \cite{Jing2014jumpest}.
		\item It is important to remark that in the absence of microstructure noise, we can take a stochastic volatility model and much more general pure-jump semimartingales $J$ of bounded variation (see \S~\ref{Sect:Extensions}). We also don't require the condition \ref{ass:finite16}.
	\end{enumerate}
	
\end{remark}   
For future reference, let us recall the following common notation for the increments and jumps of an  arbitrary continuous-time c\`adl\`ag process $\{U_{t}\}_{t\geq{}0}$:
\[
\Delta_{i}U=\Delta_{i}^{n}U=U_{t_{i}}-U_{t_{i-1}},\quad \Delta U_{t}=U_{t}-U_{t^{-}}.
\]

\section{Comparison with finite jump activity models}\label{sec:comparefa}
In this section, we present a motivating example using a simpler finite jump activity model, in order to illustrate the usefulness of the approximate Bayesian inference obtained via purposeful misspecification. \cite{Martin} proposed this approach, but did not make comparisons to the true marginal posterior for the volatility parameter. In the next subsection, we provide this comparison through a simulation experiment. We then summarize the theoretical results in \cite{Martin} and explain what issues arise when considering the more complicated and realistic setting of infinite jump activity.

\subsection{An illustration through simulation}
We first empirically compare the ``purposely misspecified" posterior from \cite{Martin}  with a marginalized full Bayesian posterior.  The goal of the comparison is to assess the accuracy of the former method and to motivate our approach. A simple jump diffusion model without noise is considered. Concretely, model \eqref{model} is used with a compound Poisson jump process:
\[
J_{t}=\sum_{i=0}^{N_{t}}\xi_{i}.
\]
Here, $N=\{N_{t}\}_{t\geq{}0}$ is a Poisson process with rate $\lambda$, and $\{\xi_i\}_{i\geq{}1}$ are i.i.d. random variables independent of $N$ and $W$. We assume that $\{\xi_i\}_{i\geq{}1}$, which represent the jump sizes, follow a uniform distribution $U(-1,1)$. This assumption enables us to derive a joint posterior and perform Gibbs sampling for the parameters $ \Theta = (\mu, \theta, \lambda)$. The other parameters and settings are  inherited from \cite{Martin}:
$$\lambda=5,\quad \mu = 1,\quad \theta = 10,\quad n = 5000,\quad T = 1.$$
For simplicity, in what follows, we approximate the Poisson process by a Bernoulli process; namely, ${N}$ is assumed to be a point process such that $P[{N}_{t_{i}}-{N}_{t_{i-1}}=1]=\lambda \Delta_{n}$ and $P[{N}_{t_{i}}-{N}_{t_{i-1}}=0]=1-\lambda \Delta_{n}$.

The joint posterior density based on the data $\pmb{X}^{(n)}=(\Delta_1 X,\dots, \Delta_n X)$ can be written as
\begin{align*}
p(\Theta| \pmb{X}^{(n)}) \propto &\prod_{i=1}^{n} \frac{1}{\sqrt{\pi\theta\Delta_n}} \exp\left\{ -\frac{(\Delta_i X-\mu\Delta_{n})^2}{2\theta\Delta_{n}} \right\}(1-\lambda\Delta_{n})p(\Theta) \\
&+ \prod_{i=1}^{n} \int_{-1}^{1} \frac{1}{\sqrt{2\pi\theta\Delta_n}} \exp\left\{ -\frac{(\Delta_i X-y-\mu\Delta_{n})^2}{2\theta\Delta_{n}} \right\}\,dy \cdot \lambda\Delta_{n} p(\Theta).
\end{align*}

The priors chosen for $\mu, \theta, \lambda$ are a standard Gaussian distribution, an inverse gamma distribution, and a beta distribution, respectively. The posterior for $\theta$ is estimated by two methods:  (i) Gibbs sampling from the full joint posterior, followed by numerical integration to yield the marginal posterior (i.e., we simply ignore the MCMC output for the nuisance parameters $\mu$ and $\lambda$); and (ii) a direct posterior for $\theta$ obtained by purposeful misspecification. We emphasize that the Gibbs sampling approach, which is exact modulo finite simulation error, is only available here because of the very strong assumptions made regarding the jump process. This method is not available for the most complicated and realistic settings we consider in this paper. The second method is an approximation using a misspecified model to directly obtain a posterior for $\theta$ without the need to first obtain the full joint posterior and marginalize. The latter method, as shown in this paper, works quite well even in much more complicated and realistic settings than those considered in this section. Figure \ref{FAtruea}-\ref{FAtrueb} compares the two approaches. Figure \ref{FAtruea} shows  the posteriors for 10 different simulations. The `purposely misspecified' posterior typically resembles quite well the Gibbs distribution of the samples simulated from the joint posterior, which is supposed to recover the correct posterior for the volatility through marginalization of the joint one. Both posteriors center around the true volatility. The 95\% highest posterior density intervals are shown in Figure \ref{FAtrueb}. The similarities of the two empirical posteriors as well as their credible intervals demonstrate the accuracy of the `purposely misspecified' posterior, and therefore, the validity of the inference based on it. 
\begin{figure}[ht]
	\begin{subfigure}{0.42\textwidth}
		\includegraphics[width=\linewidth]{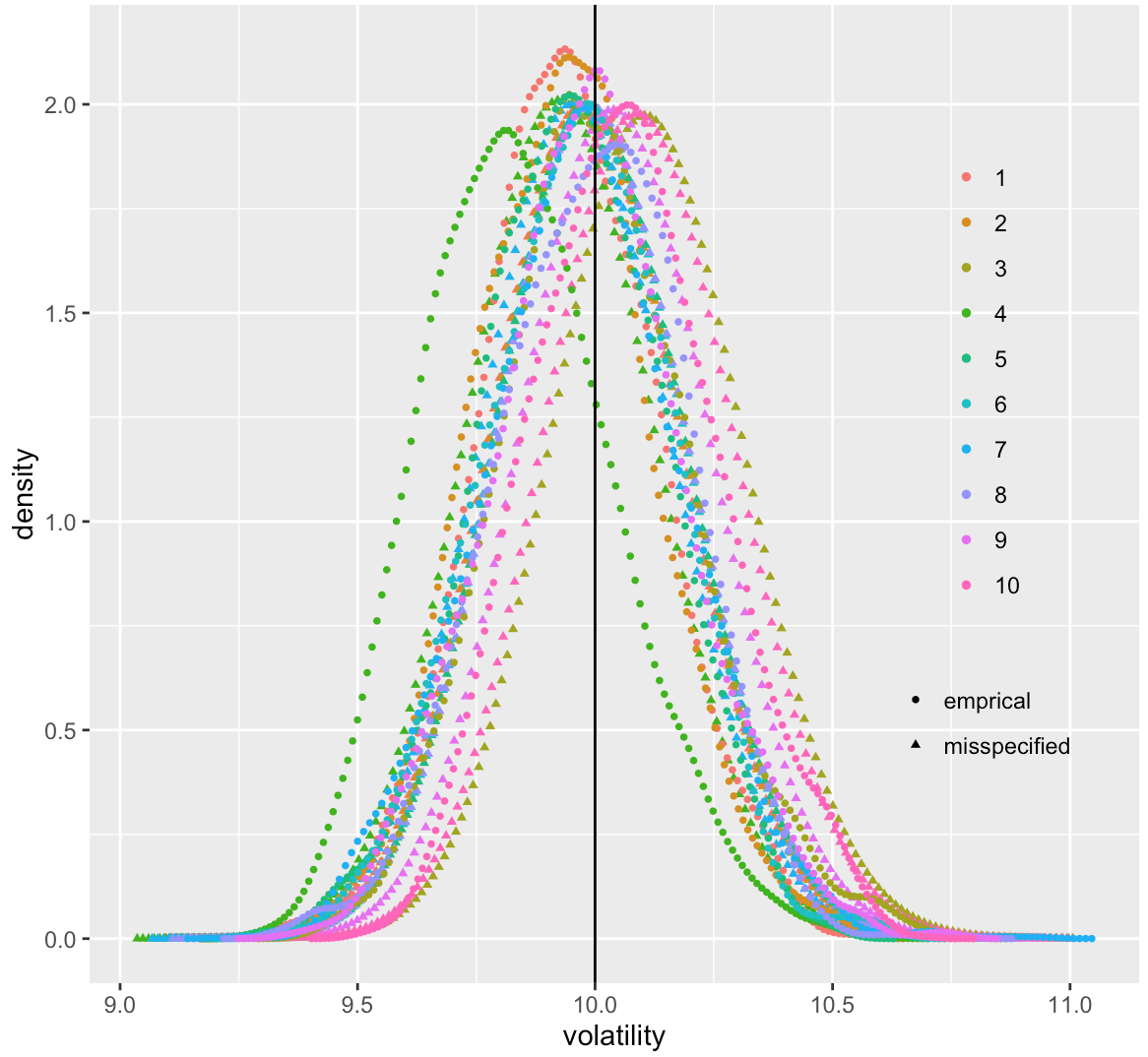} 
		\caption{Posteriors} \label{FAtruea}
	\end{subfigure}
	\hspace*{\fill} 
	\begin{subfigure}{0.56\textwidth}
		\includegraphics[width=\linewidth]{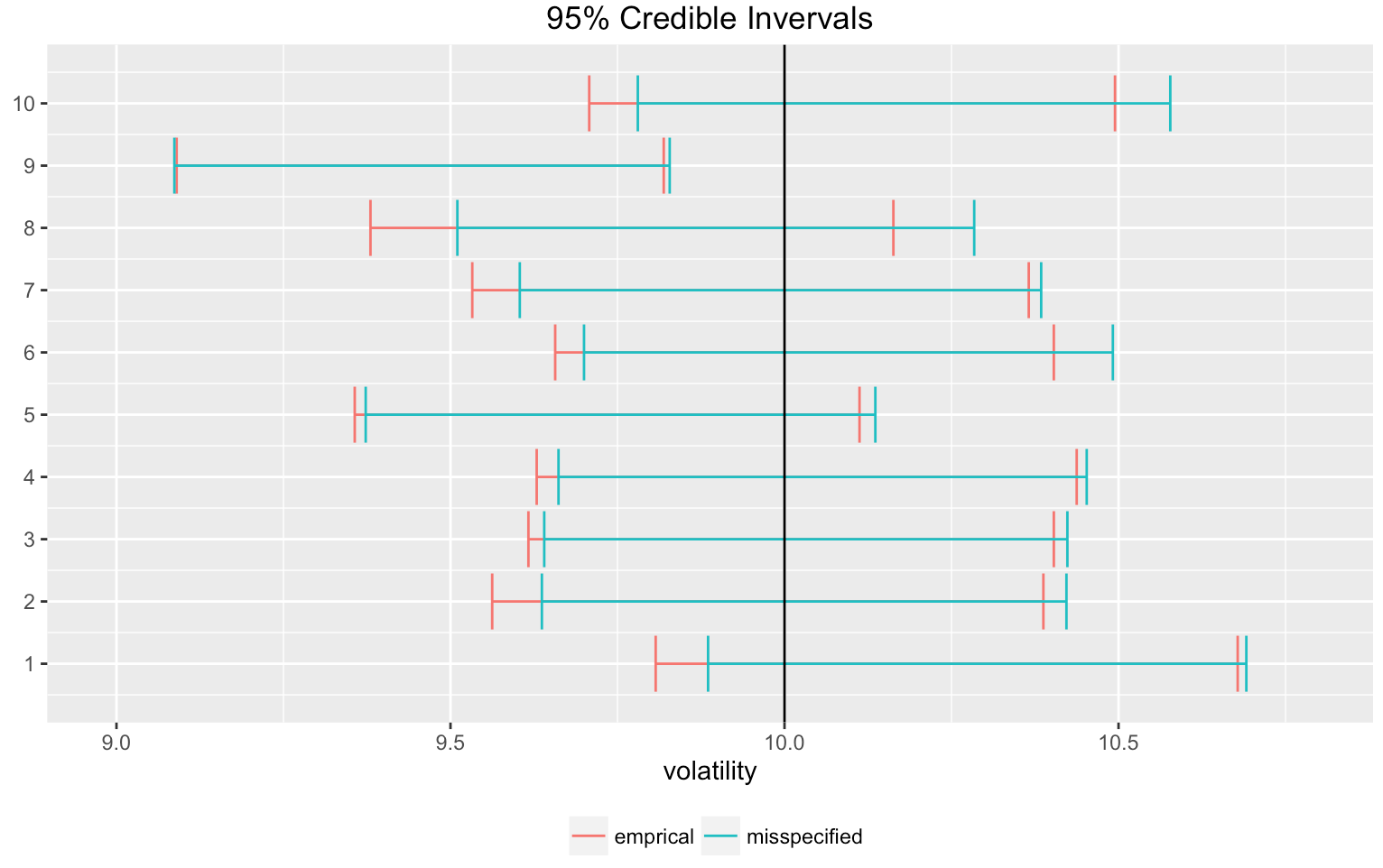}
		\caption{Credible Intervals} \label{FAtrueb}
	\end{subfigure}
	\caption{\textbf{Comparison with empirical posterior.} \footnotesize  (a) 10 different processes distinguished by 10 different colors are generated and the corresponding posteriors are plotted. Each color has two distributions. The one formed by little triangle is the misspecified posterior, while the other represents the Gibbs sampling results. (b) The red lines represent the 95\% highest posterior density (HPD) intervals calculated from the MCMC samples from the joint posterior. The blue lines are the 95\% HPD intervals for the `purposely misspecified' posterior. }
\end{figure}

In general, it is quite complicated to perform fully Bayesian analysis for infinite jump activity models based on  high-frequency data because of the lack of tractable joint posteriors. To perform MCMC sampling from those joint posteriors, some studies (e.g. \cite{Li2008Bayesian}) consider the unobserved jump increments $\Delta_i J$ as a latent parameter. However, with high frequency data, this may cause numerical difficulties. Taking model 3 in \cite{mancini2009}, for example, a variance gamma jump component $J$ is utilized with drift $-0.2$, variance $0.2$, and variance of the subordinator $0.23$ (see more details in \S~\ref{sec:simu}). The sample size is $1000$, and the time interval   $\Delta_n=0.001$. Under these model settings, the jumps  have extremely small sizes. More than $90\%$ of the jump sizes are less than $10^{-7}$, and, hence, they are difficult to be recovered in the MCMC sampling. As mentioned in \S~\ref{sec:intro}, in the previous studies that incorporate infinite jump activity for high-frequency data, models are simplified in order to conduct MCMC sampling.

\subsection{Theoretical challenges}

\cite{Martin} applied their purposely misspecified approach to the simpler model setting of an uncontaminated jump-diffusion model with constant volatility and finite jump activity. They first constructed a misspecified model by  omitting the jump part $J$. Under this misspecified model, the resulting misspecified posterior was shown to be asymptotically normal conditionally on a given path of $J$. Since the result works for all possible $J$, it can be generalized to a version which does not depend on $J$. Even though such an asymptotic normality does hold for a suitably centered and scaled misspecified posterior for the volatility, the misspecification of the model has the adverse effect of causing this misspecified posterior to center in the wrong place and to have an incorrect and inefficient variance compared to the true marginal posterior obtained by marginalizing the full joint posterior over the drift and jump parts of the model.  Therefore, \cite{Martin} proposed to correct for the bias and inefficiency of the misspecified posterior by, respectively, shifting the center by an estimate of the bias, and rescaling the log likelihood using a properly chosen temperature parameter. Since the Bernstein-von Mises theorem involves convergence in total variation norm, and this norm is invariant with respect to location shifts, the resulting corrected posterior for volatility still admits a Bernstein-von Mises theorem but with a correct center and efficient variance equal to the Cram\'{e}r-Rao lower bound. 

In a model with infinite jump activity, we can similarly ignore the jump part and consider a misspecified model, but it is impossible to conclude an unconditional Bernstein-Von Mises theorem from the analogous result for the conditional posterior given a fixed path of the jump process $J$. The  main reason is that for a jump process with infinite activity, the realized quadratic variation $[J]_n = \sum_{i=1}^{n} \Delta_i J^2 = \sum_{i=1}^{n} (J_{t_{i}}- J_{t_{i-1}})^2$ does not converge to the  quadratic  variation $[J]= \sum_{0\le t<T} (J_t - J_{t-})^2$ for almost every path of $J$ (i.e., a.s. convergence does not hold but merely convergence in probability). The almost sure consistency is necessary to prove the properties of the posterior, which is required when proving the local asymptotic normality (LAN) of the likelihood and an optimal convergence rate of the posterior mean. The satisfaction of these two conditions facilitates the establishment of a Bernstein-von Mises theorem under misspecification (see \cite{BVM}). 
In \cite{Martin}'s model, because the jump part $J$ is assumed to have finitely many jumps in a finite time interval, $J$ can be expressed as a finite summation of the discontinuities. Thus, there exists $n_0\in\mathbb{N}$, such that for $n>n_0$, the quadratic variation $[J] $ is exactly equal to $ [J]_n$. However, with infinite jump activity, the convergence of $[J]_n$  to $[J]$ does not hold for almost every path of $J$. This complication leads to the failure of the usual conditions used to prove a Bernstein-von Mises theorem.

On the other hand, for a general semimartingale, it is well-known that $[J]_n$ does converge to $[J]$ in probability \cite{JacodAndProtterDiscret}. Furthermore, for L\'evy processes, a rather good rate of convergence of $O_{p}(n^{-1/2})$ can be obtained (see Lemma \ref{realized1} below). 
We find that this weaker convergence (i.e. in probability rather than almost surely) is enough to demonstrate the desired property of the posterior by applying an unconditional version of the Bernstein-Von Mises theorem and skipping the intermediate results under the conditional probability measure given the jump part. 

Besides the infinite jump activity complicating the nonparametric part of the model, the parametric part is also affected by the presence of the noise $\varepsilon$. Good news is that the variance of the noise, $\sigma_\varepsilon^2$, is an additional nuisance parameter. The adjusted posterior for volatility, and the associated Bernstein-von Mises theorem, must include corrections for deliberately ignoring the presence of microstructure noise. 

\section{A semiparametric version of the misspecified BvM Theorem}\label{sec:bvm}
As explained in the previous section, the  misspecified  Bernstein-von-Mises Theorem of \cite{BVM} plays a crucial role in proving the asymptotic properties of the purposely misspecified posterior.  To accommodate the more complicated settings of our model, the result needs to be generalized to a  semiparametric  version, which is stated as follows.
\begin{thm}  \label{thm:bvm}
	Consider the space $\Omega^{(n)}:=\Omega_{1}^{(n)}\times \Omega_{2}:=\mathbb{R}^{n}\times D([0,\infty))$ (D represents the Skorokhod space of all c\`{a}dl\`{a}g $\mathbb{R}$-valued functions) and a collection of semiparametric models on $\Omega^{(n)}$,
	\[
	\left\{P^{(n)}_{\left( (\theta, \eta), \nu \right)}  : (\theta, \eta) \in \Theta, \nu\in U \right\},
	\]
	where $\Theta$ is an open subset of $\mathbb{R} \times \mathbb{R}^d$ and $U$ is an open subset of an infinite dimensional topological space $\mathbb{F}$.
	Let $P_0^{(n)}:={P}^{(n)}_{\left( (\theta^*, \eta^*), \nu^{*} \right)}$
	and 
	let $Z^{(n)}=(Z_{1},\dots,Z_{n})$ and $\{Y_{t}\}_{t\geq{}0}$ be the canonical processes on  $\Omega^{(n)}$ defined for $\omega=(\omega_{1},\omega_{2})\in \Omega_{1}^{(n)}\times \Omega_{2}$ as $Z_{i}(\omega)=\omega_{1i}$ and $Y_{t}(\omega)=\omega_{2}(t)$, respectively. 
	
	Define  $\Phi = g( \theta, \eta, Y_{\cdot})$ and $\Phi^\dag = g( \theta^*, \eta^*, Y_{\cdot})$, where $Y_{\cdot}$ denotes the sample path of $J$ and $g:\Theta\times D([0,\infty))\to\Theta'\subset \mathbb{R}$ is a known deterministic function. Our data consists of $X^{(n)}:=(X_{1},\dots,X_{n}):= T(Z^{(n)},Y_{\cdot})$, where the function $T:\mathbb{R}^{n}\times D([0,\infty)]\to\mathbb{R}^{n}$ is known.
	
	Suppose there are purposely misspecified models for $X^{(n)}$ denoted as $\tilde{P}_{{\vartheta}}(\cdot):=\tilde{P}(\cdot |{\vartheta})$, $\vartheta\in \Theta'$, which are distributions on $\mathbb{R}^n$ parameterized by $\vartheta$ with densities  
	$ \tilde{p}_{{\vartheta}}$.
	Let $\Pi$ be a prior  distribution with a density $\pi$ that is continuous and positive on $\Theta'$ . 
	Define the misspecified posterior distribution based on $\Pi$ and  $\tilde{P}_{{\vartheta}}(\cdot)$  as
	$$  
	\Pi^n(\varphi \in B|X^{(n)}) =  \frac{\int_B \tilde{p}_\varphi(X^{(n)}) \pi(\varphi) \, d\varphi}{\int \tilde{p}_\xi(X^{(n)}) \pi(\xi)\, d\xi}, \quad B\in \mathcal{B}(\Theta').$$
	
	Assume $\{ \tilde{P}_\vartheta , \vartheta \in \Theta' \}$ satisfy a stochastic local asymptotic normality (LAN) condition relative to a given sequence $\delta_n \rightarrow0$ as norming rate,
	i.e. there exist  some random quantities $\Delta_{n,}$ and $V_{n}$ such that for every compact set $K\in\mathbb{R}$ and $\epsilon>0$,
	\begin{align}\label{eq:con1}
	{P}^{(n)}_{ 0} \left( \sup_{h\in K} \left| \log \frac{\tilde{p}_{\Phi^\dag+\delta_n h} }{\tilde{p}_{\Phi^\dag} } (X^{(n)})-V_{\Phi^\dag} \Delta_{n,\Phi^\dag} h -\frac{1}{2}V_{\Phi^\dag}h^2 \right| > \epsilon \right) \rightarrow  0, \quad \text{as } n  \rightarrow \infty.
	\end{align}
	Also, for any sequence of constants $M_n \rightarrow \infty$, the posterior $\Pi^n$ is assumed to satisfy
	\begin{equation}\label{eq:con2}
	{\Pi}^n \left(  {| \varphi-\Phi^\dag|} >\delta_n M_n | X^{(n)} \right) \stackrel{{P}^{(n)}_{ 0}}{\rightarrow} 0,  \quad n\rightarrow \infty.
	\end{equation}
	Then, $\Pi^n$ converges to a sequence of normal distributions in total variation:
	\begin{equation*}
	{P}^{(n)}_{ 0} \left(  \sup_B\left| {\Pi}^n \left(  (\varphi - \Phi^\dag )/\delta_n \in B | X^{(n)} \right) - N_{\Delta_{n,\Phi^\dag}, V_{\Phi^\dag}^{-1}}(B) \right| >\epsilon\right) \rightarrow 0, \quad n\rightarrow \infty.
	\end{equation*}	
\end{thm}
The proof of the above result follows the original proof in \cite{BVM}. The main modifications are  changing the almost sure convergence to convergence in probability, and adding a nuisance parameter which does not affect the proof.

\begin{remark}
	Condition (\ref{eq:con1}) above is equivalent to
	$$ 	{{P}^{(n)}_{ 0}} \left[ 	{P}^{(n)}_{ 0}\left( \sup_{h\in K} \left| \log \frac{\tilde{p}_{\Phi^\dag+\delta_n h} }{\tilde{p}_{\Phi^\dag} } -V_{\Phi^\dag} \Delta_{n,\Phi^\dag} h -\frac{1}{2} V_{\Phi^\dag}h^2 \right| > \zeta\  \bigg|\ Y_{\cdot}  \right) > \epsilon\right] \rightarrow 0, $$
	for all $\zeta, \epsilon>0$.
	This means that as $n\rightarrow \infty$, the set of those ${Y_{\cdot}} $ which satisfy the condition will cover its probability space with probability 1.
	This is weaker than the misspecified Bernstein-von-Mises Theorem in \cite{BVM} when applying their theorem with ${P}^{(n)}_{ 0} ( \cdot |\ {Y_{\cdot}})$, which implies for almost all paths ${Y_{\cdot}} $,
	$$	{P}^{(n)}_{ 0} \left( \sup_{h\in K} \left| \log \frac{\tilde{p}_{\Phi^\dag+\delta_n h} }{\tilde{p}_{\Phi^\dag} } - V_{\Phi^\dag}\Delta_{n,\Phi^\dag} h -\frac{1}{2} V_{\Phi^\dag}h^2 \right| > \eta\  \bigg|\ {Y_{\cdot}} \right)\rightarrow 0, \text{  for all $\epsilon>0$}. $$
	The second condition (\ref{eq:con2}) and conclusion can be compared with their counterparts in \cite{BVM} in the same way.
\end{remark}

\section{The misspecified model}\label{sec:misspecified}
Our methodology starts with a misspecified model ignoring the drift and the jump component. Namely, $Y_t$ is assumed to follow 
\begin{equation}
Y_{t_j} = X_{t_j} + \varepsilon_{t_j},\quad \text{where}\quad X_{t} = {\theta}^{1/2}W_t. \label{mis}
\end{equation}
This means that we first misinterpret the increments of the underlying process $X$ as independent Gaussian variables, with mean zero, and variance $\theta$. 

Under the misspecified model \eqref{mis}, our target of estimation is still $\theta$, but what it represents changes because of the misspecification. In the absence of jumps, $\theta$ measures the total variation of the underlying process $X$ per unit time and, hence, it can efficiently be estimated by the scaled realized quadratic variation,  
\begin{equation}\label{RQV00}
[X]_{n}:=\frac{1}{T}\sum_{i=1}^{n}(\Delta_{i} X)^{2},
\end{equation}
which coincides with the maximum likelihood estimator (MLE) of the parameter $\theta$ in the underlying misspecified model $X_{t}=\theta^{1/2}W_{t}$. However, under the model $X_{t}=\theta^{1/2}W_{t}+J_{t}$, $\theta$ merely controls the variation of the continuous component and, in the infill limit, the realized quadratic variation (\ref{RQV00}) will aggregate both the true volatility, $\theta^*$, and the scaled variation introduced by the jump process $J$, $T^{-1}[J]$, where $[J]:=\sum_{s\leq{}T}(\Delta Y_{s})^{2}$. 
Throughout, this total variation is denoted as
\begin{equation}\label{Dfnthetadagger0}
\theta^\dag := \theta^* + T^{-1}[J], 
\end{equation}
which takes values on the random parameter domain
$$
\Theta' := \left\{ \theta + T^{-1}[J]; \theta\in\Theta\right\}.
$$
For any sample path of $J$, $\Theta'$ is an open set in  $(0,+\infty)$, and $0\notin \bar{\Theta}'$. Furthermore, there exists some deterministic constant $\delta_0 >0$ such that $\Theta\subset (\delta_0,+\infty)$ and, hence, $\Theta'\subset (\delta_0,+\infty)$.

In \S~\ref{sec:likelihood}, we explicitly write the misspecified likelihood function and the corresponding MLE for $\theta$ under the model \eqref{mis}. Bayesian inference under this misspecified model is proposed in \S~\ref{sec:theorem}. We will show that, given that the data $Y$ is misinterpreted by the model \eqref{mis}, the posterior of $\theta$  can be approximated by a normal distribution.  Further extensions are subsequently considered.

\subsection{Misspecified likelihood function and MLE}\label{sec:likelihood}
Let us  first note that, because of the presence of the noise $\varepsilon$, the increments $\Delta_j Y = Y_{t_j}-Y_{t_{j-1}}$, $j = 1,2,\ldots,n$, are not independent. To deal with the dependency and write an explicit likelihood function, we follow {\cite{gloter_jacod_2001_est, lan}} and transform the observed data $\{\Delta_j Y\}_{j}$ into  independent random variables $\{R_j\}_{j}$ via ${\textbf{R}} = (P_n) (\Delta {\textbf{Y}})$, where ${\textbf{R}} = (R_1, \ldots, R_n)'$, $\Delta{\textbf{Y}} = (\Delta_1Y,\Delta_2Y,\ldots,\Delta_{n-1}Y,\Delta_n Y)'$, and $P_n$ is a symmetric orthogonal matrix with entries
$$ p_{ij}^n := \sqrt{\frac{2}{n+1}}\sin  \frac{ij \pi}{n+1}, \quad i,j = 1,2,\ldots,n.$$
\cite{gloter_jacod_2001_est, lan} showed that, under the misspecified model \eqref{mis}, ${R}_j$ is Gaussian distributed, with mean zero, and variance equal to 
$$\lambda_j^n(\theta) := \theta \Delta_{n} + 2{\sigma_{\varepsilon}}^2 \left(1-\cos \frac{j\pi}{n+1} \right),\quad j=1,2,\ldots,n.$$
For future reference, let us also note that under the true model \eqref{modelWithNoiseb}, the conditional distribution of $R_j$ given $J$ is 	$$R_{j}|J\sim\mathcal{N} \left( \mu \Delta_{n} + \sum_{i=1}^n p_{ij}\Delta_i J, \lambda_j^n(\theta) \right).$$

Based on these Gaussian variables, the likelihood function of the parameters $\theta$ and $\sigma_\varepsilon^2$ given the data $\{\Delta_j Y\}$ can be explicitly written under the misspecified model.
However, note that only $\theta$ is the parameter of interest, while $\sigma_\varepsilon^2$ is merely the nuisance parameter. Instead of writing the likelihood function based on $\lambda_j^n(\theta)$ and maximizing it over a two dimensional space, we replace the nuisance parameter, $\sigma_\varepsilon^2$, with its consistent estimator  $\hat{\sigma}_\varepsilon^2 =\frac{1}{2n} \sum_{j=1}^n \Delta_j Y^2$, and then, obtain a pseudo-likelihood function for $\theta$. The properties of $\hat{\sigma}_\varepsilon^2$ and the rationale of the replacement {are} further demonstrated in Lemmas \ref{lem:Nconv} and \ref{lem:likeliep}. Then, it is natural to {consider the following} misspecified log likelihood function  $\tilde{l}_n$ of $\theta$ given the data $\Delta Y_1, \Delta Y_2,\ldots,\Delta Y_n$:
\begin{equation}\label{likelihood}
\begin{aligned}
\tilde{l}_n(\theta) = -\frac{1}{2} \sum_{j=1}^{n} \left\{\log\left(\theta \Delta_{n} + 2\hat{\sigma}_\varepsilon^2 \left(1-\cos \frac{j\pi}{n+1} \right)\right) \right.\\
\left.+   \frac{R_j^2}{ \theta \Delta_{n} + 2\hat{\sigma}_\varepsilon^2 \left(1-\cos\frac{j\pi}{n+1} \right)} \right\}.
\end{aligned}
\end{equation}
The corresponding MLE $\tilde{\theta}_n$ is the root of the score function 
\begin{equation}\label{likelihoodscore0}
\dot{\tilde{l}}_n(\theta) = -\frac{1}{2n}\sum_{j=1}^{n}\left\{ \frac{1}{\theta \Delta_{n} + 2\hat{\sigma}_\varepsilon^2 \left(1-\cos \frac{j\pi}{n+1} \right)} - \frac{R_j^2}{\left[ \theta \Delta_{n} + 2\hat{\sigma}_\varepsilon^2 \left(1-\cos \frac{j\pi}{n+1} \right)\right]^2} \right\}.
\end{equation}
We further assume that the MLE $\tilde{\theta}_n$ is unique.
\begin{remark}\label{rem:likelinonoise}
	The misspecified likelihood function \eqref{likelihood} can be simplified and directly applied to a model without the microstructure noise  (i.e., $Y=X$ in (\ref{modelWithNoiseb})) by taking $\sigma_\varepsilon^2 = 0$ and $\hat{\sigma}_\varepsilon^2 = 0$. Then, 
	\begin{equation}\label{likelihoodNNC}
	\tilde{l}_n(\theta) = -\frac{1}{2} \sum_{j=1}^{n} \left\{\log \theta \Delta_{n} +   \frac{R_j^2}{\theta \Delta_{n} } \right\}=-\frac{1}{2} \sum_{j=1}^{n} \left\{\log \theta \Delta_{n} +   \frac{\Delta_j Y^2}{\theta \Delta_{n} } \right\}.
	\end{equation}
	In this case, the MLE can be obtained in closed form as
	\begin{equation}\label{MLENNCs}
	\tilde{\theta}_n = \frac{1}{T}\sum_{i=1}^n R_i^2 = \frac{1}{T} \sum_{i=1}^n (\Delta_i Y)^2 = \frac{1}{T} \sum_{i=1}^n (\Delta_i X)^2.
	\end{equation}
	Thus, the misspecified model is consistent with the one in \cite{Martin} and, hence, the model with finite jump activity can be viewed as a particular case of our results.
\end{remark}

\section{Bernstein-von Mises Theorems}\label{sec:theorem}
We assume that the prior distribution $\Pi$ of $\theta$ possesses a continuous and positive density $\pi$ on $(\delta_0,+\infty)$. Denote ${P_*}$ as the distribution of the process $\{  Y_t \}_{t\ge 0}$ under the true model \eqref{modelWithNoiseb}, and ${E_*}$ as the corresponding expectation. Based on the prior $\Pi$ and the likelihood function \eqref{likelihood}, we introduce the Gibbs posterior $\Pi^n$  \cite{Zhang2006gibbs,jiang2008gibbs} with temperature parameters $\kappa_{n}$ as  
\begin{align}\label{MSWTEst}
{\Pi_{n}(A)} = \frac{\int_A \tilde{l}_n({\vartheta})^{1/\kappa_{n}} \pi({\vartheta})\, d{\vartheta}}{\int_\Theta \tilde{l}_n(\zeta)^{1/\kappa_{n}} \pi(\zeta)\, d\zeta},
\end{align}
where $A$ is a Borel set of $\mathbb{R}^+$.
The Gibbs posterior increases the flexibility of the Bayesian procedure, which would allow us to further correct for the misspecification. Specifically, the misspecification causes the posterior for volatility to contract too quickly, making the Bayes estimator (e.g. the posterior mean) superefficient. Rescaling the likelihood flattens out the likelihood and also the posterior, slowing down the contraction of the posterior. Choosing the temperature parameter optimally will make the posterior contract at the efficient rate established by frequentist asymptotic analysis.  We assume that $\kappa_{n}$ converges in probability to a random variable $\kappa^\dagger$ as $n\rightarrow \infty$ under the true measure ${P}_*$.  Note that $\kappa_n$ may be data-dependent, and therefore it is possible that the random variable $\kappa^\dag$ also depends on the data under  ${P}_*$. 

Our main result states that, as the sample size $n$ increases,  the misspecified posterior based on $\Pi$ and the misinterpreted data $\{\Delta Y_i\}$ will be approximately normal and centered at the maximum likelihood estimator $\tilde{\theta}_n$ obtained from the misspecified likelihood \eqref{likelihood} under the true measure $P_*$. The asymptotic variance is equal to the temperature parameter $\kappa^\dag$ times the inverse of the Fisher information of the misspecified likelihood. We give two versions. The first result covers situations where the microstructure noise $\varepsilon$ can be ignored. This is the case when, for instance, we use medium range frequencies such as 5-minute or daily observations. We achieve the standard $n^{-1/2}$ rate of convergence. The second result covers the more realistic case where the microstructure noise is explicitly incorporated in the model. This is needed when working with ultra high frequencies. This comes at the cost of a slower $n^{-1/4}$ rate of convergence.

\begin{thm}\label{firstthm}
	Suppose that the data $Y_{t_{0}},\dots,Y_{t_{n}}$ is generated according to (\ref{modelWithNoiseb})-(\ref{model-XY}) with $\varepsilon_{t}\equiv 0$ and Assumption \ref{ass:bddvar}-1 is satisfied. Then, the misspecified posterior defined in (\ref{MSWTEst}) with $\tilde{l}_{n}$ given as in (\ref{likelihoodNNC}) and $\kappa_{n}\stackrel{P_{*}}{\to} \kappa^{\dagger}$, for some positive r.v. $\kappa^{\dagger}$, can be approximated by a normal distribution in the sense that
	$$ TV\left(\Pi_n, \ \mathcal{N}( \tilde{\theta}_n, 2\kappa^\dag \theta^{\dagger 2}n^{-1} ) \right)  \stackrel{{P}_*}{\rightarrow} 0,  \quad \text{ as }n\rightarrow \infty,$$
	where $TV$ represents the total variation distance, $\tilde{\theta}_{n}$ is the MLE (\ref{MLENNCs}), and $\theta^{\dagger}$ is defined in (\ref{Dfnthetadagger0}).
\end{thm}

\begin{thm}\label{noisemisthm}
	Under the framework and assumptions \ref{ass:noise}, \ref{ass:bddvar}, and \ref{ass:finite16} above, the misspecified posterior $\Pi^n$ defined in \eqref{MSWTEst} with $\tilde{l}_{n}$ given as in (\ref{likelihood}) and $\kappa_{n}\stackrel{P_{*}}{\to} \kappa^{\dagger}$, for some positive r.v. $\kappa^{\dagger}$, is such that
	$$ TV\left(\Pi_n, \ \mathcal{N}(\tilde{\theta}_n, 8\kappa^\dag \theta^{\dagger 3/2} {\sigma_{\varepsilon}} n^{-1/2} ) \right)  \stackrel{{P_*}}{\rightarrow} 0,  \quad \text{ as }n\rightarrow \infty,$$
	where $\tilde{\theta}_{n}$ is the corresponding MLE (i.e., the root of the score function (\ref{likelihoodscore0})) and $\theta^{\dagger}$ is defined in (\ref{Dfnthetadagger0}).
\end{thm}

The proofs of the two theorems utilize Theorem \ref{thm:bvm}, and are contained in the Appendix.
\begin{remark}\label{MGSMJmp}
	It is worth nothing that Theorem \ref{firstthm} holds without any restriction on the Blumenthal-Getoor index $\alpha$. In fact, this result holds for a large class of pure-jump semimartingales $J$ and even quite general stochastic volatility models (see Section \ref{Sect:Extensions}). The restriction of $\alpha<1$ is needed when correcting the posterior as shown below.
\end{remark}

\section{Correcting for misspecification}\label{sec:correct}
The main conclusions of Theorems \ref{firstthm} and \ref{noisemisthm}, namely, as $n\rightarrow \infty$,
\begin{equation*}
	\begin{aligned}
	& TV\left(\Pi_n, \ \mathcal{N}(\tilde{\theta}_n, V_{asy}n^{2\beta} ) \right)  \stackrel{{P_*}}{\rightarrow} 0,  \\
	& V_{asy}n^{2\beta} = 8\kappa^\dag \theta^{\dagger 3/2} {\sigma_{\varepsilon}} n^{-1/2}{1}_{\sigma_{\varepsilon}\neq 0} + 2\kappa^\dag \theta^{\dagger 2}  n^{-1}{1}_{\sigma_{\varepsilon}= 0}, 
	\end{aligned}
\end{equation*}
state that the misspecified posterior $\Pi_n$ is approximately normally distributed, centered at $\tilde{\theta}_n$, which is a biased estimator of $\theta^{*}$ in the presence of jumps. Furthermore, the asymptotic variance may not be the most efficient either since we ignored the drift and the jump components on purpose. To adjust the bias and variance, what we need is a consistent estimator for the true parameter $\theta^*$, which admits a feasible central limit theorem. 
In what follows, we will first propose a general correction procedure and the corresponding Bernstein-von Mises theorem for any estimator with these two  properties. Concrete instances of these estimators for both the no-noise and the general cases are presented thereafter.

Suppose we have  an estimator $\hat{\theta}_n$ of $\theta^*$ such that
\begin{align}\label{est_con_eff}
{\hat{\theta}_n \stackrel{P_*}{\rightarrow} \theta^*,\qquad  n^{-\beta}(\hat{\theta}_n -\theta^* )\stackrel{\mathcal{L}}{\rightarrow} \mathcal{N}(0,V),\quad \text{ as } n\to\infty,}
\end{align}
where, in accordance with Theorems \ref{firstthm} and \ref{noisemisthm}, the rate of convergence $\beta$ is $-\frac{1}{4}$ when $\sigma_\varepsilon \neq 0$, and $-\frac{1}{2}$ when $\sigma_\varepsilon = 0$.

Our goal is to adjust the posterior so that it centers at $\hat{\theta}_{n}$ and matches the asymptotic variance of $\hat{\theta}_{n}$. For the center, we simply shift the posterior by the right amount, while for the asymptotic variance, we adjust the temperature parameter. Concretely, define the estimator
\begin{align}\label{Jestimator}
\widehat {[J]}_n := T( \tilde{\theta}_n - \hat{\theta}_n).
\end{align}
The notation $\widehat {[J]}_n$ comes from the fact that this is a consistent estimator for the quadratic variation of the jump component $J$, because, as shown in the Appendix (see \eqref{DRA0} and \eqref{CnstTildeTheta2}), $\tilde{\theta}_n$ converges to $\theta^\dag = \theta^* + T^{-1}[J]$ and $\hat{\theta}_{n}$ is a consistent estimator of $\theta^{*}$ by construction.
We will then adjust the location of the posterior by subtracting  $T^{-1}\widehat {[J]}_n$ (this operation will necessarily center the posterior at $\tilde{\theta}_{n}-T^{-1}\widehat {[J]}_n=\hat{\theta}_{n}$). 
To adjust the variance, we  adopt a sequence of the temperature parameters and its limit of the form:
\begin{align}\label{kappa}
\kappa_n = \frac{\hat{V}_n}{\hat{V}_{asy,n}}  \quad \text{and} \quad \kappa^\dag = \frac{V}{V_{asy}},
\end{align}
where the quantities $\hat{V}_n$ and $\hat{V}_{asy,n}$ are suitable consistent estimators of ${V}$ and ${V_{asy}}$, respectively. The choice of these estimators will be specified below in \S~\ref{sec:correctionnonoise}-\S~\ref{sec:correctiongeneral}. 

Finally, we can define the adjusted misspecified posterior $\widetilde{\Pi}_n $ as one having the density function
\begin{equation}\label{AdjQuasiPost}
\widetilde{\pi}_n(\vartheta) = \pi_n\left(\vartheta + T^{-1} \widehat{[J]}_n\right),
\end{equation}
where $\pi_n$ is the misspecified posterior obtained in Theorems \ref{firstthm} and \ref{noisemisthm}  with $\kappa_{n}$ and $\kappa^\dagger$ defined in \eqref{kappa}. Asymptotic normality of the adjusted posterior is established by the following result.

\begin{thm}\label{secondthm}
	With the same conditions as in Theorem \ref{noisemisthm} or Theorem \ref{firstthm} except for the temperature parameter $\kappa_n$ defined as in \eqref{kappa}, the adjusted posterior $\widetilde{\Pi}_n$ defined above can be approximated by a normal distribution in the sense that,
	\begin{equation} \label{BvMNN2}
	TV\left(\widetilde{\Pi}_n, \ \mathcal{N}\left(\hat{\theta}_n, V n^{2 \beta} \right) \right) \stackrel{P_*}{\rightarrow} 0\quad \text{ as }n\rightarrow \infty.
	\end{equation}
\end{thm}
A location shift in Theorem \ref{firstthm} or Theorem \ref{noisemisthm} with $\kappa_n$ defined in \eqref{kappa} gives us the proof of Theorem \ref{secondthm}.

This theorem illustrates that any type of $1-\alpha$ credible interval ($CI_{B,\alpha}$) of $\widetilde{\Pi}_n$ is asymptotically the same as a $1-\alpha$ confidence interval for $\theta^{*}$ based on $\mathcal{N}(\hat{\theta}_n, Vn^{2\beta} )$. The upper and lower bounds of the $CI_{B,\alpha}$ can then be approximated by $\hat{\theta}_n \pm \sqrt{Vn^{2\beta}} z_{\alpha/2}$ as $n\rightarrow\infty$, where $z_{\alpha/2}$ is the $\alpha/2$ quantile of the standard normal distribution. Because $\hat{\theta}_n$ satisfies a central limit theorem with asymptotic variance  $V$, we have that
$$ P_* ( \theta\in CI_{B,\alpha}) \approx P_* ( \theta\in \hat{\theta}_n \pm \sqrt{Vn^{2\beta}} z_{\alpha/2}) = P_* ( \hat{\theta}_n \in \theta\pm \sqrt{Vn^{2\beta}} z_{\alpha/2}) \approx  1-\alpha.$$
Therefore, the $1-\alpha$ credible interval has approximately the correct repeated sampling coverage under $P_*$, which indicates frequentist validation of the Bayesian inference based on the adjusted posterior.

\subsection{Correction for a model without microstructure noise}\label{sec:correctionnonoise} 
When the variance $\sigma^{2}_{\varepsilon}$ of the noise is $0$, we can use the thresholded realized quadratic variation of \cite{mancini2009}:
\begin{align}\label{TrctRQV22}
\hat{\theta}_n =  \frac{1}{T}\sum_{i=1}^{n} \Delta_i Y^2 {1}_{|\Delta_i Y|\leq{}\eta_n},
\end{align}     
where $\eta_n$ is a threshold proportional to $n^{-w}$ for some suitable exponent $w$. Consistency of $\hat{\theta}_{n}$ is established in \cite{mancini2009} for any $w\in(0,1/2)$ when $J$ consists of the superposition of a general finite-jump activity process and an independent L\'evy process. \cite{cont2011} showed that $\hat{\theta}_n$ satisfies a central limit theorem with asymptotic variance  $2 \theta^{* 2}n^{-1}$ under Assumption \ref{ass:bddvar} provided that $w\in\left(\frac{1}{4-2\alpha}, \frac{1}{2}\right)$. The existence of $w$ is guaranteed because $\alpha<1$ and, hence, $\frac{1}{4-2\alpha}< \frac{1}{2}$. 

With the estimator $\hat{\theta}_{n}$ described above, we apply Theorem \ref{secondthm} with $\beta = -1/2$, $V = 2 \theta^{* 2}$, and the temperature parameters taken as
\begin{align}\label{kappaIA}
\kappa_{n} =\left(\frac{\hat{\theta}_{n}}{\tilde{\theta}_{n}}\right)^2 \quad \text{and} \quad \kappa^\dagger =  \left(\frac{\theta^*}{\theta^\dag}\right)^2.
\end{align}
By Slutsky's Theorem,  it is clear that $\kappa_{n}  \rightarrow \kappa^\dag$ in $P_*$-probability. We then obtain the following.

\begin{cor}\label{cor:IA}
	Using the same conditions as in Theorem \ref{firstthm} except for the temperature parameter $\kappa_n$ defined as in (\ref{kappaIA}), and assume \ref{ass:bddvar}, the adjusted posterior $\widetilde{\Pi}^n$ with density (\ref{AdjQuasiPost}) can be approximated by a normal distribution in the sense that,
	$$ TV\left(\widetilde{\Pi}_n, \ \mathcal{N}(\hat{\theta}_n , 2 \theta^{* 2}n^{-1} ) \right) \stackrel{P_*}{\rightarrow} 0\quad \text{ as }n\rightarrow \infty.$$
\end{cor}
\begin{remark}\label{MGSMJmpb}	
	As we will show in  Section \ref{Sect:Extensions} below, the result above also holds for stochastic volatility models and more general pure-jump processes $J$.
\end{remark}

\subsection{Correction for the general model}\label{sec:correctiongeneral}
When the variance of the noise is positive, one possible solution is to adopt  the estimator $\hat{\Sigma}_n$ proposed in \cite{Jing2014jumpest}, {which} combines the thresholding approach of \cite{mancini2009} with the pre-averaging method of \cite{JACOD2009preaveraging} (see also \cite{JacodAndProtterDiscret} for a detailed exposition of the theory). The pre-averaging method is used to {mitigate} the effect of the noise $\varepsilon$. Utilizing this method, we formulate several overlapping blocks of increments, and calculate proxies of the increments of the uncontaminated process $X$ by taking the weighted average of the increments of $Y$ within {each} block. Then, the estimator is defined as the sum of the squares of those new quasi-increments that are less than some threshold, and is further debiased using {an} estimator of the variance of the noise. This estimator meets our requirements, when we include both infinitely many jumps with bounded variation and normally distributed microstructure noise. For completeness, we describe the key aspects of this estimator below. 

The estimator depends on two parameters: the length of the block $k_n$ and the weight function $g$. The latter satisfies the following regularity conditions:
\begin{itemize}
	\item $g$ is continuous on $[0,1]$, piecewise $C^1$ with a piecewise Lipschitz derivative $g'$, and
	\item $g(0) = g(1) = 0$, and $\bar{g} = \int_0^1 g^2(s)\,ds<\infty$.
\end{itemize}
One simple and common choice is $g(s) = s\wedge (1-s)$. Next, for some constant $c$, let $k_n = \lfloor cn^{1/2}\rfloor$ (the notation $\lfloor a\rfloor$ defines the largest interger that is smaller than $a$), $c_1 = c\bar{g}$, $c_2 = \int_0^1 (g'(s))^2\,ds/c$, and also define
\begin{align*}
\hat{\Sigma}_n & =c_1^{-1}\left(n^{-1/2} U(Y,g)_n -c_2 \hat{\sigma}_\varepsilon^2 \right),\\
U(Y,g)_n &= \sum_{i=1}^{n-k_n} \left( \Delta_{i,k_n}^n Y(g) \right)^2 {1}\{|\Delta_{i,k_n}^n Y(g) |\le u_n\},\\
\Delta_{i,k_n}^n Y(g) & = \sum_{j=1}^{k_n-1}g(j/k_n)\Delta_{i+j} Y,
\end{align*}
where we recall that $\hat{\sigma}_\varepsilon^2 =\frac{1}{2n} \sum_{j=1}^n \Delta_j Y^2$ and {the} threshold $u_n$ satisfies 
$$ u_nn^{w_1}\rightarrow 0, \; u_nn^{w_2}\rightarrow \infty, \; \text{ as } n\rightarrow\infty,$$
for some $ 0\le w_1< w_2<{1}/{4} \text{ and } w_1 >1/(8-4\beta).$
The estimator is consistent and admits a central limit theorem. More specifically, by Theorems 1 and 3 in \cite{Jing2014jumpest}, \eqref{est_con_eff} holds with $\hat{\theta}_n = \hat{\Sigma}_n$, $\beta = -1/4$, and 
\[
V = V_{noise}:= \frac{c}{\bar{g}^2}\left[ 4\theta^2 \Phi_{22} + \frac{2\theta\sigma_\varepsilon^2}{c^2}\Phi_{12} + \frac{\sigma_\varepsilon^4}{c^4}\Phi_{11} \right],
\]
where $\Phi_{ij} = \int_0^1 \phi_i(x)\phi_j(x)\,dx$, $\phi_1 = \int_x^1 g'(y)g'(y-x)\,dy$, and $\phi_2(x) = \int_x^1 g(y)g(y-x)\,dy$. 

The temperature parameters in \eqref{kappa} can be defined as 
\begin{align}\label{kappanoise}
\kappa_n =  \frac{ \frac{c}{\bar{g}^2}(4\Phi_{22}\hat{\Sigma}_n+\frac{2\Phi_{12}}{c^2} \hat{\Sigma}_n^{1/2}\hat{\sigma}_\varepsilon^2+  \frac{\Phi_{11}}{c^4} \hat{\sigma}_\varepsilon^4) }{8\tilde{\theta}_n^{3/2} \hat{\sigma}_\varepsilon}  \quad \text{and} \quad \kappa^\dagger = \frac{V_{noise}}{8\theta^{\dag 3/2} \sigma_\varepsilon} .
\end{align}

The convergence of $\kappa_n$ to $\kappa^\dag$ can be established through the consistency of $\hat{\Sigma}_n$ and $\hat{\sigma}_\varepsilon^2$ for $\theta^*$ and $\sigma_\varepsilon^2$, respectively, as well as the property that when $X_n = O_{P_*}(1)$ and $Y_n \stackrel{P_*}{\rightarrow} 0$, then $X_n Y_n \stackrel{P_*}{\rightarrow} 0$.

Then, we have the following corollary of Theorem \ref{secondthm}.
\begin{cor}\label{cor:noise}
	With the same conditions as in Theorem \ref{noisemisthm} and with the temperature parameter $\kappa_n$ defined as in \eqref{kappanoise}, the adjusted posterior $\widetilde{\Pi}_n$ defined above can be approximated by a normal distribution in the sense that,
	$$ TV\left(\widetilde{\Pi}_n, \ \mathcal{N}\left(\hat{\Sigma}_n, \frac{c}{\bar{g}^2}\left[ 4\theta^2 \Phi_{22} + \frac{2\theta\sigma_\varepsilon^2}{c^2}\Phi_{12} + \frac{\sigma_\varepsilon^4}{c^4}\Phi_{11} \right] n^{-1/2} \right) \right) \stackrel{P_*}{\rightarrow} 0,\  \text{ as }n\rightarrow \infty.$$
\end{cor}

\section{Extension to more general semimartingales  without noise}\label{Sect:Extensions}
Thus far, we have assumed constant parameters for both the drift and diffusion components and a L\'evy process for the jump component $J$. In this section, we show that, in fact, when the microstructure noise can be ignored, the purposely misspecified posterior approach can also be applied to stochastic volatility models and more general jump processes $J$. As mentioned before, it is generally believe that the microstructure noise is relatively negligible when using medium range frequencies such as 5-minute or daily observations.

We consider the model
\begin{align}\label{GenSmrtMld}
dX_t = \beta_t dt + \sigma_t dW_t + dJ_t, \text{ for } t\in [0,T],
\end{align}
where $W$ is a Wiener process, $J$ is a suitable pure-jump semimartingale, and $\beta = \{\beta_t\}_{t\ge 0}$ and 
$\sigma = \{\sigma_t\}_{t\ge 0}$ are c\`{a}dl\`{a}g adapted processes.
The parameter of interest is the scaled integrated variance 
\begin{equation}\label{trgtEstVol1}
\theta^* = \frac{1}{T}\int_0^T \sigma_t^2\,dt.
\end{equation}

We again use the misspecified model \eqref{mis} for $X$ with $\varepsilon=0$. The corresponding log likelihood function would then be the same as in Remark \ref{rem:likelinonoise} with associated MLE 
\begin{equation}\label{RlzQV22}
\tilde{\theta}_n = \frac{1}{T}\sum_{i=1}^n \Delta_i X^2.
\end{equation}
An analysis of the proof of Theorem \ref{firstthm} reveals that the key for the result therein is the CLT stated in Lemma \ref{realized1}. Specifically, what is needed is that the misspecified MLE (\ref{RlzQV22}) converges to (\ref{trgtEstVol1}) at the rate $O_{p}(n^{-1/2})$ (see Eq.~(\ref{thetathm1}) in the proof). \cite{Jacod2008} (see Theorem 2.12 and Remark 2.13 therein) shows an analogous CLT to that of Lemma  \ref{realized1} (with the same rate of convergence) under the more general setting (\ref{GenSmrtMld}) when $\sigma$ and $J$ are of the form:
\begin{align*}
\sigma_{t} &= \sigma_{0}+\int_{0}^{t}\tilde{b}_{s}ds+
\int_{0}^{t}\tilde{\sigma}_{s}dW_{s}+
\int_{0}^{t}\tilde{\sigma}'_{s}dW'_{s}\\
& \quad +\int_{0}^{t}\int \tilde{\delta}(s,x){\bf 1}_{\{|\tilde{\delta}(s,x)|\leq{}1\}}(\mu(ds,dx)-\nu(ds,dx))\\
& \quad +\int_{0}^{t}\int \tilde{\delta}(s,x){\bf 1}_{\{|\tilde{\delta}(s,x)|>{}1\}}\mu(ds,dx)\\
J_{t}&=\int_{0}^{t}\int \delta(s,x){\bf 1}_{\{|\delta(s,x)|\leq{}1\}}(\mu(ds,dx)-\nu(ds,dx))\\
& \quad +\int_{0}^{t}\int \delta(s,x){\bf 1}_{\{|\delta(s,x)|>{}1\}}\mu(ds,dx),
\end{align*}
where $W'$ is a Wiener process independent of $W$ and $\mu$ is a Poisson random measure on $\mathbb{R}_{+}\times \mathbb{R}$ with predictable compensator $\nu(ds,dx)=dsdx$, independent of $(W,W')$. 
The coefficients of $\sigma$ and $J$ (including $\delta:\Omega\times \mathbb{R}_{+}\times \mathbb{R}\to\mathbb{R}\backslash\{0\}$ and $\tilde{\delta}:\Omega\times \mathbb{R}_{+}\times \mathbb{R}\to\mathbb{R}\backslash\{0\}$) are random processes satisfying standard conditions for the integrals therein to be well defined.

As explained in Section \ref{sec:correct}, the step to correct the center and variance of the misspecified posterior $\Pi_{n}$ requires an estimator $\hat{\theta}_{n}$ of $\theta^{*}$ enjoying a CLT with a rate of $n^{-1/2}$. As it turns out, the thresholded realized quadratic variation of \cite{mancini2009}, defined in (\ref{TrctRQV22}), does again the job at least in the case of bounded variation jump process $J$. Specifically, \cite{Jacod2008} (see Theorems 2.4 and 2.11 therein) obtains a feasible CLT for (\ref{TrctRQV22}) under the same framework as above, but with an additional condition on $J$ that amounts to $J$ having bounded variation paths.

When the microstructure noise is taken into account, the extension is not as direct as for the no noise case, because after applying an   orthonormal transformation to remove the autocovariance introduced by the noise, similar to that at the beginning of Section \ref{sec:likelihood}, the distribution of the transformed data does not depend anymore only on the target parameter $\theta^* = T^{-1}\int_0^T \sigma_t^2\,dt$. Instead, the variance of each transformed data depends on a weighted sum of the `volatility' of each increments.
Then, analyzing the transformed data using the same procedure as before can only provide us an estimator of some value larger than the integrated volatility, but not about the exact parameter  $\theta^{*}$.

\section{Simulation}\label{sec:simu}
This section discusses the finite sample performance of the adjusted posterior defined in Theorem \ref{secondthm}.
We aim to show the plausibility of the limit (\ref{BvMNN2}) at large sample size. This is demonstrated through comparing the empirical coverage probability of the credible interval derived from the adjusted posterior and the confidence interval from its corresponding asymptotic normal distribution in the theorem.
We also aim to compare the ``purposely misspecified" method with the frequentist central limit theorem (CLT) \eqref{est_con_eff}.

\subsection{Infinite jump activity without noise}
In order to incorporate infinite jump activity, the jump component is set be a variance gamma process
\begin{equation}
J_t = aG_t +bB_{G_t},
\end{equation}
where $a = -0.2$, $b= 0.2$, $\{G_t\}_{t\geq{}0} $ is Gamma process such that $G_h\Gamma(\Delta_n/c,c)$, with $c=0.23$, and $\{B_{t}\}_{t\geq{}0}$ is an Wiener process independent of the Wiener process $W$. For the drift and diffusion components, let $ \mu = 0.1$ and $ \theta = 0.3$. All the parameters are taken from \cite{mancini2009}. 
For simplicity, 
we adopt the widely-used threshold  $\eta_n = n^{-w}$, where $w \in (0,0.5)$ and $n$ is the sample size. This is a possible and conventional choice in terms of consistency and efficiency. In the following simulation, we use $w = 0.39$.

For the prior of $\theta$, an inverse gamma distribution is applied with shape and scale both equal to one. Since the temperature parameters do not affect the conjugacy, the misspecified posterior and the adjusted posterior both follow inverse gamma distribution.

\subsubsection*{Single sample path}
First of all, 5000 equally spaced observations are simulated based on the parameters defined above (sample size $n=5000$). The adjusted posterior $\widetilde{\Pi}_n$ is generated using Corollary \ref{cor:IA}. The results are shown in Figure \ref{one}. The adjusted posterior for one possible sample path is plotted as the dashed line and compared with the corresponding asymptotic normal distribution $\mathcal{N}(\hat{\theta}_n, 2\theta^{*2}n^{-1})$ (the solid line). These two lines can hardly be distinguished from each other. Moreover, they are both roughly centered at the true volatility 0.3. This true volatility also lies between the dashed vertical lines which mark out the 95\% highest posterior density (HPD) interval of the adjusted posterior. 
\begin{SCfigure}[][ht]
	\caption{\textbf{Comparison of adjusted posterior and asymptotic normal distribution for one sample path.} \footnotesize The solid line represents the asymptotic normal distribution in Corollary \ref{cor:IA}. The dashed line is the adjusted posterior. 95\% HPD interval lies between the two black dashed lines. }
	\includegraphics[scale=0.3]{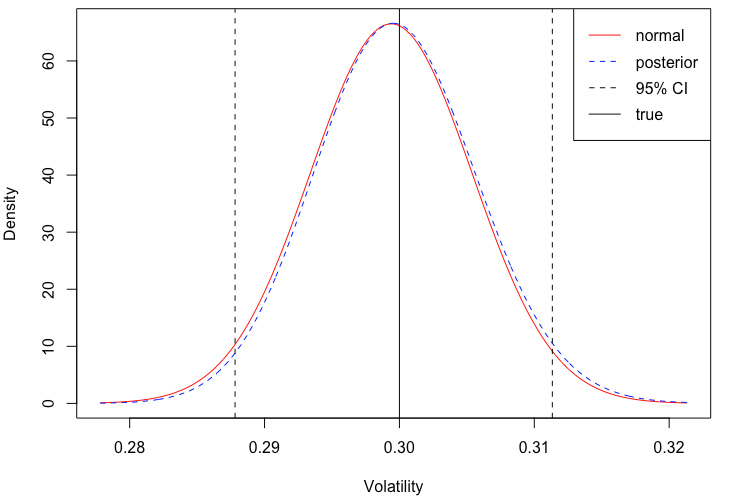} \label{one}
\end{SCfigure}

It turns out that the adjusted posterior recovers the asymptotic normal distribution, which proves the validity of Corollary \ref{cor:IA}.  This suggests that the adjusted posterior will be centered at an efficient estimator with optimal variance when the sample size is large enough.
\subsubsection*{Point estimators}
In the second step, we evaluate the consistency of the point estimators. The biases of the means of two distributions defined in Corollary \ref{cor:IA} are compared: the mean of the adjusted posterior $\widetilde{\Pi}_n$, and the mean of the asymptotic normal distribution $\hat{\theta}_n=\tilde{\theta}_n - T^{-1}\widehat{[J]}_n$, which is also the threshold estimator in \cite{mancini2009}. We also consider the misspecified posterior adjusted by the latent realized quadratic variation of the jump component $T^{-1}{[J]}_n$ instead of $ T^{-1}\widehat{[J]}_n$. The corresponding  asymptotic normal  distribution has mean $\hat{\theta}_n^*=\tilde{\theta}_n - T^{-1}{[J]}_n$.  The analysis of these four point estimators is based on 1000 simulations. For each simulation, 5000 equally spaced observations are generated and used to calculate the biases.

The distribution of the biases is plotted in Figure \ref{bias}. The solid line is formed by the biases of the threshold estimator, while the dashed line is formed by the biases of the mean of the adjusted posterior $\widetilde{\Pi}_n$. The dotted line represents the bias of $\hat{\theta}_n^*$. The bias of the mean of the adjusted posterior using the realized quadratic variation is represented by the dashed-dotted line.
\begin{SCfigure}[][ht]
	\caption{\textbf{Bias of point estimators.} \footnotesize  The biases of the mean of the asymptotic normal distribution in Corollary \ref{cor:IA} forms the solid lines. While the dashed line is the distribution of the mean of the adjusted-posterior. The dotted and the dashed-dotted lines represent the distributions of the means of asymptotic normal distribution and posterior in Theorem \ref{firstthm} with location shift equals the realized quadratic variation of the jumps. }
	\includegraphics[scale = 0.4]{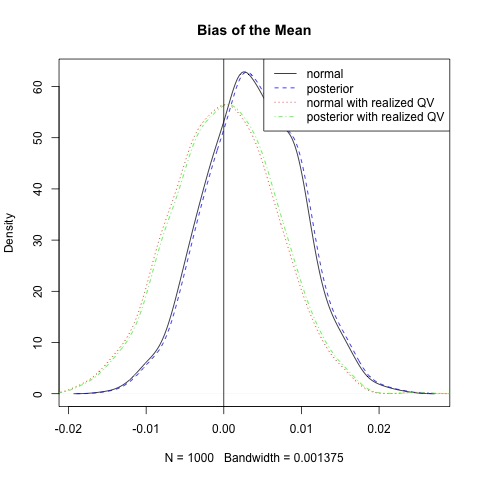} \label{bias}
\end{SCfigure}

The similarity of the solid and the dashed lines as well as the similarity of the dotted and the dashed-dotted lines suggest that the posterior mean and the mean of the asymptotic normal distribution have similar behavior in terms of their difference with the true volatility. The biases are relatively small since the volatility is 0.3 while most of the biases are within 0.01 range.

\begin{remark}
	We may increase the accuracy of the adjusted posterior $\widetilde{\Pi}_n$ by using a better estimator of the quadratic variation of the jump $[J]$ to correct the misspecified posterior $\Pi_n$ defined in Theorem \ref{firstthm}.  While the dashed and the solid lines have higher probability for the positive values, the dotted and the dashed-dotted lines are more symmetric. This suggests that the right-skewed tendency of our posterior mean might be because of the poor estimates for the jump component. The better the estimation of the $[J]$ we applied, the closer distribution to the symmetric dashed-dotted line we will get. One approach is to optimize the threshold $\eta_{n}$ in the threshold parameter $\hat{\theta}_n$.
\end{remark}

\subsubsection*{Comparison based on confidence interval}
In order to evaluate the accuracy of the inference, we compare the empirical coverage probability of the credible interval of the posterior with the confidence intervals of the asymptotic distribution and of the CLT based on the threshold estimator. For simplicity, in this section, we use ``CI" to represent both the credible interval and the confidence interval. We increase the sample size $n$ to 105000, which is approximately the number of the stock data obtained within one year with 5-minute interval. The coverage probabilities of the 95\% CIs based on 1000 repeats are listed in Table \ref{tab:ProbCI}. For each repeat, we simulate a sample path with 105000 observations.
\begin{table}[ht]
	\begin{center}
		\begin{tabular}{r|l} 
			Probability & CI's are built from \\ \hline\hline
			0.943 & Asymptotic normal distribution $\mathcal{N}(\hat{\theta}_n, 2\theta^{*2}n^{-1})$\\ \hline
			{\color{red} {0.944}}& HPD interval based on posterior $\tilde{\pi}_n(\theta) = \pi_n(\theta+ \widehat{[J]}_n )$\\\hline
			0.940& Equal-tail credible interval based on posterior $\tilde{\pi}_n(\theta) = \pi_n(\theta+\widehat{[J]}_n )$\\\hline
			0.940 & Normal distribution using \cite{mancini2009}'s estimator and variance
		\end{tabular}
		\caption{Empirical coverage probability of the confidence interval.}
		\label{tab:ProbCI}
	\end{center}
\end{table}

Therefore, it can be concluded that the HPD interval has the highest coverage probability among all the CIs derived from various distributions defined above. However, all the coverage probabilities are slightly less than $0.95$. We may fill this gap by adopting a more refined frequentist estimator $\hat{\theta}_{n}$ to serve as the center of the posterior or by utilizing a more accurate misspecified model.

Except the conjugate prior, an non-informative prior, the uniform distribution, and an exponential distribution are also for the simulation based on the same model. The results are compatible with the inverse-gamma prior. 

\subsection{L\'evy Model with microstructure noise}
In order to illustrate the results in presence of both infinitely many jumps and microstructure noise, we conduct simulation for the following model from \cite{Jing2014jumpest}:
$$
X_t = W_t + J_t, \quad Y_{i/n} = X_{i/n} + \varepsilon_{i/n}, \quad \varepsilon_{i/n} \sim \mathcal{N}(0, 0.01^2),
$$
for $i = 1,2,\ldots, n$. The jump part $J_t$ is a trimmed symmetric $\beta$-stable process with $\beta = 0.5$. The trimmed process means that after we simulated the increments of all the jumps, the largest $2\%$ of them (ranked by absolute values) will be discard to match the behaviour for high-frequency tick-by-tick data. To allow a comparison, simulation is conducted based on exactly the same parameters described in the paper except one constant $c$, which determines the length of the preaveraging blocks by $ k_n = \lfloor c\Delta_n^{-1/2} \rfloor.$ The choice of $c$ is not clearly stated in the paper. Thus, we choose the same $c=1/3$ as in the original work \cite{JACOD2009preaveraging}. The sample size is set as $n=15600$ and $\Delta_n = 1/7800$ taken from \cite{Jing2014jumpest}.

For the adjusted posterior, the data is divided into two parts. The first half is used to evaluate the estimator $\hat{\Sigma}_n$ in \cite{Jing2014jumpest}, which is used in the prior, and the rest is used to make inference. The prior is chosen to be a truncated normal distribution with lower boundary $0$, centered at $\hat{\Sigma}_n$ obtained using the first half of the data, and standard deviation $0.06$. We generate 20000 MCMC samples, in which the first 5000 are burned. We generate 1000 times posteriors, each based on 20000 MCMC samples. The results can be summarized as follows.

\begin{table}[ht]
	\centering
	\begin{tabular}{l|ll}
		& Bias & s.e.  \\\hline
		$\hat{\Sigma}_n$ & 0.0131 & 0.0440  \\\hline
		MAP          & 0.0110 & 0.0631 
	\end{tabular}
	\caption{Properties of the point estimators}
\end{table}

For the comparison of the point estimators, we analyse the mean bias and the standard errors of the frequentist estimator $\hat{\Sigma}_n$ and the maximum a posterior point estimator (MAP) defined as
$$ \hat{\theta}_{MAP} = \arg\max_{\theta} \widetilde{\Pi}_n(\theta).$$
Both the mean bias and the standard errors are similar and small, suggesting the estimation accuracy of the two approaches. For inference strength, the coverage probability of the $95\%$ credible interval of the adjusted posterior is slightly better than the confidence interval derived from the CLT of the estimator $\hat{\Sigma}_n$. 

\begin{table}[ht]
	\centering
	\begin{tabular}{l|l}
		Coverage Prob. & The distribution where the CI concludes from \\\hline \hline
		0.945  & CLT of the estimator $\hat{\Sigma}_n$  \\\hline
		0.953  & HPD interval based on the adjusted posterior \\\hline
		0.952  & Asymptotic normal distribution 
	\end{tabular}
	\caption{Coverage probabilities of CIs for the model with noise}
\end{table}

\begin{figure}[ht]
	\begin{subfigure}{0.45\textwidth}
		\includegraphics[width=\linewidth]{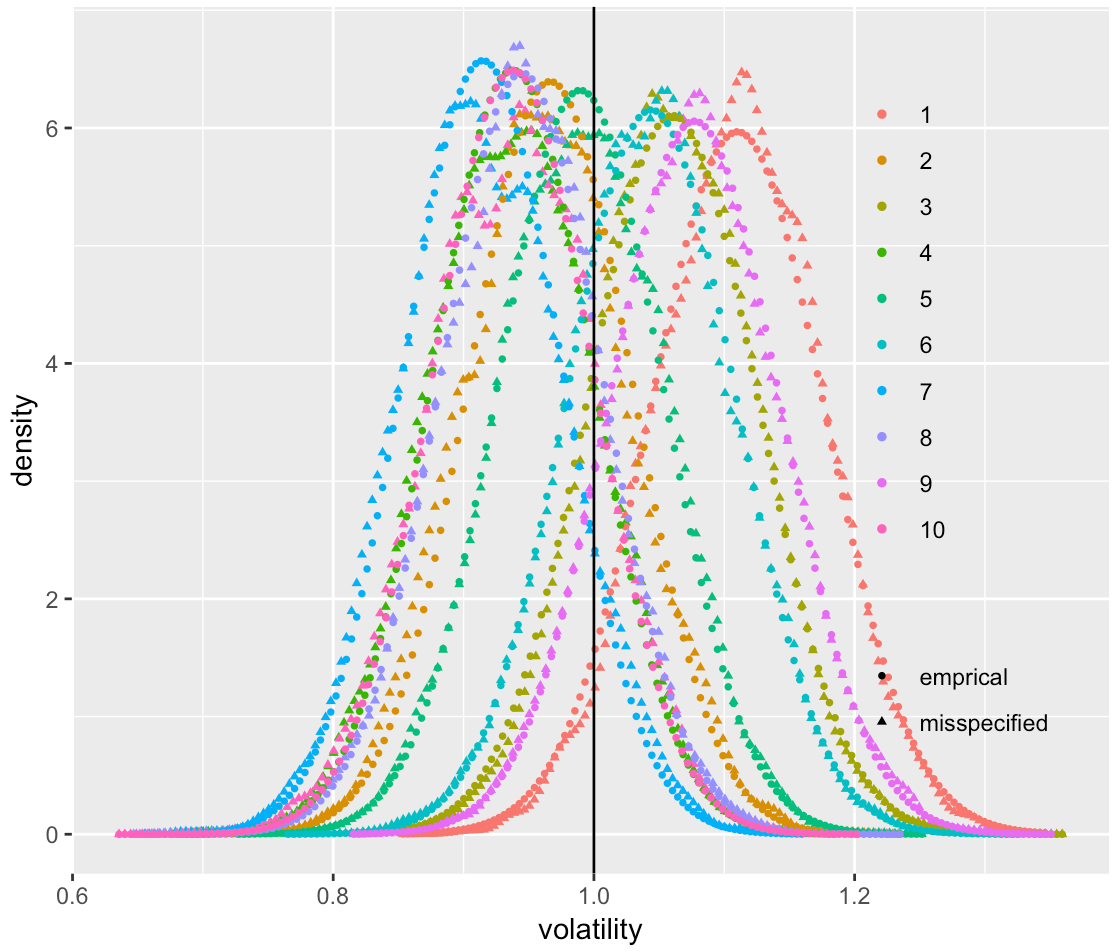} 
		\caption{Posteriors} \label{noisea}
	\end{subfigure}
	\hspace*{\fill} 
	\begin{subfigure}{0.53\textwidth}
		\includegraphics[width=\linewidth]{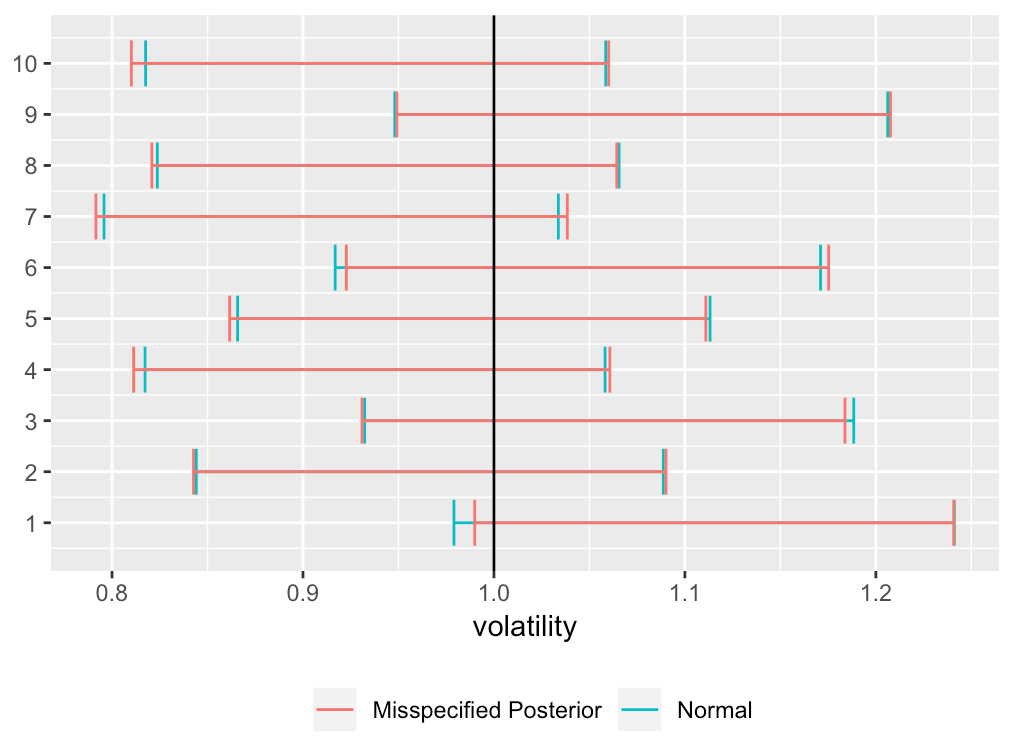}
		\caption{Credible Intervals} \label{noiseb}
	\end{subfigure}
	\caption{\textbf{Comparison with CLT for the model with noise.} \footnotesize  (a) 10 different processes distinguished by 10 different colors are generated and the corresponding posterior are compared. Each color has two distributions. The one formed by little triangle is the misspecified posterior, while the other represents the Gibbs sampling results. (b) The red lines represent the 95\% credible intervals calculated from the Gibbs sampling results of the joint posterior. The blue lines are the 95\% HPD interval for the adjusted posterior. }
\end{figure}

\section{Conclusion}\label{sec:conclusion}

In this paper, we consider an infinite activity model with microstructure noise over a fixed time horizon. A ``purposely misspecified" posterior is proposed for the volatility, the variation of the diffusion component. We theoretically and empirically prove that the posterior can be approximated by a normal distribution centered at a suitable estimator with the optimal variance. Thus, valuable inference can be developed based on the Bayesian credible intervals, whose empirical coverage probability is shown to be close to the nominal one by simulation. 

Compared to \cite{Martin}, we generalize the feature of finite many jumps to infinite jumps, propose an extension to handle stochastic volatility and general It\^o jump processes of bounded variation, and add the microstructure noise. These results suggest more possibilities of further use of the method. 

The purposely misspecified method contributes to the Bayesian framework by ignoring the infinite-dimensional nuisance parameter and directly making inference of the parameter of interest. It provides us a Bayesian posterior without any requirement for assigning a prior and specifying the likelihood of the complicated jump part. Thus, the difficulties of deriving the full posterior and obtaining a marginalized posterior can also be avoided. 

The recentering and rescaling procedure is also highly flexible. Any consistent and efficient estimator can be used as a correction. Furthermore, the variance can be adjusted in response to the information. For example, when  the variance of the volatility $\theta^*$ is foreknown, the temperature parameter can be set to the correct variance over the optimal one.

Considering the good performance of the misspecified model and the complexity of the nonparametric nuisance part, it may be noteworthy to apply the ``purposely misspecified" method to other semiparametric problems.

\appendix 

\section{Realized Quadratic Variation Asymptotics}
In this section, we give some preliminary lemmas regarding the rate of convergence of the realized quadratic variation of a L\'evy process. 
Without loss of generality, throughout we assume $T= 1$, and let $\Delta_n = \frac{1}{n}$.

As it is well-known, unlike a finite jump activity process $J$, the quadratic variation $[J]=\sum_{s\leq{}1}\Delta J_{s}^{2}$ can no longer be expressed for large enough $n$ as a sum of the squared increments of the jump part,
$$
[J]_n = \sum_{i=1}^{n} \Delta_i J^2 = \sum_{i=1}^{n} (J_{t_{i}}- J_{t_{i-1}})^2,$$ 
under the presence of infinitely many jumps on $[0,1]$. However, for bounded variation processes, we shall prove that the latter converges to $[J]$ at a rate of $n^{-1/2}$.  

\begin{lem}\label{realized0}
	Under  Assumption \ref{ass:bddvar}, $$ E_*\left| {[J]} _n {-} [J]  \right| = O(n^{-1/2}).$$
\end{lem}

\begin{proof}
	Let $U_t^n = \left| \sum_{i=1}^{[t/\Delta_{n}]} \Delta_{i} {J}^2 {-} [J]_{t} \right|$, where $[J]_{t}=\sum_{s\leq{}t}\Delta J_{s}$. 
	First, apply It\^{o}'s formula to $J$ with $f(x) = x^2$:
	$$ {J}_{t_i}^2 = {J}_{t_{i-1}}^2 + \int_{t_{i-1}}^{t_i} 2 {J}_{s-}d{J}_{s} + \sum_{s\in(t_{i-1},t_i]: \Delta {J_s}\neq 0} \left[ \left( {J}_{s-} +\Delta {J_s} \right)^2 - {J}_{s-}^2 - 2\Delta {J_s}\cdot {J}_{s-} \right].
	$$
	Then,
	\begin{align*}
	\Delta_{i} {J}^2 &= {J}_{t_i}^2 -  {J}_{t_{i-1}}^2 - 2  {J}_{t_{i-1}} \left(  {J}_{t_{i}}- {J}_{t_{i-1}} \right) \\
	&= 
	\int_{t_{i-1}}^{t_i} 2 {J}_{s-}d{J}_{s} + \sum_{t_{i-1}\le s\le t_i} |\Delta{J}_s|^2 - 2 \int_{t_{i-1}}^{t_i}  {J}_{t_{i-1}}d{J}_{s},
	\end{align*}
	and, thus,
	\[
	\sqrt{n} U_{1}^n = 2\sqrt{n} \left| \sum_{i=1}^{n}  \int_{t_{i-1}}^{t_i} \left( {J}_{s-} - {J}_{t_{i-1}}\right)\,d{J}_{s} \right|.
	\]
	Using the L\'evy-It\^o decomposition (\ref{LIDcmp}), the expression can be `split' into two terms:
	\begin{align}
	\sqrt{n} U_{{1}}^n  & \le 2 \sqrt{n} \sum_{i=1}^{n} \left| \int_{t_{i-1}}^{t_i} \left( {J}_{s-} - {J}_{t_{i-1}}\right)\,d{J}_{1s}  \right| + 2 \sqrt{n} \sum_{i=1}^{n} \left| \int_{t_{i-1}}^{t_i} \left( {J}_{s-} - {J}_{t_{i-1}}\right)\,d\tilde{J}_{2s}  \right|.
	\label{Est3trms}
	\end{align}
	For the second term of \eqref{Est3trms}, by  Assumption \ref{ass:bddvar} and applying  Lemma 2.1.5 of \cite{JacodAndProtterDiscret}
	with $Y_t = \int_{0}^{t} \left( {J}_{s-} - {J}_{t_{i-1}}\right)\,d\tilde{J}_{2s} $, $p=1$,  and $u = t_i - t_{i-1} = \frac{1}{n}$, there exist constants $C_1$, $C_2$, and  $C_3$ such that the expectation of $\left| \int_{t_{i-1}}^{t_i} \left( {J}_{s-} - {J}_{t_{i-1}}\right)\,d\tilde{J}_{2s} \right|$ can be expressed as
	\begin{align*}
	& \quad E_*\left| \int_{t_{i-1}}^{t_i}\int_{|z|\leq{}1} \left( {J}_{s-} - {J}_{t_{i-1}}\right)z(\mu(dz,ds)-\nu(dz)ds)\right|\\
	& \le C_1 E_*\left( \int_{t_{i-1}}^{t_i} \left| {J}_{s-} - {J}_{t_{i-1}}\right| ds \int_{{|z|\le 1}} |z| \,\nu (dz)   \right) \\
	& \le C_2  \int_{t_{i-1}}^{t_i} E_* \left| {J}_{s} - {J}_{t_{i-1}}\right| \,ds \\
	& \le C_2 \int_{t_{i-1}}^{t_i} \left[ E_* \left| {J}_{s} - {J}_{t_{i-1}}\right|^2 \right] ^{\frac{1}{2}} \,ds \\
	& = C_3 \int_{t_{i-1}}^{t_i}  s ^{\frac{1}{2}} \,ds = O\left({n^{-\frac{3}{2}}}\right),
	\end{align*}
	independent of $i$. Thus,
	\begin{align}
	E_*\left[2\sqrt{n} \sum_{i=1}^{n} \left| \int_{t_{i-1}}^{t_i} \left( {J}_{s-} - {J}_{t_{i-1}}\right)\,d\tilde{J}_{2s}  \right| \right] \le 2\sqrt{n} \cdot n \cdot O\left({n^{-\frac{3}{2}}}\right) =O(1).\label{bddterm2jump}
	\end{align}	
	Similarly, for the first term of \eqref{Est3trms}, applying Lemma 2.1.7 of \cite{JacodAndProtterDiscret}
	with $Y_t = \int_{0}^{t} \left( {J}_{s-} - {J}_{t_{i-1}}\right)\,d{J}_{1s} $, $p=1$,   $u = t_i - t_{i-1} = \frac{1}{n}$, and some constant $C_0$, we have
	\begin{align*}
	E_* \left| \int_{t_{i-1}}^{t_i} \left( {J}_{s-} - {J}_{t_{i-1}}\right)\,d{J}_{1s} \right|   & = E_* \left| \int_{t_{i-1}}^{t_i}\int_{|z|>1} \left( {J}_{s-} - {J}_{t_{i-1}}\right)z \mu(dz,ds) \right| \\
	&\le C_0 E_*\left( \int_{t_{i-1}}^{t_i} \left| {J}_{s-} - {J}_{t_{i-1}}\right| ds \int_{{|z|> 1}} |z| \,\nu (dz)   \right) ,
	\end{align*}
	which is  $O\left({n^{-\frac{3}{2}}}\right)$. Following the same analysis as in \eqref{bddterm2jump}, the expectation of the first term of \eqref{Est3trms} is also bounded. Thus, the proof is completed.
\end{proof}

The following lemma gives the rate of convergence of the realized quadratic variation of a general L\'evy process with nonzero Brownian component ($\theta>0$). The results is due to \cite{Jacod2007} (see Theorem 2.6 and Remark 5 therein). Related results for general semimartingales can be found in \cite{Jacod2008}. The convergence of processes below is in the sense of finite-dimensional, stably in law.
\begin{lem}\label{realized1}
	Let 
	\[
	QV_{t}=\sum_{i=1}^{[t/\Delta_{n}]}(\Delta_i^{n}X)^{2},\quad t\geq{}0,
	\]
	be the realized quadratic variation of the process $X$ defined in (\ref{model}) with $\theta>0$ and $J$ satisfying Assumption  \ref{ass:bddvar}-1 (i.e., $X$ is a L\'evy process). Then, 	\begin{equation}\label{DRA0}
	\frac{1}{\sqrt{\Delta_{n}}}\left(QV_{t}-IV_{t}\right)\stackrel{\mathcal{D}}{\to}U_{t},
	\end{equation}
	where $IV_{t}=\theta t+\sum_{s\leq{}t}(\Delta J_{s})^{2}$ and
	\[
	U_{t}=\sqrt{2\theta} W'_{t}+2\sqrt{\theta}\int_{0}^{t}\int_{\mathbb{R}\backslash\{0\}} xZ_{s}\mu(ds,dx),
	\]
	with $W'$ being a Wiener process independent of $W$ and $\{Z_{s}\}_{s\geq{}0}$  being i.i.d. $\mathcal{N}(0,1)$ variables, independent of $W$ and $W'$.
\end{lem}

\section{Proof of Theorem \ref{firstthm}}\label{sec:proofIA}

\begin{proof}[Proof Of Theorem \ref{firstthm}]
	
	We  apply Theorem  \ref{thm:bvm} with
	\begin{align*}
	& Z_j = \Delta_j Y- \Delta_j J = \theta^{1/2}\Delta_j W,\quad  Y_t = J_t, \quad \eta = 0, \quad P_0 = {P}_*,\\
	& \Phi = \theta+[J ], \quad  \Phi^\dag = \theta^*+ [J] = \theta^\dag, \quad \delta_n = n^{-1/2}, \quad \log\tilde{p}(Y^{(n)}|\vartheta) \propto \tilde{l}_n(\vartheta)^{1/\kappa_{n}}.
	\end{align*}
	We shall prove next that the two condition (\ref{eq:con1}) and (\ref{eq:con2}) are satisfied. 
	
	We start with the condition \eqref{eq:con2}, which requires that, for every sequence of constants $
	M_n \rightarrow\infty$,
	\begin{equation*}
	\Pi_n\left( | \vartheta-\theta^\dag| >n^{-1/2} M_n \bigg|   Y^{(n)} \right) \stackrel{P_{*}}{\rightarrow} 0,  \quad n \rightarrow\infty.
	\end{equation*}
	Using Markov's inequality,
	\begin{align*}
	\Pi_n\left( | \vartheta-\theta^\dag| >n^{-1/2} M_n \bigg|   Y^{(n)} \right)  & \le n M_n^{-2} {E}_{\Pi_n}(\vartheta-\theta^\dag)^2.
	\end{align*}
	Since $\tilde{\theta}_n$ is the unique MLE, we could approximate the right hand side expectation by the Laplace approximation (\cite{laplace}):
	\begin{align*}
	{E}_{\Pi_n}(\vartheta-\theta^\dag)^2\approx (\tilde{\theta}_n-\theta^\dag)^2 \{1+O(n^{-1})\}.
	\end{align*}
	Since $M_n\to\infty$, for condition \eqref{eq:con2}, it suffices to show that, for $n\rightarrow \infty$,
	\begin{align}
	|\tilde{\theta}_n - \theta^\dag  |= O_{P_*}(n^{-1/2}).\label{thetathm1}
	\end{align}
	Since $\tilde{\theta}_{n}$ is the realized quadratic variation of the L\'evy process $X$, this directly follows from the central limit theorem of Lemma \ref{realized1}.

	Second, the models should satisfy the stochastic local asymptotic normality (LAN) condition \eqref{eq:con1}. That means that, for every $\epsilon>0$,
	\begin{align*}
	P_* \left( \sup_{h\in K} \left| \frac{1}{\kappa_n} \log \frac{\tilde{l}_{n}(\theta^\dag+\delta_n h)} {\tilde{l}_{n}(\theta^\dag)} -V_{\theta^\dag} \Delta_{n,\theta^\dag} h -\frac{1}{2}V_{\theta^\dag}h^2 \right|>\epsilon \right) = o(1).
	\end{align*}
	Let $V_{\theta^\dag}=(2\kappa^\dag \theta^{\dag 2})^{-1}$ and $ \Delta_{n,\theta^\dag} = n^{1/2}\left( \tilde{\theta}_n - \theta^\dag \right) $. Using Taylor expansion to approximate the log likelihood and plugging in the first and second derivatives, the left hand side of the LAN condition can be written as
	\begin{align*}
	P_* & \left(  \sup_{h\in K}  \left| -\frac{h}{2\kappa^\dagger \theta^{\dag 2}} \left( 1-\frac{\kappa^\dagger}{\kappa_n} \right) \sqrt{n}\left( \tilde{\theta}_n - \theta^\dag \right) \right.\right.\\
	&\quad \left.\left.
	+  \frac{h^2}{4\kappa^\dag \theta^{\dag 2}} \left[ \frac{2\kappa^\dagger}{\theta^\dag \kappa_n}  \sqrt{n}\left( \tilde{\theta}_n - \theta^\dag \right) + \left(1-\frac{\kappa^\dagger}{\kappa_n}\right)  \right] 
	\right|>\epsilon\right).
	\end{align*}
	Noted that $P_*(|\kappa_n- \kappa^\dagger|>\delta) \rightarrow 0$ when $n\rightarrow\infty$ for arbitrary $\delta>0$ and, from the result obtained in Lemma \ref{realized1}, 
	the LAN condition holds.
	
	Therefore, for arbitrary $\epsilon>0$, 
	\begin{align*}
	&  P_* \left(  \sup_B\left| \Pi_n \left( (\vartheta  - \theta^\dag )/\delta_n \in B |X^{(n)}\right) - N_{\Delta_{n,\theta^\dag}, V_{\theta^\dag}^{-1}}(B) \right| >\epsilon\right) \rightarrow 0, \  n\rightarrow \infty.\\
	\Leftrightarrow & \quad P_*\left\{ TV(\Pi_n, \mathcal{N}(\tilde{\theta}_n , 2\kappa^\dag \theta^{\dag 2} n^{-1} )  ) >\epsilon \right\}  \rightarrow 0.
	\end{align*}	
	Thus, we finishes the proof of Theorem \ref{firstthm}.
\end{proof}

\section{Proof of Theorem \ref{noisemisthm}} \label{app:noiseproof}
As in the proof of Theorem \ref{firstthm},
we apply Theorem \ref{thm:bvm} with
\begin{align*}
&Z_j = \Delta_j Y- \Delta_j J = \theta^{1/2}\Delta_j W + \varepsilon_{t_j} - \varepsilon_{t_{j-1}},\quad  Y_t = J_t, \quad
\eta = \sigma_\varepsilon^2, \quad P_0 = {P}_*,  \\
& \Phi = \theta+[J ], \quad  \Phi^\dag = \theta^*+ [J] = \theta^\dag, \quad \delta_n = n^{-1/4}, \quad \log\tilde{p}(Y^{(n)}|\vartheta) \propto \tilde{l}_n(\vartheta)^{1/\kappa_{n}}.
\end{align*}
As before, there are two conditions that need to be satisfied. The first is the LAN property \eqref{eq:con1}, which will be proved in \S~\ref{sec:lan}. The second condition is \eqref{eq:con2}. By applying the same Markov inequality and Laplace approximation as in the proof of Theorem \ref{firstthm}, we can conclude that a sufficient condition for \eqref{eq:con2} is
\begin{equation*}
|\tilde{\theta}_n-\theta^\dag| = O_{P_*}(n^{-1/4}),
\end{equation*}
which will be proved in \S~\ref{sec:consistency}. Before the conditions for the theorem are demonstrated, we give some preliminary lemmas. 

The following notations are often used throughout the proof:
\begin{enumerate}
	\item $a_n\lesssim b_n$ indicates that there exists a constant $C$ such that {$|a_n|\le C |b_n|$} for every $n$ large enough. If $a_n\lesssim b_n$ and $b_n\lesssim a_n$, then we write $a_n\asymp b_n$.
	\item To simpify the notation, in what follows, we use $p_{ij}$ to represent $p_{ij}^n$, and  $\lambda_{j}$ to represent $\lambda_{j}^n$.
\end{enumerate}

\subsection{Preliminary lemmas}

Without loss of generality, we assume $E_* [\Delta_j J] = 0$. Otherwise, the drift component $\mu$ can be redefined  correspondingly to $\mu + E_* [\Delta_j J]/\Delta_n$ to obtain the required  of zero-expectation property of the jump increments $\Delta_j J$. 

The limiting behavior of the moments of the jump increments will be frequently used later in the proof. We summarize it in the following lemma.

\begin{lem} 
	Under assumptions \ref{ass:noise}, and \ref{ass:finite16}, for $k = 1,2,\ldots, 16$, and $i=1,2,\ldots,n$, 
	\begin{align}
	(i)\;E_*\left[ |\Delta_i J|^k \right] = O(n^{-1}), \label{eq:JPropkth} \quad \text{ and }\quad (ii)\;E_*\left[ \left(\frac{\mu}{n}+\sum_{j=1}^{n} p_{ij} \Delta_j J\right)^k \right] & = O(n^{-1}) . 
	\end{align}
	A faster convergence rate is needed later for the fourth moment: 
	\begin{align}
	E_*\left[ \left(\frac{\mu}{n}+\sum_{j=1}^{n} p_{ij} \Delta_j J\right)^4 \right] & = O(n^{-2 }) . \label{eq:JWconvmulti} 
	\end{align}
\end{lem}
The first statement directly follows from Theorem 4.3 of \cite{FIGUEROA2008jumpmoment}. The proof of the second and the third statements is included in Appendix \ref{AddProfApend}.

Recall that $\sigma_\varepsilon^2$ is the variance of the noise, which can be estimated using $\hat{\sigma}_\varepsilon^2 = \frac{1}{2n} \sum_{i=1}^n \Delta_i Y^2 $. The following result states some needed properties of $\hat{\sigma}_\varepsilon$.
\begin{lem}\label{lem:Nconv} Under assumption \ref{ass:finite16}, we have
	$$P_*\left( \left|\hat{\sigma}_\varepsilon^2-\sigma_\varepsilon^2\right| \ge \frac{\sigma_\varepsilon^2}{2} \right) = O({n^{-4}}),\quad \text{ and }\quad E_*|\hat{\sigma}_\varepsilon^2-\sigma_\varepsilon^2| = O(n^{-1/2}).$$
\end{lem}
\begin{proof}
	Let $H = \sqrt{n}\left( \hat{\sigma}_\varepsilon^2-\sigma_\varepsilon^2\right)$. When $H$ has finite $8$th moment, by Markov's inequality, we have  
	\begin{align*}
	& P_*\left( \left|\hat{\sigma}_\varepsilon^2-\sigma_\varepsilon^2\right| \ge \frac{ \sigma_\varepsilon^{2}}{2}\right)= P_*\left( n^4\left|\hat{\sigma}_\varepsilon^2-\sigma_\varepsilon^2\right|^{8} \ge n^4\frac{\sigma_\varepsilon^{16}}{2^{8}} \right) \le  \frac{2^{8}E_*\left[H^{8}\right] }{n^4\sigma_\varepsilon^{16}}  = O(n^{-4}), \\
	& E_*\left[ \left|\hat{\sigma}_\varepsilon^2-\sigma_\varepsilon^2\right|^2 \right]= E_*\left[H^{2}\right]/n =O(n^{-1}).
	\end{align*}
	Then, it suffices to demonstrate that $H$ has finite $8$th moment. We prove this by analyzing the moment generating function (m.g.f.) of $H$ conditioned on $J$, and showing that the $8$th derivative of that m.g.f. is finite at zero.
	
	Recall that for $j\in\{1,2,\ldots,n\}$, given $J$, ${R}_j$ independently follow normal distribution $\mathcal{N}\left( \frac{\mu}{n} + \sum_{i=1}^{n} p_{ij} \Delta_i J, \lambda_j  \right)$, where $\lambda_j = \frac{\theta^*}{n}+ 2\sigma_\varepsilon^2\left( 1-\cos\frac{j\pi}{n+1} \right)$. Then,
	$$ \frac{{R}_j^2}{\lambda_j} \sim \chi_1^2(\alpha_j), \quad \text{where }\alpha_j = \frac{\left(\frac{\mu}{n}+\sum_{i=1}^{n} p_{ij} \Delta_i J\right)^2}{\lambda_j}.$$
	Because $P_n$ is an the orthogonal matrix , we have $\hat{\sigma}_\varepsilon^2 =\frac{1}{2n}  \sum_{j=1}^{n} (\Delta_j {Y})^2 = \frac{1}{2n} \sum_{j=1}^{n} R_j^2 $. Then, $H$ can be written as $\frac{1}{2\sqrt{n}}\sum_{j=1}^{n}{R}_j^2 - \sqrt{n}\sigma_\varepsilon^2$. The m.g.f. of $H$ conditioned on $J$ is therefore
	\begin{align*}
	M_H(t) & = e^{-\sqrt{n} \sigma_\varepsilon^2 t} \prod_{i=1}^{n} M_{R_i^2}\left( \frac{t}{2\sqrt{n}} \right)\\
	& = e^{-\sqrt{n} \sigma_\varepsilon^2 t} \prod_{i=1}^{n}exp\left( \frac{\alpha_i t_i}{1-2t_i} \right) (1-2t_i)^{-\frac{1}{2}}, \quad \text{where } t_i = \frac{\lambda_i t}{2\sqrt{n}}.
	\end{align*}
	Let $M^{(0)} = M_H$, and the $k$th-derivative $\frac{d^k}{dt^k} M_H(t):= M^{(k)}(t)$. Then,
	\begin{align*}
	M^{(1)}(t)  = &  \left[ -\sqrt{n} \sigma_\varepsilon^2 + \sum_{i=1}^{n}(1-2t_i)^{-2} \frac{\alpha_i \lambda_i}{2\sqrt{n}} + \sum_{i=1}^{n}(1-2t_i)^{-1} \frac{\lambda_i}{2\sqrt{n}} \right] M^{(0)}(t),\\
	M^{(2)}(t)  = & \left[ -\sqrt{n} \sigma_\varepsilon^2 + \sum_{i=1}^{n}(1-2t_i)^{-2} \frac{\alpha_i \lambda_i}{2\sqrt{n}} + \sum_{i=1}^{n} (1-2t_i)^{-1} \frac{\lambda_i}{2\sqrt{n}} \right] M^{(1)}(t)\\
	& + \left[\sum_{i=1}^{n}4(1-2t_i)^{-3}\frac{\alpha_i \lambda_i^2}{4n} + \sum_{i=1}^{n}2 (1-2t_i)^{-2}\frac{\lambda_i^2}{4{n}}\right] M^{(0)}(t).
	\end{align*}
	Similarly, 
	\begin{align*}
	& M^{(8)}(t) =   -\sqrt{n} \sigma_\varepsilon^2 M^{(7)} + \\
	& \sum_{k=0}^7 C_k \left[\sum_{i=1}^{n}\frac{k+1}{(1-2t_i)^{k+2}}\frac{\alpha_i \lambda_i^{k+1}}{(2\sqrt{n})^{k+1}} + \sum_{i=1}^{n} \frac{1}{(1-2t_i)^{k+1}}\frac{\lambda_i^{k+1}}{(2\sqrt{n})^{k+1}}\right] M^{(7-k)},
	\end{align*}
	where $C_0 = 1$, $C_1=14$, $C_2=168$, $C_3=1680$, $C_4=13440$, $C_5=80640$, $C_6=322560$, and $C_7=645120$.
	
	Let $J_W = \sum_{i=1}^{n}\left(\frac{\mu}{n}+\sum_{j=1}^{n} p_{ij} \Delta_j J\right)^2$,  $J_{Wk} = \sum_{i=1}^{n}\left(\frac{\mu}{n}+\sum_{j=1}^{n} p_{ij} \Delta_j J\right)^2\lambda_i^k$ and $\Lambda^k = \sum_{i=1}^{n}\lambda_i^k$. We adopt $c_{n,0} = 1, c_{n,1}, c_{n,2}, \ldots, c_{n,8}$ to express the numerator of the $i$-th moment of $H$, which is defined as follows.
	\begin{align*}
	M^{(1)}(0)  & =
	\frac{\theta^* + \sum_{i=1}^{n}\left(\frac{\mu}{n}+\sum_{j=1}^{n} p_{ij} \Delta_j J\right)^2}{2\sqrt{n}} = \frac{\theta^*+J_W}{2\sqrt{n}} :=\frac{c_{n,1}}{2\sqrt{n}},\\ 
	M^{(2)}(0) &= \frac{(J_W+\theta^*)c_{n,1} + 2(2J_{W1} + \Lambda^2 )}{4n} := \frac{c_{n,2}}{4n},\\
	&\ldots\\
	M^{(8)}(0) & = \frac{1}{2^{8}n^4} \left[ (J_W+\theta^*) c_{n,7} + \sum_{k=1}^{7} C_{k}((k+1)J_{Wk}+\Lambda^{k+1}) c_{n,7-k} \right] := \frac{c_{n,8}}{2^{8}n^4}.
	\end{align*}
	The terms in $c_{n,8}$ can be expressed in general as 
	$$(\theta^*+ J_W)^{m_1}\prod_{k=2}^{8} \left( (k+1)J_{Wk}+\Lambda^k\right)^{m_k} \lesssim (\theta^*+ J_W)^{m_1}\prod_{k=2}^{8} \left( J_{W}+\Lambda^k\right)^{m_k} ,$$
	where $\sum_{k=1}^{8} km_k= 8$, $m_k\in \mathbb{N}$. 	
	By \eqref{eq:JPropkth}-\eqref{eq:JWconvmulti}, and following the similar procedure as proving \eqref{eq:JPropkth}, for any $1\le k\le 8$, 
	\begin{align*}
	E_*[J_W^k] =E_*\left[  \left(\sum_{i=1}^{n}\left(\frac{\mu}{n}+\sum_{j=1}^{n} p_{ij} \Delta_j J\right)^2\right)^k\right] = O(1).
	\end{align*}
	Then,
	\begin{align*}
	& E_*[J_{Wk}] \le E_*\left[ \sum_{i=1}^{n} \left(\frac{\mu}{n}+\sum_{j=1}^{n} p_{ij} \Delta_j J\right)^2 \left(\frac{\theta^*}{n}+4\sigma_\varepsilon^2\right)^k \right] \lesssim E_*[J_W] = O(1),\\
	& \Lambda^k \le \sum_{j=1}^{n} \left(\theta^*+ 2\sigma_\varepsilon^2 \right)^{k}  =O(n). 
	\end{align*}
	
	Thus, all the terms in $c_{n,10}$ have expectations of  $O(n^4)$. The rate $n^4$ can be only achieved when $m_2 = 4, m_j = 0, j = 1,3,4,\ldots,8$.
	Thus, $E_*[c_{n,10}]=O(n^4)$. This means that $H$ has $8$th finite moment. We complete the proof.
\end{proof}

The following Lemma will be used later to prove the asymptotic properties of the log likelihood function \eqref{likelihood}, and its derivative.

\begin{lem}\label{lem:sumco} 
	For any fixed constants $a>0,b>0${,} and $p\ge 1$, 
	\begin{equation*}
	\sum_{j=1}^n \frac{1}{\left(  a+ 2bn(1-\cos \frac{j\pi}{n+1}) \right)^p } \asymp n^{\frac{1}{2}}.
	\end{equation*}
\end{lem}

\begin{proof}
	For the lower bound, note that since $\sin x\le x$,
	\begin{align*}
	\sum_{j=1}^n & \frac{1}{\left(  a+ 2bn(1-\cos \frac{j\pi}{n+1}) \right)^p }  \ge \sum_{j=1}^{\sqrt{n}} \frac{1}{\left(  a+ 4bn(\sin\frac{j\pi}{2(n+1)})^2 \right)^p }\\
	& \ge \sum_{j=1}^{\sqrt{n}} \frac{1}{\left(  a+ 4bn \frac{j^2\pi^2}{4(n+1)^2} \right)^p } \stackrel{j\le \sqrt{n}}{\ge} \sum_{j=1}^{\sqrt{n}} \frac{1}{\left(  a+ b\pi^2 \right)^p } \gtrsim \sqrt{n}.
	\end{align*}
	For the upper bound, we divide the summation into two parts. 
	For $j\le \sqrt{n}$, we have
	$$\sum_{j=1}^{\sqrt{n}} \frac{1}{\left(  a+ 2bn(1-\cos \frac{j\pi}{n+1}) \right)^p } \le \sum_{j=1}^{\sqrt{n}} \frac{1}{a^p } \lesssim \sqrt{n}.$$
	For $j > \sqrt{n}$, since $\sin{x}>\frac{2}{\pi}x$ for $0< x<\frac{\pi}{2}$, 
	$$ n\left( 1-\cos{ \frac{j\pi}{n+1} } \right)  = 2n\sin^2{\frac{j\pi}{2(n+1)}} \ge 2\frac{j^2n}{(n+1)^2}>\frac{j^2}{n+1}.$$
	Then,
	\begin{align*}
	\sum_{j=\sqrt{n}}^{{n}} \frac{1}{\left(  a+ 2bn(1-\cos \frac{j\pi}{n+1}) \right)^p } &\le \sum_{j=\sqrt{n}}^{{n}} \frac{1}{\left(  2bn(1-\cos \frac{j\pi}{n+1}) \right)^p } \\
	& \le \sum_{j=\sqrt{n}}^n \left( \frac{n+1}{2bj^2}\right)^p\\
	& \le \int_{\sqrt{n}-1}^n \left( \frac{n+1}{2b x^2}\right)^p\,dx\lesssim\sqrt{n}.
	\end{align*}
\end{proof}

Recall that $\lambda_j(\theta) = \frac{\theta}{n} + 2\sigma_\varepsilon^2 \left(1-\cos \frac{j\pi}{n+1}\right)$. Applying Lemma \ref{lem:sumco} with $a=\theta$ and $b = \sigma_\varepsilon^2$, and since $0\notin \bar{\Theta}'$, we get
\begin{align}\label{sumsqrtn}
\sup_{\theta\in\Theta'}\sum_{j=1}^n \frac{1}{n^p\lambda_j^p(\theta)}= O(n^{{1}/{2}}), 
\quad p=2,3,\ldots
\end{align}

\subsection{Likelihood functions}

In this section, we introduce several properties of the misspecified likelihood function $\tilde{l}_n$ defined in \eqref{likelihood}. The results will be needed when checking the conditions of Theorem \ref{thm:bvm}. 

In the misspecified model \eqref{mis}, when the variance of the noise is assumed to be known, the likelihood function of $\theta$ is given by
\begin{align}
l_n(\theta) & = \sum_{j=1}^{n}\left\{ -\frac{1}{2} \log \left(\frac{\theta }{n} + 2\sigma_\varepsilon^2 \left(1-\cos \frac{j\pi}{n+1}\right)\right) -\frac{1}{2} \frac{{R}_j^2}{\frac{\theta}{n}+2\sigma_\varepsilon^2 \left(1-\cos \frac{j\pi}{n+1}\right)}\right\}\nonumber\\
& = \sum_{j=1}^{n}\left\{  -\frac{1}{2}\log \lambda_j(\theta) - \frac{1}{2}\frac{R_j^2}{\lambda_j(\theta)}\right\}.\label{eq:loglikeli2}
\end{align}
Recall that 
$$\tilde{l}_n(\theta) = \sum_{j=1}^{n} \left\{ -\frac{1}{2} \log \left(\frac{\theta}{n} + 2\hat{\sigma}_\varepsilon^2 \left(1-\cos \frac{j\pi}{n+1}\right)\right) -\frac{1}{2} \frac{R_j^2}{\frac{\theta}{n}+2\hat{\sigma}_\varepsilon^2 \left(1-\cos \frac{j\pi}{n+1}\right)}\right\}.$$
In what follows, we denote the corresponding first and second derivatives of $l_n$ and $\tilde{l}_n$ with respect to $\theta$ as $\dot{l}_n$,  $\dot{\tilde{l}}_n$, $\ddot{l}_n$, and $\ddot{\tilde{l}}_n$, respectively. 

The moments of the variable $R_j$ are frequently used in the following lemmas. We summarize them here. Recall that given $J$, the ${R}_j$s independently follow the normal distribution $$\mathcal{N}\left(\frac{\mu}{n}+\sum_{i=1}^{n} p_{ij} \Delta_j J, \lambda_j(\theta^*) \right),$$ where $ \lambda_j(\theta^*) = \frac{\theta^*}{n}+ 2\sigma_\varepsilon^2\left( 1-\cos\frac{j\pi}{n+1} \right)$.
By \eqref{eq:JPropkth}-\eqref{eq:JWconvmulti}, the moments of $R_j$ are such that
\begin{align}\label{eq:ER2}
E_*{R_j^2} &= \lambda_j(\theta^*)+ O(n^{-1}),\\
E_*{R_j^4} &= E_*\left[ \left( \frac{\mu}{n}+\sum_{i=1}^{n} p_{ij} \Delta_j J\right)^4 + 6 \left( \frac{\mu}{n}+\sum_{i=1}^{n} p_{ij} \Delta_j J\right)^2\lambda_j(\theta^*) + 3 \lambda_j^2(\theta^*) \right] \nonumber \\
& \lesssim n^{-2} + n^{-1}\lambda_j(\theta^*)+\lambda_j^2(\theta^*) \lesssim \lambda_j^2(\theta^*)+ O(n^{-2}) \label{eq:ER4}.
\end{align}
The last inequality holds because $n^{-1} \le \lambda_j(\theta^*)/ \theta^*$.

The following Lemma \ref{lem:likeliep} aims to control the difference between $l$ and $\tilde{l}$ and their corresponding derivatives. Recall that $l$ uses the true variance $\sigma_\varepsilon^2$, while $\tilde{l}$ adopts  $\hat{\sigma}_\varepsilon^2$ to replace $\sigma_\varepsilon^2$. With Lemma \ref{lem:likeliep}, the desired properties of $\tilde{l}$ can be attained through the analysis of $l$, which simplifies the proof by using the true parameter instead of an estimator. 

\begin{lem}\label{lem:likeliep} 
	Let $l$ and $\tilde{l}$ be given as in \eqref{likelihood} and \eqref{eq:loglikeli2}, respectively. If assumptions \ref{ass:noise} and \ref{ass:finite16} hold true, then
	for any  integer $k\ge 1$,
	\begin{equation}\label{eq:lemlikeli2}
	\sup_{\theta\in\Theta'}\left|\frac{d^k l_n(\theta)}{d\theta^k} - \frac{d^k \tilde{l}_n(\theta)}{d\theta^k}\right| = O_{P_*}(1).
	\end{equation}
\end{lem}
\begin{proof}
	Define 
	\begin{align*}
	& \lambda_j(\theta, x) = \frac{\theta}{n}+ 2x\left( 1-\cos  \frac{j\pi}{n+1}\right).
	\end{align*} 
	Then, $\lambda_j(\theta) = \lambda_j(\theta, \sigma_\varepsilon^2)$, and $l_n$ and $\tilde{l}_n$ can be written using $\lambda_j(\theta, \sigma_\varepsilon^2)$ and $\lambda_j(\theta, \hat{\sigma}_\varepsilon^2)$ as
	\begin{align*}
		&	l_n(\theta) = -\frac{1}{2}\sum_{i=1}^n \left\{ \log\lambda_j(\theta, \sigma_\varepsilon^2) + \frac{R_j^2}{\lambda_j(\theta, \sigma_\varepsilon^2)} \right\},\\
		&	\tilde{l}_n(\theta) = -\frac{1}{2}\sum_{i=1}^n \left\{ \log\lambda_j(\theta, \hat{\sigma}_\varepsilon^2) + \frac{R_j^2}{\lambda_j(\theta, \hat{\sigma}_\varepsilon^2)}\right\}.
	\end{align*}
	The expressions inside the absolute values in \eqref{eq:lemlikeli2} can {then} be expressed as
	$$ \sum_{j=1}^n \left[g_j(\hat{\sigma}_\varepsilon^2) - g_j(\sigma_{\varepsilon}^2) \right] + \sum_{j=1}^n n R_j^2 \left[h_j(\hat{\sigma}_\varepsilon^2) - h_j(\sigma_{\varepsilon}^2) \right] =: \sum_{j=1}^n \alpha_j +\sum_{j=1}^n \beta_j =: G_n+H_n,$$
	where
	$$ g_j(x) = \frac{(-1)^{k}(k-1)!}{2n^k\lambda_j^k(\theta, x)}, \quad  h_j(x) =  \frac{(-1)^{k+1}k!}{2n^{k+1}\lambda_j^{k+1}(\theta, x)}.$$
	
	We first derive an upper bound for the first and second derivative of $g$ and $h$. Fixed a $\tau>0$, and consider all $x\ge \tau$. Noted that $2n\left(1-\cos\frac{j\pi}{n+1}\right) \le n\lambda_j(\theta, x)/x \le n\lambda_j(\theta, x)/\tau $, we have
	\begin{align*}
	\sum_{j=1}^n \left|\frac{\partial^m  }{\partial x^m} \frac{1}{n^k\lambda_j^k(\theta,x)} \right| &= \sum_{j=1}^n \frac{2^m n^m (1-\cos  \frac{j\pi}{n+1})^m}{\left( n\lambda_j(\theta,x) \right)^{m+k}} \\
	&\le \frac{1}{\tau^{m}}\sum_{j=1}^n \frac{1}{\left( n\lambda_j(\theta,x) \right)^{k}} \stackrel{n\lambda_j(\theta, x) \ge \theta}{\le} \frac{1}{\tau^{m}\theta^{k-1}} \sum_{j=1}^n \frac{1}{n\lambda_j(\theta,\tau)}.
	\end{align*}
	By \eqref{sumsqrtn}, 
	\begin{align}\label{eq:likeli_tau}
	& \sup_{\theta \in \Theta'} \sup_{x>\tau} \sum_{j=1}^n |g'_j(x)| = O(n^{{1}/{2}}),\quad \sup_{\theta \in \Theta'}\sup_{x>\tau} \sum_{j=1}^n |g''_j(x)| = O(n^{{1}/{2}}),\\
	& |h'_j(\tau)| \lesssim \frac{1}{n^2\lambda_j^2(\theta,\tau)},
	\quad  \sup_{x>\tau}|h''_j(x)| \lesssim \frac{1}{n^2\lambda_j^2(\theta,\tau)}.\label{eq:likeli_tau_h}
	\end{align}
	
	Next, let $\tau = 0$. Noted that $\lambda_j(\zeta, x)\ge\lambda_j(\theta, x)\ge  \frac{\theta}{n}$, and
	\begin{align*}
	&    \frac{\partial^2  }{\partial x^2} \frac{1}{n^k\lambda_j^k(\theta,x)} = \frac{2^2 n^2 (1-\cos  \frac{j\pi}{n+1})^2}{n^{2+k}\lambda_j^{2+k}(\theta, x)} \lesssim \frac{n^2}{\theta^{2+k}} = O(n^2),
	\end{align*}
	we have
	\begin{align}\label{eq:likeli_0}
	&\sup_{\theta \in \Theta}\sup_{x\ge 0}|g''_j(x)| = O(n^2),\quad \sup_{\theta \in \Theta}\sup_{x\ge 0}  |h''_j(x)| = O(n^2).
	\end{align}
	
	Now, we are ready to prove the boundedness of $G_n$ and $H_n$ uniformly for $\theta\in\Theta'$. Let $\mathcal{D}_n$ be the interval between $\hat{\sigma}_\varepsilon^2$ and $\sigma_\varepsilon^2$. By Taylor expansion, $\alpha_j = g_j(\hat{\sigma}^2_\varepsilon) - g_j({\sigma}^2_\varepsilon)$ can be bounded:
	\begin{align*} 
	\sum_{j=1}^n |\alpha_j| 
	&\le |\hat{\sigma}_\varepsilon^2 - \sigma_\varepsilon^2| \sum_{j=1}^n |g'_j(\sigma_\varepsilon^2)| + \frac{1}{2}\sum_{j=1}^n \sup_{x\in \mathcal{D}_n} |g''_j(x)| \cdot |\hat{\sigma}_\varepsilon^2 - \sigma_\varepsilon^2|^2
	\end{align*}
	The first term is $O_{P_*}(1)$ by Lemma \ref{lem:Nconv}, and \eqref{eq:likeli_tau} with $\tau = \sigma_\varepsilon^2/2$. For the second term, we break up $\mathcal{D}_n$ into two sub-regions: $\{x> {\sigma_\varepsilon^2}/{2}\}\cap\mathcal{D}_n$ and $\{x\le {\sigma_\varepsilon^2}/{2}\}\cap\mathcal{D}_n$. Then, the second term, $\frac{1}{2}\sum_{j=1}^n \sup_{x\in \mathcal{D}_n} |g''_j(x)| \cdot |\hat{\sigma}_\varepsilon^2 - \sigma_\varepsilon^2|^2$, can be bounded by 
	\begin{align*}
	\frac{1}{2}&\sum_{j=1}^n \sup_{x>{\sigma_\varepsilon^2}/{2}} |g''_j(x)| \cdot |\hat{\sigma}_\varepsilon^2 - \sigma_\varepsilon^2|^2 \\
	&+ \frac{1}{2}\sum_{j=1}^n \sup_{x\le {\sigma_\varepsilon^2}/{2}} |g''_j(x)| \cdot |\hat{\sigma}_\varepsilon^2 - \sigma_\varepsilon^2|^2 {1}_{\left\{ \frac{\sigma_\varepsilon^2}{2}\le  |\hat{\sigma}_\varepsilon^2 - \sigma_\varepsilon^2|\le {\sigma_\varepsilon^2} \right\} }
	\end{align*}
	In the first sub-region, $x$ is bounded away from zero, so we can still apply Lemma \ref{lem:Nconv}, and \eqref{eq:likeli_tau} with $\tau = \sigma_\varepsilon^2/2$, and obtain  $O_{P_*}(1)$. For the second sub-region, $x\le {\sigma_\varepsilon^2}/{2}$, which means $x$ lies far away from the true variance, $\sigma_\varepsilon^2$. If this sub-region is not empty, then, $0\le \hat{\sigma}_\varepsilon^2\le {\sigma_\varepsilon^2}/{2}$. This is covered by the region $\frac{\sigma_\varepsilon^2}{2}\le  |\hat{\sigma}_\varepsilon^2 - \sigma_\varepsilon^2|\le {\sigma_\varepsilon^2} $. Thus, we can bound $|\hat{\sigma}_\varepsilon^2 - \sigma_\varepsilon^2|^2$ using ${\sigma_\varepsilon^4}$. By \eqref{eq:likeli_0},
	\begin{align*}
	E_*&\left[ \sum_{j=1}^n \sup_{x\le {\sigma_\varepsilon^2}/{2}} |g''_j(x)| \cdot |\hat{\sigma}_\varepsilon^2 - \sigma_\varepsilon^2|^2 {1}_{\left\{\frac{\sigma_\varepsilon^2}{2}\le  |\hat{\sigma}_\varepsilon^2 - \sigma_\varepsilon^2\le {\sigma_\varepsilon^2} |\right\} }\right] \\
	& \le E_*\left[ \sum_{j=1}^n n^2 2\sigma_\varepsilon^4 {1}_{\left\{  |\hat{\sigma}_\varepsilon^2 - \sigma_\varepsilon^2|\ge \frac{\sigma_\varepsilon^2}{2} \right\} }\right]
	= 2n^3\sigma_\varepsilon^4 P_*\left( |\hat{\sigma}_\varepsilon^2 - \sigma_\varepsilon^2| \ge \frac{\sigma_\varepsilon^2}{2} \right) = O(n^{-1}).
	\end{align*}
	The last equality is based on Lemma \ref{lem:Nconv}.
	Therefore, $G_n = O_{P_*}(1)$. 
	
	For $H_n$, similarly, using Taylor expansion,
	\begin{align}\label{eq:Hn}
	\sum_{j=1}^n |\beta_j| &\le \sum_{j=1}^n |h'_j(\sigma_\varepsilon^2)| \cdot |nR_j^2 (\hat{\sigma}_\varepsilon^2 - \sigma_\varepsilon^2)|  + \frac{1}{2}\sum_{j=1}^n \sup_{x\in \mathcal{D}_n} |h''_j(x)| \cdot |nR_j^2(\hat{\sigma}_\varepsilon^2 - \sigma_\varepsilon^2)^2|.
	\end{align}
	For the first term of \eqref{eq:Hn}, recalling that $\hat{\sigma}_\varepsilon^2=\frac{1}{2n}\sum_{i=1}^n R_i^2$, in order to obtain the expectation of $R_j^2(\hat{\sigma}_\varepsilon^2 - {\sigma}_\varepsilon^2)$, we  take expectation of $R_j^2$ and $\hat{\sigma}_\varepsilon^2 - {\sigma}_\varepsilon^2$ seperately, subtract the term $\left(E_*R_j^2\right)^2/(2n)$, and then add back the term $E_*R_j^4/(2n)$:
	\begin{align*}
	E_*\left| nR_j^2  (\hat{\sigma}_\varepsilon^2 - \sigma_\varepsilon^2) \right| &\le  nE_* R_j^2 \cdot E_*\left| \hat{\sigma}_\varepsilon^2  - \sigma_\varepsilon^2  \right| + nE_* R_j^2 \cdot \frac{1}{2n} E_* R_j^2 + \frac{1}{2}E_*R_j^4.
	\end{align*}
	All the three terms above can be bounded applying Lemma \ref{lem:Nconv},  \eqref{eq:ER2}, and \eqref{eq:ER4}:
	\begin{align*}
	& nE_* R_j^2 \cdot E_*\left| \hat{\sigma}_\varepsilon^2  - \sigma_\varepsilon^2  \right| \lesssim n(\lambda_j(\theta^*) + n^{-1})n^{-1/2} \lesssim n^{\frac{1}{2}}\lambda_j(\theta^*) + n^{-\frac{1}{2}},\\
	& nE_* R_j^2 \cdot \frac{1}{2n} E_* R_j^2 \lesssim (\lambda_j(\theta^*) + n^{-1})^2 \lesssim \lambda_j(\theta^*) + n^{-2}, \\
	& E_*R_j^4 = \lambda_j^2(\theta^*) + n^{-2} \lesssim \lambda_j(\theta^*) + n^{-2}.
	\end{align*}
	The last inequality is because $\lambda_j(\theta) \le \theta + 2\sigma_\varepsilon^2$.
	We then have 
	$$E_*\left| nR_j^2  (\hat{\sigma}_\varepsilon^2 - \sigma_\varepsilon^2) \right|\lesssim n^{\frac{1}{2}}\lambda_j(\theta^*) + n^{-\frac{1}{2}}. $$ By \eqref{eq:likeli_tau_h},
	\begin{align*}
	E_*\sum_{j=1}^n |nR_j^2 (\hat{\sigma}_\varepsilon^2 - \sigma_\varepsilon^2)|  |h'_j(\sigma_\varepsilon^2)|
	&\lesssim \sum_{j=1}^n \left(n^{\frac{1}{2}}\lambda_j(\theta^*) +n^{-\frac{1}{2}}\right) \frac{  1}{ n^2\lambda_j^2(\theta)}\\
	&\lesssim \sum_{j=1}^n \left( \frac{ n^{\frac{1}{2}}}{ n^2\lambda_j(\theta)}+\frac{  n^{-\frac{1}{2}}}{ n^2\lambda_j^2(\theta)}\right) \stackrel{\eqref{sumsqrtn}}{=} O(1).
	\end{align*}
	The expectation of the first term of \eqref{eq:Hn} is bounded.
	
	Similarly to the proof of the second term of $G_n=\sum_{i=1}^n|\alpha_i|$, we derive the limiting behavior of  the second term of \eqref{eq:Hn} by dividing $\mathcal{D}_n$ into two {sub-regions:
		\begin{align*}
		E_*\sum_{j=1}^n  & \sup_{x\in \mathcal{D}_n} |h''_j(x)| \cdot 2n|R_j^2(\hat{\sigma}_\varepsilon^2 - \sigma_\varepsilon^2)^2|
		 \le E_*\sum_{j=1}^n \sup_{x>\sigma_\varepsilon^2/2} |h''_j(x)| \cdot 2n R_j^2(\hat{\sigma}_\varepsilon^2 - \sigma_\varepsilon^2)^2\\
		&\quad+ E_*\sum_{j=1}^n \sup_{x\le \sigma_\varepsilon^2/2} |h''_j(x)| \cdot 2n R_j^2(\hat{\sigma}_\varepsilon^2 - \sigma_\varepsilon^2)^2 {1}_{\left\{ \frac{\sigma_\varepsilon^2}{2}\le  |\hat{\sigma}_\varepsilon^2 - \sigma_\varepsilon^2| \le {\sigma_\varepsilon^2} \right\} }
		\end{align*}
		By} Lemma \ref{lem:Nconv} and \eqref{eq:ER4}, for $j \in \{1,2,\ldots,n\}$,
	\begin{align*}
	2E_* R_j^2(\hat{\sigma}^2_\varepsilon - {\sigma}^2_\varepsilon)^2 &  \le E_* R_j^4 + E_* (\hat{\sigma}^2_\varepsilon - {\sigma}^2_\varepsilon)^4  \lesssim \lambda_j^2(\theta^*) + n^{-2}.
	\end{align*}
	Then, using \eqref{eq:likeli_tau_h},
	\begin{align*} 
	E_*\sum_{j=1}^n \sup_{x>\sigma_\varepsilon^2/2} |h''_j(x)|  2n R_j^2(\hat{\sigma}_\varepsilon^2 - \sigma_\varepsilon^2)^2 & \lesssim \sum_{j=1}^n   \frac{1}{n^2\lambda_j^2(\theta,\sigma_\varepsilon^2/2)}  n\left( \lambda_j^2(\theta^*) + n^{-2} \right)\\
	& \lesssim \sum_{j=1}^n   \frac{1}{n} +  \frac{1}{n} \sum_{j=1}^n   \frac{1}{n^2\lambda_j^2(\theta)} \stackrel{\eqref{sumsqrtn}}{=} O(1).
	\end{align*}
	Noted that when  $\frac{\sigma_\varepsilon^2}{2}\le  |\hat{\sigma}_\varepsilon^2 - \sigma_\varepsilon^2| \le {\sigma_\varepsilon^2}$, we have $(\hat{\sigma}_\varepsilon^2 - \sigma_\varepsilon^2)^2 \le \sigma_\varepsilon^4$ and $\sum_{j=1}^n R_j^2 = 2n\hat{\sigma}_\varepsilon^2\le 4n\sigma_\varepsilon^2$. Then,
	\begin{align*}
	& E_*\sum_{j=1}^n \sup_{x\le \sigma_\varepsilon^2/2} |h''_j(x)| \cdot 2n R_j^2(\hat{\sigma}_\varepsilon^2 - \sigma_\varepsilon^2)^2 {1}_{\left\{ \frac{\sigma_\varepsilon^2}{2}\le  |\hat{\sigma}_\varepsilon^2 - \sigma_\varepsilon^2| \le {\sigma_\varepsilon^2} \right\} }\\
	&\quad \stackrel{\eqref{eq:likeli_0}}{\le}  E_*\sum_{j=1}^n n^2 \cdot 2 nR_j^2 (\hat{\sigma}_\varepsilon^2 - \sigma_\varepsilon^2)^2 {1}_{\left\{ \frac{\sigma_\varepsilon^2}{2}\le  |\hat{\sigma}_\varepsilon^2 - \sigma_\varepsilon^2| \le {\sigma_\varepsilon^2} \right\} }\\
	&\quad =  2n^3E_*\left[ \left(\sum_{j=1}^n R_j^2\right) (\hat{\sigma}_\varepsilon^2 - \sigma_\varepsilon^2)^2 {1}_{\left\{ \frac{\sigma_\varepsilon^2}{2}\le  |\hat{\sigma}_\varepsilon^2 - \sigma_\varepsilon^2| \le {\sigma_\varepsilon^2} \right\} } \right] \\
	&\quad \le    2n^3E_*\left[ 4n{\sigma}_\varepsilon^2 \sigma_\varepsilon^4 {1}_{\left\{ \frac{\sigma_\varepsilon^2}{2}\le  |\hat{\sigma}_\varepsilon^2 - \sigma_\varepsilon^2| \le {\sigma_\varepsilon^2} \right\} } \right]\\
	&\quad= 8n^4 \sigma_\varepsilon^6 P_*{\left( |\hat{\sigma}_\varepsilon^2 - \sigma_\varepsilon^2| \ge \frac{\sigma_\varepsilon^2}{2} \right)} \stackrel{\text{Lemma }\ref{lem:Nconv}}{=}O(1).
	\end{align*}
	Thus, $H_n = O_{P_*}(1)$.
	Because equations \eqref{sumsqrtn}, \eqref{eq:likeli_tau} and \eqref{eq:likeli_0} all hold uniformly for $\theta\in \Theta'$, we have $\sup_{\theta\in \Theta'}G_n = O_{P_*}(1)$ and  $\sup_{\theta\in \Theta'}H_n = O_{P_*}(1)$.
\end{proof}

The following lemma states a property of the jump components {that} will be used in Lemma \ref{lem:randomlikeli} {below}.
\begin{lem}\label{lem:levysamemean} 
	Let $g$ and $h$ be known deterministic functions such that the  expectation below is finite. Then, for any $a,b,c,d\in \{1,2,\ldots,n\}$, $a<b<c<d$,  we have
	$$E_*\left[ {g(\Delta_a J, \Delta_b J,\Delta_c J, \Delta_d J)}{h([J])}\right] = E_*\left[ {g(\Delta_1 J, \Delta_2 J,\Delta_3 J, \Delta_4 J)}{h([J])}\right]. $$
\end{lem}
\begin{proof}
	Let $A=(\Delta_a J, \Delta_b J,\Delta_c J, \Delta_d J, [J])$, and $B=(\Delta_1 J, \Delta_2 J, \Delta_3 J, \Delta_4 J, [J])$. It is sufficient to show that $A$ and $B$ have the same distribution, i.e. their characteristic functions are the same. Denote $\mu$ as the jump random measure corresponding to the process $J$, and $\tilde{\mu} = \mu-\bar{\mu}$ as the compensated random measure, where $\bar{\mu}(dx, dt) := \nu(dx)dt$. Then,
	$$ \Delta_a J = \int_{t_{a-1}}^{t_a} \int_{\mathbb{R}} y \tilde{\mu}(ds\ dy), \text{ and } [J] = \int_{0}^{1} \int_{\mathbb{R}} y^2 {\mu}(ds\ dy).$$
	Let $U = (0,t_{a-1})\cup(t_a,t_{b-1})\cup(t_b,t_{c-1})\cup(t_c,t_{d-1})\cup(t_d,1)$. The characteristic function of $A$ is
	\begin{align*}
	E_* e^{i\textbf{t}A} & = E_* e^{it_1 \Delta_a J+it_2 \Delta_b J+it_3 \Delta_c J+it_4 \Delta_d J+it_5 [J]} \\
	& = E_* \exp\left\{ \int_{U} \int_{\mathbb{R}} it_5 y^2 {\mu}(ds\ dy)\right.\\
	& \quad  \left. + \sum_{(k,m)\in\{(a,1),(b,2),(c,3),(d,4)\}} \int_{t_{k-1}}^{t_k} \int_{\mathbb{R}} it_m y\tilde{\mu}(ds\ dy)+it_5y^2 {\mu}(ds\ dy) \right\} \\
	& = \exp\left\{ \left(1-\frac{4}{n}\right)\int_{\mathbb{R}} (e^{it_5y^2}-1)\,\nu(dy) \right.\\
	& \quad \left.+ \sum_{k=1}^4 \frac{1}{n} \int_{\mathbb{R}} (e^{it_ky+it_5y^2}-1-it_k y {1}_{|y|<1})\,\nu(dy)  \right\}.
	\end{align*}
	This is exactly  the same characteristic function of $B$, because of its independence of $a,b,c$, and $d$.
\end{proof}

The following result will be needed in Lemmas \ref{lem:likeli} and \ref{lem:likelid2}. 
\begin{lem}\label{lem:randomlikeli} 
	Recalling that $\theta^\dag = \theta^*+[J]$ and $\lambda_j(\theta) = \frac{\theta}{n} + 2\sigma^2_\varepsilon \left(1-\cos\frac{j\pi}{n+1}\right)$, for $p = 2,3$, under the assumptions \ref{ass:noise}, \ref{ass:bddvar}, and \ref{ass:finite16}, we have
	$$ \left|\sum_{j=1}^n \frac{n\lambda_j(\theta^\dag) - n R_j^2}{n^p\lambda_j^p(\theta^\dag)}\right| = O_{P_*}(n^{1/4}).$$
\end{lem}
\begin{proof}
	Recall that $R_j$ follows a normal distribution when conditioning on $J$. Denote $E_J(\cdot) = E_*(\cdot|J)$. Our first step is to take the conditional expectation of the expanded square given $J$:  
	\begin{equation}\label{eq:squareorigin}
		\begin{aligned}
		E_J\left(\sum_{j=1}^n \frac{n\lambda_j(\theta^\dag) - n R_j^2}{n^p\lambda_j^p(\theta^\dag)} \right)^2 = &\sum_{j=1}^n \frac{E_J\left(n\lambda_j(\theta^\dag) - n R_j^2\right)^2}{n^{2p}\lambda_j^{2p}(\theta^\dag)}\\ &+ E_J\sum_{j\neq k} \frac{n\lambda_j(\theta^\dag) - n R_j^2}{n^p\lambda_j^p(\theta^\dag)}\frac{n\lambda_k(\theta^\dag) - n R_k^2}{n^p\lambda_k^p(\theta^\dag)}.
		\end{aligned}
	\end{equation}
	We compare \eqref{eq:squareorigin} with the following equation:
	\begin{equation}\label{eq:squareEJ}
		\begin{aligned}
		\left(\sum_{j=1}^n \frac{n\lambda_j(\theta^\dag) - n E_JR_j^2}{n^p\lambda_j^p(\theta^\dag)} \right)^2 = &\sum_{j=1}^n \frac{\left(n\lambda_j(\theta^\dag) - n E_JR_j^2\right)^2}{n^{2p}\lambda_j^{2p}(\theta^\dag)} \\ & + \sum_{j\neq k} \frac{n\lambda_j(\theta^\dag) - n E_JR_j^2}{n^p\lambda_j^p(\theta^\dag)}\frac{n\lambda_k(\theta^\dag) - n E_JR_k^2}{n^p\lambda_k^p(\theta^\dag)}.
		\end{aligned}
	\end{equation}
	By the mutually independence of $R_j$s, the second term of the right-hand side of \eqref{eq:squareorigin} is  equal to the second term of \eqref{eq:squareEJ}. Then, the absolute value of the difference between the left-hand side of \eqref{eq:squareorigin} and \eqref{eq:squareEJ} is 
	\begin{align*}
	& \left| \sum_{j=1}^n \frac{n^2 E_JR_j^4-n^2\left(E_JR_j^2\right)^2}{n^{2p}\lambda_j^{2p}(\theta^\dag)} \right| \\ & \stackrel{\eqref{eq:ER2}\eqref{eq:ER4}}{\lesssim} \sum_{j=1}^n \frac{n^2 (\lambda_j^{2}(\theta)+n^{-2})+ n^2\left(\lambda_j(\theta)+n^{-1}\right)^2}{n^{2p}\lambda_j^{2p}(\theta^\dag)}\\
	& \lesssim \sum_{j=1}^n \frac{1}{n^{2p-2}\lambda_j^{2p-2}(\theta^\dag)}+ \sum_{j=1}^n \frac{1}{n^{2p-1}\lambda_j^{2p-1}(\theta^\dag)} + \sum_{j=1}^n \frac{1}{n^{2p}\lambda_j^{2p}(\theta^\dag)} \stackrel{\eqref{sumsqrtn}}{=} O(n^{1/2}),
	\end{align*}
	whose square root is $O(n^{1/4})$. Then, to prove the result, it suffices to show 
	$$ \left|\sum_{j=1}^n \frac{n\lambda_j(\theta^\dag) - n E_JR_j^2}{n^p\lambda_j^p(\theta^\dag)} \right| = O_{P_*}(n^{1/4}).$$
	Note that 
	\begin{align*}
	& E_JR_j^2  = \lambda_j(\theta^*) + \left(\frac{\mu}{n}+\sum_{i=1}^np_{ij}\Delta_i J\right)^2 \\
	&= \lambda_j(\theta^\dag)-\frac{[J]}{n} + \sum_{i=1}^np_{ij}^2\left(\Delta_i J\right)^2 + \sum_{i\neq k}p_{ij}p_{kj}\Delta_i J\Delta_k J + \frac{\mu^2}{n^2} +\frac{2\mu}{n}\sum_{i=1}^np_{ij}\Delta_i J.
	\end{align*}
	We divide the absolute value into five terms:
	\begin{align}
	& \left|\sum_{j=1}^n \frac{n\lambda_j(\theta^\dag) - n E_JR_j^2}{n^p\lambda_j^p(\theta^\dag)} \right| \nonumber\\
	& \le \left| \sum_{j=1}^n\frac{[J] - \sum_{i=1}^n\left( \Delta_i J\right)^2}{n^p\lambda_j^p(\theta^\dag)}\right| + \left|{\sum_{i=1}^n\left( \Delta_i J\right)^2\sum_{j=1}^n \frac{1 - np_{ij}^2}{n^p\lambda_j^p(\theta^\dag)}}\right| \nonumber\\
	&\quad  + \left|\sum_{j=1}^n\frac{ n\sum_{i\neq k}p_{ij}p_{kj}\Delta_i J\Delta_k J }{n^p\lambda_j^p(\theta^\dag)} \right|+  \left| \sum_{j=1}^n\frac{ \frac{\mu^2}{n} }{n^p\lambda_j^p(\theta^\dag)}\right| +  \left| \sum_{j=1}^n\frac{ 2\mu \sum_{i=1}^n p_{ij}\Delta_i J }{n^p\lambda_j^p(\theta^\dag)}\right|.\label{eq:likelifiveterm}
	\end{align}
	We will show that each term is $O_{P_*}(n^{1/4})$. 
	By Lemma \ref{realized0}, i.e. $\left|[J] - \sum_{i=1}^n\left( \Delta_i J\right)^2\right| = O_{P_*}(n^{-1/2})$, and $\lambda_j^p(\theta^\dag)\ge \lambda_j^p(\theta^*)$, the first term is such that:
	\begin{align*}
	\left| \sum_{j=1}^n\frac{[J] - \sum_{i=1}^n\left( \Delta_i J\right)^2}{n^p\lambda_j^p(\theta^\dag)}\right| & \le O_{P_*}(n^{-1/2})\sum_{j=1}^n\frac{1 }{n^p\lambda_j^p(\theta^*)} \stackrel{\eqref{sumsqrtn}}{=} O_{P_*}(1).
	\end{align*}
	Similarly, for the fourth term of \eqref{eq:likelifiveterm},
	\begin{align*}
	\left| \sum_{j=1}^n\frac{ \frac{\mu^2}{n} }{n^p\lambda_j^p(\theta^\dag)}\right| &\stackrel{\eqref{sumsqrtn}}{=} O_{P_*}(n^{-1}n^{1/2}) =  O_{P_*}(n^{-1/2}).
	\end{align*}
	
	For the fifth term of \eqref{eq:likelifiveterm}, by \eqref{eq:JPropkth} and $|p_{ij}| \le \sqrt{2(n+1)^{-1}}$,
	\begin{align*}
	E_*\left| \sum_{j=1}^n\frac{ 2\mu \sum_{i=1}^n p_{ij}\Delta_i J }{n^p\lambda_j^p(\theta^\dag)}\right| & \le  E_*\sum_{j=1}^n\frac{ 2\mu \sum_{i=1}^n |p_{ij}|\left|\Delta_i J\right|  }{n^p\lambda_j^p(\theta^\dag)} \\
	& \lesssim  \sum_{j=1}^n\frac{ \sum_{i=1}^n \sqrt{2(n+1)^{-1}} E_*|\Delta_j J| }{n^p\lambda_j^p(\theta^*)}  \lesssim \sum_{j=1}^n\frac{ n^{-1/2} }{n^p\lambda_j^p(\theta^*)}, 
	\end{align*}
	which is bounded by \eqref{sumsqrtn}.
	
	For the second term of \eqref{eq:likelifiveterm}, we square it to avoid analyzing the absolute value:
		\begin{align}
		& \left( \sum_{i=1}^n\left( \Delta_i J\right)^2\sum_{j=1}^n \frac{1 - np_{ij}^2}{n^p\lambda_j^p(\theta^\dag)}\right)^2=\sum_{i=1}^n\left( \Delta_i J\right)^4\left(\sum_{j=1}^n \frac{1 - np_{ij}^2}{n^p\lambda_j^p(\theta^\dag)}\right)^2 \nonumber\\
		&\quad\quad\quad\quad + \sum_{i\neq k} \left( \Delta_i J\right)^2\left( \Delta_k J\right)^2 \left(\sum_{j=1}^n \frac{1 - np_{ij}^2}{n^p\lambda_j^p(\theta^\dag)}\right)\left(\sum_{j=1}^n \frac{1 - np_{kj}^2}{n^p\lambda_j^p(\theta^\dag)}\right).\label{eq:squareexpan}
		\end{align}
		Next,  applying Lemma \ref{lem:levysamemean} with {$g(x,y,u,v) = x^4$} and $h(x) = \left(\sum_{j=1}^n \frac{1 - np_{ij}^2}{n^p\lambda_j^p(\theta^*+ x)}\right)^2$, the expectation of the first {term} of \eqref{eq:squareexpan} is 
	\begin{align*}
	E_*\left[ \sum_{i=1}^n\left( \Delta_i J\right)^4\left(\sum_{j=1}^n \frac{1 - np_{ij}^2}{n^p\lambda_j^p(\theta^\dag)}\right)^2 \right]
	= & E_*\left[ \left( \Delta_1 J\right)^4 \sum_{i=1}^n\left(\sum_{j=1}^n \frac{1 - np_{ij}^2}{n^p\lambda_j^p(\theta^\dag)}\right)^2 \right].
	\end{align*}
	By \eqref{eq:JPropkth}, to {achieve} the convergence rate $n^{1/2}$, we need to prove
	\begin{align}\label{eq:sumsumpij}
	{\underset{\omega\in\Omega}{{\rm ess\,sup}}}\sum_{i=1}^n\left(\sum_{j=1}^n \frac{1 - np_{ij}^2}{n^p\lambda_j^p(\theta^\dag)}\right)^2 = O(n^{3/2}).
	\end{align}
	In fact, we can expand the square of the left-hand side, and then interchange the summations. The expression $\sum_{i=1}^n\left(\sum_{j=1}^n \frac{1 - np_{ij}^2}{n^p\lambda_j^p(\theta^\dag)}\right)^2$ can be written as
	\begin{align*}
	 & \sum_{i=1}^n\sum_{j=1}^n \frac{1 - 2np_{ij}^2+ n^2p_{ij}^4}{n^{2p}\lambda_j^{2p}(\theta^\dag)} + \sum_{i=1}^n\sum_{j\neq k} \frac{1 - np_{ij}^2 - np_{ik}^2 +  n^2p_{ij}^2p_{ik}^2}{n^{2p}\lambda_j^{p}(\theta^\dag)\lambda_k^{p}(\theta^\dag)}\\
	& = \sum_{j=1}^n \frac{\sum_{i=1}^n(1 - 2np_{ij}^2+ n^2p_{ij}^4)}{n^{2p}\lambda_j^{2p}(\theta^\dag)} + \sum_{j\neq k} \frac{\sum_{i=1}^n(1 - np_{ij}^2 - np_{ik}^2 +  n^2p_{ij}^2p_{ik}^2)}{n^{2p}\lambda_j^{p}(\theta^\dag)\lambda_k^{p}(\theta^\dag)}.
	\end{align*}
	By the orthogonality of matrix $P_n$, we have $\sum_{i=1}^n p_{ij}^2 = 1$, $\sum_{i=1}^n p_{ij}^4 = O(n^{-1})${,} and $\sum_{i=1}^n p_{ij}^2p_{ik}^2 = \frac{1}{n+1}$ (see  Appendix \ref{AddProfApend} for the detailed {derivations}). Then, since {$\theta^*\le \theta^\dag$,
		\begin{align*}
		\sum_{j=1}^n  \frac{ \left|\sum_{i=1}^n(1 - 2np_{ij}^2+ n^2p_{ij}^4)\right| }{n^{2p}\lambda_j^{2p}(\theta^\dag)}&\leq{}
		\sum_{j=1}^n  \frac{n^2 \sum_{i=1}^{n}p_{ij}^4 }{n^{2p}\lambda_j^{2p}(\theta^*)}  \lesssim  \sum_{j=1}^n\frac{n}{n^{2p}\lambda_j^{2p}(\theta^*)} \stackrel{\eqref{sumsqrtn}}{=} O(n^{\frac{3}{2}}),
		\end{align*}
		\begin{align*}
		\sum_{j\neq k} \frac{ \left| \sum_{i=1}^n(1 - np_{ij}^2 - np_{ik}^2 +  n^2p_{ij}^2p_{ik}^2)\right| }{n^{2p}\lambda_j^{p}(\theta^\dag)\lambda_k^{p}(\theta^\dag)}  \leq{} \sum_{j\neq k} \frac{ \left|-n+ n^2/(n+1)\right| }{n^{2p}\lambda_j^{p}(\theta^*)\lambda_k^{p}(\theta^*)}\\
		\lesssim \sum_{j\neq k} \frac{1}{n^{2p}\lambda_j^{p}(\theta^*)\lambda_k^{p}(\theta^*)} \le \left( \sum_{j=1}^n \frac{1}{n^{p}\lambda_j^{p}(\theta^*)}\right)^2 \stackrel{\eqref{sumsqrtn}}{=} O(n).
		\end{align*}
		These imply} \eqref{eq:sumsumpij}, and thus the first term of \eqref{eq:squareexpan} is $O_{P_*}(n^{1/2})$. \\
	For the second component of \eqref{eq:squareexpan}, applying Lemma \ref{lem:levysamemean} with $g(x,y,u,v) = x^2 y^2$ and $h(x) =  \left(\sum_{j=1}^n \frac{1 - np_{ij}^2}{n^p\lambda_j^p(\theta+x)}\right)\left(\sum_{j=1}^n \frac{1 - np_{kj}^2}{n^p\lambda_j^p(\theta+x)}\right)$, the expectation of {\eqref{eq:squareexpan}} can be written as 
	\begin{align*}
	&E_*\left[ \sum_{i\neq k} \left( \Delta_i J\right)^2\left( \Delta_j J\right)^2 \left(\sum_{j=1}^n \frac{1 - np_{ij}^2}{n^p\lambda_j^p(\theta^\dag)}\right)\left(\sum_{j=1}^n \frac{1 - np_{kj}^2}{n^p\lambda_j^p(\theta^\dag)}\right)\right] \\ 
	& \quad =E_*\left[ \left( \Delta_1 J\right)^2\left( \Delta_2 J\right)^2 \sum_{i\neq k} \left(\sum_{j=1}^n \frac{1 - np_{ij}^2}{n^p\lambda_j^p(\theta^\dag)}\right)\left(\sum_{j=1}^n \frac{1 - np_{kj}^2}{n^p\lambda_j^p(\theta^\dag)}\right) \right]\\
	& \quad =E_*\left\{ \left( \Delta_1 J\right)^2\left( \Delta_2 J\right)^2 \left[ \left(\sum_{i=1}^n\sum_{j=1}^n \frac{1 - np_{ij}^2}{n^p\lambda_j^p(\theta^\dag)}\right)^2 - \sum_{i=1}^n \left(\sum_{j=1}^n \frac{1 - np_{ij}^2}{n^p\lambda_j^p(\theta^\dag)}\right)^2 \right] \right\}
	\end{align*}
	By \eqref{eq:PropOfPn}, i.e. $\sum_{i=1}^n p_{ij}^2 = 1$ ,
	\begin{align*}
	\sum_{i=1}^n\sum_{j=1}^n \frac{1 - np_{ij}^2}{n^p\lambda_j^p(\theta^\dag)} = \sum_{j=1}^n \frac{\sum_{i=1}^n (1 - np_{ij}^2)}{n^p\lambda_j^p(\theta^\dag)} = 0{,}
	\end{align*}
	which combined with \eqref{eq:sumsumpij} and \eqref{eq:JPropkth} implies that the second component \eqref{eq:squareexpan} is $O_{P}(n^{1/2})$. This finishes the proof to bound the second term of \eqref{eq:likelifiveterm}.
	
	To analyze the third term of \eqref{eq:likelifiveterm}, consider 
	\begin{align}\label{eq:squareexpan2}
	\left( \sum_{j=1}^n\frac{n\sum_{i\neq k} p_{ij}p_{kj}\Delta_i J\Delta_k J }{n^p\lambda_j^p(\theta^\dag)} \right)^2 &= \sum_{j=1}^n \frac{n^2\Sigma^2(j)}{n^{2p}\lambda_j^{2p}(\theta^\dag)} + \sum_{s\neq t} \frac{n^2\Sigma(s)\Sigma(t)}{n^{2p}\lambda_s^{p}(\theta^\dag)\lambda_t^{p}(\theta^\dag)}.
	\end{align}
	where $\Sigma(j) = \sum_{i\neq k} p_{ij}p_{kj}\Delta_i J\Delta_k J$. We will show that the two components of the right-hand side of \eqref{eq:squareexpan2} are both $O_{P_*}(n^{1/2})$. For the the first component of \eqref{eq:squareexpan2}, first note that,
	since $E_*\Delta_i J = 0$, the expectation of the numerator is bounded since
	$$ E_*n^2\Sigma^2(j) = n^2\sum_{i\neq k} p_{ij}^2p_{kj}^2E_*[(\Delta_i J)^2] E_*[(\Delta_k J)^2] \stackrel{\eqref{eq:JPropkth}}{\lesssim} n^2\cdot n^{-1}n^{-1} \left(\sum_{i=1}^np_{ij}^2\right)^2.$$
	Then, for the first component of \eqref{eq:squareexpan2}, we have
	\begin{align*}
	E_*\sum_{j=1}^n \frac{n^2\Sigma^2(j)}{n^{2p}\lambda_j^{2p}(\theta^\dag)} \le E_*\sum_{j=1}^n\frac{n^2\Sigma^2(j)}{n^{2p}\lambda_j^{2p}(\theta^*)} = \sum_{j=1}^n\frac{n^2E_*\Sigma^2(j)}{n^{2p}\lambda_j^{2p}(\theta^*)} \lesssim \sum_{j=1}^n\frac{1}{n^{2p}\lambda_j^{2p}(\theta^*)}.
	\end{align*}
	The last term is $O(n^{1/2})$ by \eqref{sumsqrtn}.
	For the second component of  \eqref{eq:squareexpan2}, we first look at the numerator
	$$\Sigma(s)\Sigma(t)=\left(\sum_{i\neq k} p_{is}p_{ks}\Delta_i J\Delta_k J\right) \left(\sum_{u\neq v} p_{ut}p_{vt}\Delta_u J\Delta_v J\right).$$ 
	Expanding the expression results in three different types of terms:
	\begin{align*}
	& p_{is}p_{ks}p_{ut}p_{vt}\Delta_i J\Delta_k J\Delta_u J\Delta_v J, & \text{ for } i<k<u<v,&\\
	& p_{is}p_{ks}p_{it}p_{ut}(\Delta_i J)^2\Delta_k J\Delta_u J, & \text{ for } i<k<u, v=i,&\\
	& p_{is}p_{ks}p_{it}p_{kt}(\Delta_i J)^2(\Delta_k J)^2, & \text{ for } i<k, u=i, v=k.&
	\end{align*}
	For these different terms, we apply Lemma \ref{lem:levysamemean} with corresponding $g$ and $h$ functions: 
	\begin{align*}
	& {g(x,y,u,v)} = xyuv,\quad h(x) =  \frac{ p_{is}p_{ks}p_{ut}p_{vt}}{n^{2p}\lambda_s^{p}(\theta^*+x)\lambda_t^{p}(\theta^*+x)},\\
	& {g(x,y,u,v)} = x^2yu,\quad h(x) =  \frac{  p_{is}p_{it}p_{ks}p_{ut}}{n^{2p}\lambda_s^{p}(\theta^*+x)\lambda_t^{p}(\theta^*+x)},\\
	& {g(x,y,u,v)} = x^2y^2,\quad h(x) =  \frac{ p_{is}p_{ks}p_{it}p_{kt}}{n^{2p}\lambda_s^{p}(\theta^*+x)\lambda_t^{p}(\theta^*+x)}.
	\end{align*}
	Then, after combining the same terms, the expectation of the second component of \eqref{eq:squareexpan2} is the summation over $s$ and $t$ from $1$ to $n$ ($s\neq t$) of the following:
	\begin{align*}
	& E_*\Delta_1 J\Delta_2 J\Delta_3 J\Delta_4 J  \frac{\sum_{i\neq k \neq u\neq v} p_{is}p_{ks}p_{ut}p_{vt}}{n^{2p}\lambda_s^{p}(\theta^\dag)\lambda_t^{p}(\theta^\dag)},\\
	& 2 E_*(\Delta_1 J)^2\Delta_2 J\Delta_3 J  \frac{\sum_{i\neq k \neq u} p_{is}p_{it}p_{ks}p_{ut}}{n^{2p}\lambda_s^{p}(\theta^\dag)\lambda_t^{p}(\theta^\dag)},\\
	& E_*(\Delta_1 J)^2(\Delta_2 J)^2  \frac{\sum_{i\neq k} p_{is}p_{it}p_{ks}p_{kt}}{n^{2p}\lambda_s^{p}(\theta^\dag)\lambda_t^{p}(\theta^\dag)}.
	\end{align*}
	Next, taking absolute value of each component above, which allows us to shrink the denominator to the fixed value $ n^{2p}\lambda_s^{p}(\theta^*)\lambda_t^{p}(\theta^*) $, and noting that 
	\begin{align}\label{eq:PropOfppp}
	\left| \sum_{i\neq k\neq u} p_{is}p_{it}p_{ks}p_{ut}\right| & = O(1), \quad \text{ and } \quad
	\left| \sum_{i\neq k} p_{is}p_{it}p_{ks}p_{kt}\right| {=} \frac{1}{n+1},
	\end{align}
	(see more details in Appendix \ref{app:pn}), we have the following upper bounds:
	\begin{align}
	& E_*|\Delta_1 J\Delta_2 J\Delta_3 J\Delta_4 J | \frac{\left| \sum_{i\neq k \neq u\neq v} p_{is}p_{ks}p_{ut}p_{vt}\right|}{n^{2p}\lambda_s^{p}(\theta^*)\lambda_t^{p}(\theta^*)}, \nonumber \\
	& 2 E_*|(\Delta_1 J)^2\Delta_2 J\Delta_3 J | \frac{\left|\sum_{i\neq k \neq u} p_{is}p_{it}p_{ks}p_{ut}\right|}{n^{2p}\lambda_s^{p}(\theta^*)\lambda_t^{p}(\theta^*)} \stackrel{\eqref{eq:PropOfppp}}\lesssim  \frac{E_*|(\Delta_1 J)^2\Delta_2 J\Delta_3 J |}{n^{2p}\lambda_s^{p}(\theta^*)\lambda_t^{p}(\theta^*)} \nonumber \\
	& \hspace{189pt}\stackrel{\eqref{eq:JPropkth}}{\lesssim} \frac{n^{-3}}{n^{2p}\lambda_s^{p}(\theta^*)\lambda_t^{p}(\theta^*)},\label{eq:jjjppp}\\
	& E_*|(\Delta_1 J)^2(\Delta_2 J)^2 | \frac{\left|\sum_{i\neq k} p_{is}p_{it}p_{ks}p_{kt}\right|}{n^{2p}\lambda_s^{p}(\theta^*)\lambda_t^{p}(\theta^*)}\stackrel{\eqref{eq:PropOfppp}}{=} \frac{1}{n+1}\frac{E_*|(\Delta_1 J)^2(\Delta_2 J)^2 |}{n^{2p}\lambda_s^{p}(\theta^*)\lambda_t^{p}(\theta^*)}\nonumber\\
	& \hspace{167pt} \stackrel{\eqref{eq:JPropkth}}{\lesssim}\frac{n^{-3}}{n^{2p}\lambda_s^{p}(\theta^*)\lambda_t^{p}(\theta^*)} .\label{eq:jjpp}
	\end{align}
	For \eqref{eq:jjjppp} and  \eqref{eq:jjpp}, we take summation over $s,t $ from $1$ to $n$, and then multiply it by $n^2$. The resulting expressions are both bounded, which is ensured by \eqref{sumsqrtn}. Combined with \eqref{eq:JPropkth}, to show that the second term of \eqref{eq:squareexpan2} is $O(n^{1/2})$, it suffices to show that 
	\begin{align}\label{sigmasigma1+2}
	\sum_{s=1}^n\sum_{t=1}^n \frac{\left|\sum_{i\neq k \neq u\neq v} p_{is}p_{ks}p_{ut}p_{vt}\right|}{n^{2p}\lambda_s^{p}(\theta^*)\lambda_t^{p}(\theta^*)} = O(n^{1/2+4-2})= O(n^{5/2}).
	\end{align}
	{Indeed, let us start by noting that, since $\sum_{i=1}^n p_{ij}^2 = 1$,}
	\begin{align*}
	\left(\sum_i p_{is}\right)^2\left(\sum_u p_{ut}\right)^2 {=} & \left(1+\sum_{i\neq k} p_{is}p_{ks}\right)\left(1+\sum_{u\neq v} p_{ut}p_{vt}\right) \\
	=&  \sum_{u\neq v} (p_{ut}p_{vt}+p_{us}p_{vs}) + \sum_{u\neq v} p_{us}p_{vs}p_{ut}p_{vt}+ \\
	& + 4\sum_{i\neq u\neq v} p_{is}p_{us}p_{ut}p_{vt} + \sum_{i\neq k\neq u\neq v} p_{is}p_{ks}p_{ut}p_{vt} +1.
	\end{align*}
	By \eqref{eq:PropOfppp}, the second term and the third term are both $O(1)$. Then,
	$$ \left|\sum_{i\neq k \neq u\neq v} p_{is}p_{ks}p_{ut}p_{vt}\right| \lesssim \left(\sum_i p_{is}\right)^2\left(\sum_u p_{ut}\right)^2 +1 + 2\left|\sum_{u\neq v} p_{ut}p_{vt}+\sum_{u\neq v} p_{us}p_{vs}\right|.$$
	By \eqref{eq:sumpij} of Appendix \ref{app:pn}, i.e. $\left(\sum_{i=1}^n p_{is}\right)^2 \lesssim n/s^2$, and $\lambda_j(\theta)\ge \theta/n$, we have
	\begin{align*}
	\sum_{s=1}^n \frac{\left(\sum_i p_{is}\right)^2}{n^{p}\lambda_s^{p}(\theta^*)}  \lesssim \sum_{s=1}^n \frac{n/s^2}{n^{p}(\theta/n)^p} \lesssim  n + \sum_{s=2}^n \frac{n}{s^2} \le n+ \int_1^n \frac{n}{s^2}\,ds = O(n).
	\end{align*}
	Then,
	\begin{align}\label{sigmasigma1}
	\sum_{s=1}^n\sum_{t=1}^n \frac{\left(\sum_i p_{is}\right)^2\left(\sum_u p_{ut}\right)^2}{n^{2p}\lambda_s^{p}(\theta^*)\lambda_t^{p}(\theta^*)} = \left[ \sum_{s=1}^n \frac{\left(\sum_i p_{is}\right)^2}{n^{p}\lambda_s^{p}(\theta^*)} \right]^2 =O(n^2).
	\end{align}
	Since $|p_{ut}| = O(n^{-1/2})$,
	\begin{align}\label{sigmasigma2}
	\sum_{s=1}^n\sum_{t=1}^n \frac{\left|\sum_{u\neq v} p_{ut}p_{vt}\right|}{n^{2p}\lambda_s^{p}(\theta^*)\lambda_t^{p}(\theta^*)} \lesssim \sum_{s=1}^n\sum_{t=1}^n  \frac{n^2\cdot \frac{1}{n}}{n^{2p}\lambda_s^{p}(\theta^*)\lambda_t^{p}(\theta^*)} =n\left(\sum_{s=1}^n \frac{1}{n^{p}\lambda_s^{p}(\theta^*)} \right)^2,
	\end{align}
	which is $O(n^2)$ by \eqref{sumsqrtn}.
	Then, combining \eqref{sigmasigma1} and \eqref{sigmasigma2}, we proved \eqref{sigmasigma1+2}, and thus, the third term of \eqref{eq:likelifiveterm} is $O_{P_*}(n^{1/2})$. This completes the proof of the Lemma.
\end{proof}

{In the following two lemmas, we establish some needed} asymptotic properties of the misspecified likelihood function under all the assumptions in \S~\ref{sec:setup}.
\begin{lem}\label{lem:likeli} 
	Let $\Gamma(\theta,\zeta):= \frac{(\sqrt{\theta} - \sqrt{\zeta})^2}{2\sqrt{\zeta}\sigma_\varepsilon},$ recalling that $\sigma_\varepsilon$ is the standard deviation of the noise. Then, 
	$$\sup_{\zeta \in \Theta'}|n^{-\frac{1}{2}} (\tilde{l}_n(\theta^\dag) - \tilde{l}_n(\zeta)) - \Gamma(\theta^\dag,\zeta)| = o_{{P_*}}(1),$$
	and
	$$ \left|n^{-\frac{1}{4}} \dot{\tilde{l}}_n(\theta^\dag)\right| = O_{P_*}(1).$$
\end{lem}

\begin{proof}
	The first statement can be proved as follows. For any $\zeta \in \Theta'$, there exist a $\zeta^*\in \Theta$ such that $\zeta = \zeta^*+[J]$ {and, noting that $\theta^\dag = \theta^*+[J]$, we have} $\zeta - \theta^\dag = \zeta^* - \theta^*$. Hence, by the boundedness of $\Theta$, it is enough to prove {that,} for any fixed $\delta>0$,
	$$ \sup_{\zeta\in\Theta':|\zeta-\theta^\dag|<\delta} |n^{-\frac{1}{2}} (\tilde{l}_n(\theta^\dag) - \tilde{l}_n(\zeta)) - \Gamma(\theta^\dag,\zeta)| \stackrel{{P_*}}{\rightarrow} 0.$$
	
	By Lemma \ref{lem:likeliep}, and applying the mean value theorem to function $l(\cdot)-\tilde{l}(\cdot)$, we can conclude that 
	\begin{equation}\label{eq:lemlikeli1}
	\sup_{\theta,\zeta\in\Theta': \theta\neq \zeta}\left|\frac{l_n(\theta)-  {l}_n(\zeta)}{\theta-\zeta}-\frac{\tilde{l}_n(\theta)-  \tilde{l}_n(\zeta)}{\theta-\zeta} \right| \le  \sup_{\theta\in\Theta'} \left| \frac{dl_n(\theta)}{d\theta} - \frac{d\tilde{l}_n(\theta)}{d\theta} \right| = O_{P_*}(1).
	\end{equation}
	Then, for every $\zeta\in\Theta'$ such that $|\zeta-\theta^\dag|<\delta$, 
	\begin{align*}
	n^{-\frac{1}{2}} (\tilde{l}_n(\zeta) - \tilde{l}_n(\theta^\dag) )
	&= n^{-\frac{1}{2}}({l}_n(\zeta) - {l}_n(\theta^\dag) ) + O_{P_*}(n^{-\frac{1}{2}}).
	\end{align*}
	Hence, it suffices to show the statement for $ l_n(\zeta)-l_n(\theta^\dag)$. 
	Next, recalling (\ref{eq:loglikeli2}),
	the difference of the log likelihoods can be rewritten as 
	\begin{align*}
	& n^{-\frac{1}{2}}({l}_n(\theta^\dag) - {l}_n(\zeta) )   = n^{-\frac{1}{2}} \sum_{j=1}^n \left[ \frac{R_j^2}{\lambda_j(\theta^\dag)} \left( \frac{\lambda_j(\theta^\dag)}{\lambda_j(\zeta)}-1\right) -\log \frac{\lambda_j(\theta^\dag)}{\lambda_j(\zeta)} \right] \\
	& =n^{-\frac{1}{2}} \sum_{j=1}^n \left[ \frac{\lambda_j(\theta^\dag)}{\lambda_j(\zeta)}-1 -\log \frac{\lambda_j(\theta^\dag)}{\lambda_j(\zeta)}\right] + n^{-\frac{1}{2}} \sum_{j=1}^n \frac{nR_j^2-n\lambda_j(\theta^\dag)}{n^2\lambda_j^2(\theta^\dag)} \frac{(\theta^\dag-\zeta)\lambda_j(\theta^\dag)}{\lambda_j(\zeta)}.
	\end{align*}
	The second term is $o_{P_*}(1)$ uniformly for $\zeta$ such that $|\zeta-\theta^\dag|<\delta$, which follows from Lemma \ref{lem:randomlikeli} with $p=2$, and
	\begin{align}\label{eq:lambda/lambda}
	\frac{\lambda_j(\theta^\dag)}{\lambda_j(\zeta)} =\frac{\theta^\dag+ 2n\sigma_\varepsilon^2\left( 1-\cos\frac{j\pi}{n+1} \right)}{\zeta+ 2n\sigma_\varepsilon^2\left( 1-\cos\frac{j\pi}{n+1} \right)}\le \frac{\theta^\dag+ 2n\sigma_\varepsilon^2\left( 1-\cos\frac{j\pi}{n+1} \right)}{\delta_0+ 2n\sigma_\varepsilon^2\left( 1-\cos\frac{j\pi}{n+1} \right)}\le \frac{\theta^\dag}{\delta_0},
	\end{align}
	{where recall that $\delta_{0}>0$ is such that $\Theta'\subset (\delta_{0},\infty)$.}	 
	Then, it remains to prove the following convergence:
	\begin{align}\label{eq:unif}
	\sup_{\zeta\in\Theta':|\zeta-\theta^\dag|\le \delta}\left|n^{-\frac{1}{2}} \sum_{j=1}^n \left[ \frac{\lambda_j(\theta^\dag)}{\lambda_j(\zeta)}-1 -\log \frac{\lambda_j(\theta^\dag)}{\lambda_j(\zeta)}\right] - \Gamma(\theta^\dag,\zeta)\right| \stackrel{P_*}{\rightarrow}  0.
	\end{align}
	Based on (4.5) and (6.6) in \cite{gloter_jacod_2001_est} with $a = \theta^\dag$ and $b = \zeta$, the uniform convergence holds almost surely for $\zeta\in[1/C,C]$, where $C$ is  some constant. The result can be generalized to  $\zeta\in[\delta_0,\theta^\dag+\delta]$ when we use convergence in probability instead of almost surely convergence, because $\theta^\dag$ can be treated as a constant under measure $P_*$ (see more details in Appendix \ref{AddProfApend}).
	Hence, the first statement can be demonstrated.
	
	For the second statement, again, by Lemma \ref{lem:likeliep},
	\begin{align*}
	\left|n^{-\frac{1}{4}} \dot{\tilde{l}}_n(\theta^\dag)\right| &\le \left|n^{-\frac{1}{4}} (\dot{\tilde{l}}_n(\theta^\dag)-\dot{{l}}_n(\theta^\dag))\right|+ \left|n^{-\frac{1}{4}} \dot{{l}}_n(\theta^\dag)\right| = \left|n^{-\frac{1}{4}} \dot{{l}}_n(\theta^\dag)\right|+o_{P_*}(1).
	\end{align*}
	Then, it suffices to prove the boundedness of $\left|n^{-\frac{1}{4}} \dot{{l}}_n(\theta^\dag)\right|$, which follows directly from Lemma \ref{lem:randomlikeli} with $p=2$ since
	$$\dot{{l}}_n(\theta^\dag) = -\frac{1}{2}\sum_{i=1}^n\frac{\lambda_j(\theta^\dag)-R_j^2}{n\lambda_j^2(\theta^\dag)}.$$
	This completes the proof.
	
\end{proof}

\begin{lem}\label{lem:likelid2} 
	Let $I(\theta) = \frac{1}{8\theta^{3/2}\sigma_\varepsilon}$. Then, under all the assumptions in \S~\ref{sec:setup},
	$$ n^{-\frac{1}{2}} \ddot{\tilde{l}}_n(\theta^\dag) + I(\theta^\dag) \stackrel{{P_*}}{\rightarrow} 0,$$
	{and, for} any sequence of {nonnegative} random variables $\{\eta_n\}$ such that $\eta_n \stackrel{P_*}{\rightarrow} 0$, 
	$$  \sup_{\zeta,\zeta'\in\Theta':|\zeta-\zeta'|\le \eta_n} n^{-\frac{1}{2}} \left| \ddot{\tilde{l}}_n(\zeta)-\ddot{\tilde{l}}_n(\zeta') \right|  \stackrel{{P_*}}{\rightarrow} 0, \text{ as } n\rightarrow \infty.$$
\end{lem}
\begin{proof}
	By Lemma \ref{lem:likeliep}, it suffices to prove the statements only for $l_n$. For the first statement, we {split} $\ddot{l}_n$ into two components:
	\begin{align*}
	n^{-\frac{1}{2}} \ddot{l}_n(\theta^\dag) & = -n^{-\frac{1}{2}} \sum_{j=1}^n  \left( \frac{2nR_j^2}{n^3\lambda_j^3(\theta^\dag)} - \frac{1}{n^2\lambda_j^2(\theta^\dag)} \right)\\
	& = -2{n}^{-\frac{1}{2}} \sum_{j=1}^n \frac{nR_j^2-n\lambda_j(\theta^\dag)}{n^3\lambda_j^3(\theta^\dag)} -  n^{-\frac{1}{2}}\sum_{j=1}^n \frac{1}{n^2\lambda_j^2(\theta^\dag)}.
	\end{align*}
	The first component is $O_{P_*}(n^{-1/4})$ by Lemma \ref{lem:randomlikeli} with $p=3$. The second component converges to $-I(\theta)$, which results from (3.1)-(3.2) in \cite{lan}. This concludes the proof of the first assertion. 
	
	For the second statement, without loss of generality, let $\zeta\le \zeta'$ {and note that 
		\begin{align*}
		\left| \frac{1}{\lambda_j^2(\zeta)}- \frac{1}{\lambda_j^2(\zeta')} \right| & = |\zeta'-\zeta|\frac{\lambda_j(\zeta)+\lambda_j(\zeta')}{n\lambda_j^2(\zeta)\lambda_j^2(\zeta')}\\
		&= |\zeta'-\zeta| \left(\frac{1}{n\lambda_j(\zeta)\lambda_j^2(\zeta')} + \frac{1}{n\lambda_j^2(\zeta)\lambda_j(\zeta')}\right)\\
		&\stackrel{\zeta\le \zeta'}{\le} 
		2|\zeta'-\zeta| \frac{1}{n\lambda_j^3(\zeta)} .
		\end{align*}
		Similarly,}
	\begin{align*}
	\left| \frac{1}{\lambda_j^3(\zeta)}- \frac{1}{\lambda_j^3(\zeta')}\right|&= 
	|\zeta'-\zeta|\left( \frac{1}{n\lambda_j(\zeta)\lambda_j^3(\zeta')} +\frac{1}{n\lambda_j^2(\zeta)\lambda_j^2(\zeta')} +\frac{1}{n\lambda_j^3(\zeta)\lambda_j(\zeta')} \right)\\
	& \le 3|\zeta'-\zeta|\frac{1}{n\lambda_j^4(\zeta)}.
	\end{align*}
	Recalling that $$ \ddot{\tilde{l}}_n(\zeta) = \sum_{i=1}^n \frac{1}{2n^2\lambda_j^2(\zeta)}- \sum_{i=1}^n \frac{R_j^2}{n^2\lambda_j^3(\zeta)},$$ 
	we have
	\begin{align*}
	\left| \ddot{\tilde{l}}_n(\zeta)-\ddot{\tilde{l}}_n(\zeta') \right| & \le \sum_{j=1}^n \frac{1}{2n^2}\left| \frac{1}{\lambda_j^2(\zeta)}- \frac{1}{\lambda_j^2(\zeta')} \right|  + \sum_{j=1}^n \frac{R_j^2}{n^2}\left| \frac{1}{\lambda_j^3(\zeta)}- \frac{1}{\lambda_j^3(\zeta')}\right|\\
	& \le 3|\zeta'-\zeta| \left( \sum_{j=1}^n \frac{1}{n^3 \lambda_j^3(\zeta)}  + \sum_{j=1}^n \frac{nR_j^2}{n^4\lambda_j^4(\zeta)} \right) .
	\end{align*}
	Then, because {$\Theta' \subset (\delta_0,\infty)$,
		\begin{align}
		\sup_{|\zeta-\zeta'|\le \eta_n} n^{-\frac{1}{2}}\left| \ddot{\tilde{l}}_n(\zeta)-\ddot{\tilde{l}}_n(\zeta') \right|
		&\le  \eta_n n^{-\frac{1}{2}}\left( \sup_{\zeta\in\Theta'}\sum_{j=1}^n \frac{1}{n^3 \lambda_j^3(\zeta)}  + \sup_{\zeta\in\Theta'}\sum_{j=1}^n \frac{ n R_j^2}{n^4\lambda_j^4(\zeta)} \right) \nonumber\\
		& \le \eta_n n^{-\frac{1}{2}}\left( \sum_{j=1}^n \frac{1}{n^3 \lambda_j^3(\delta_0)}  + \sum_{j=1}^n \frac{ n R_j^2}{n^4\lambda_j^4(\delta_0)} \right) . \label{eq:simpliked2}
		\end{align}
		By} \eqref{eq:ER2},
	\begin{align*} 
	E_* \left[ n^{-\frac{1}{2}} \sum_{j=1}^n \frac{ n R_j^2}{n^4\lambda_j^4(\delta_0)} \right] & =  n^{-\frac{1}{2}} \sum_{j=1}^n \frac{ n \lambda_j(\theta^*)+O(1)}{n^4\lambda_j^4(\delta_0)}\\
	& \lesssim n^{-\frac{1}{2}} \sum_{j=1}^n \left\{ \frac{1}{n^3\lambda_j^3(\delta_0)}+ \frac{1}{n^4\lambda_j^4(\delta_0)} \right\} \stackrel{\eqref{sumsqrtn}}{=} O(1).
	\end{align*}
	Then, 
	\begin{align}\label{eq:ER2uni}
	n^{-\frac{1}{2}} \sum_{j=1}^n \frac{ n R_j^2}{n^4\lambda_j^4(\delta_0)} = O_{P_*}(1).
	\end{align}
	Applying Slutsky's theorem, \eqref{eq:ER2uni} and \eqref{sumsqrtn} to \eqref{eq:simpliked2}, the lemma can be proved.
\end{proof}

\subsection{MLE and its convergence rate}\label{sec:consistency}
In this section, {we prove that 
	\begin{equation}\label{CnstTildeTheta2}
	\left| \tilde{\theta}_n - \theta^\dag \right| = O_{P_*}(n^{-1/4}),
	\end{equation}
	which, as explained at the beginning of Appendix \ref{app:noiseproof}, implies the condition (\ref{eq:con2}) of Theorem \ref{thm:bvm}}.

Since $\tilde{\theta}_n$ be the maximum of the misspecified log likelihood function  $\tilde{l}_n$ defined in (\ref{likelihood}). We then have
\begin{align*}
-\dot{\tilde{l}}_n(\theta^\dag) &=(\tilde{\theta}_n-\theta^\dag) \int_0^1 \ddot{\tilde{l}}_n(\theta^\dag + w(\tilde{\theta}_n -\theta^\dag))\, dw\\
& =(\tilde{\theta}_n-\theta^\dag)\ddot{\tilde{l}}_n(\theta^\dag) + (\tilde{\theta}_n-\theta^\dag)\int_0^1 \left[ \ddot{\tilde{l}}_n(\theta^\dag + w(\tilde{\theta}_n -\theta^\dag)) - \ddot{\tilde{l}}_n(\theta^\dag)\right]\, dw.
\end{align*}
Rearranging the terms, we obtain the following equation:
\begin{equation}\label{eq:estfrac}
n^\frac{1}{4}\left| \tilde{\theta}_n - \theta^\dag\right| = \frac{|n^{-\frac{1}{4}}\dot{\tilde{l}}_n(\theta^\dag)|}{\left| n^{-\frac{1}{2}} \ddot{\tilde{l}}_n(\theta^\dag) + \int_0^1 n^{-\frac{1}{2}} \left[ \ddot{\tilde{l}}_n(\theta^\dag+w(\tilde{\theta}_n-\theta^\dag))-\ddot{\tilde{l}}_n(\theta^\dag) \right]\,dw \right|}.
\end{equation}
Next, we apply Theorem 1 of \cite{xiu2010} with
	\[
	Q_n(x) = -n^{-1/2}\left(\tilde{l}_n(\theta^\dag)-\tilde{l}_n(x+[J])\right), \;
	\bar{Q}_n(x) = -\Gamma\left(\theta^\dag, x+[J]\right), 
	\]
	$\text{and }\; \hat{\theta}_n = \tilde{\theta}_n-[J],$ to conclude that $\tilde{\theta}_n-[J]$ is a consistent estimator of $\theta^*$. This implies $\tilde{\theta}_n-\theta^\dag\stackrel{P_*}{\rightarrow} 0$. Indeed, Lemma \ref{lem:likeli} and the definition of $\Gamma$ yield that the conditions are satisfied since the maximum of {$ -\Gamma(\theta^\dag, x+[J])$ is  $0$ when $x=\theta^{\dag}-[J]=\theta^{*}$}. 

The consistency {stated in the paragraph above} combined with Lemma \ref{lem:likelid2} ({applied with} $\eta_n = |\tilde{\theta}_n - \theta^\dag|$) implies that the denominator of the right-hand side of \eqref{eq:estfrac} converges to some constant value. 
The numerator is $O_{P_*}(1)$ by Lemma \ref{lem:likeli}. Then, we can conclude that the convergence rate of $\tilde{\theta}_n$ to $\theta^\dag$ is $O_{P_*}(n^{-\frac{1}{4}})$, {as claimed.}

\subsection{Local Asymptotic Normality}\label{sec:lan}
The following local asymptotic normality (LAN)  is the condition \eqref{eq:con1} required in Theorem \ref{thm:bvm}.  For notational simplicity, in this section, we write $\tilde{l}_n({\theta})$ as  $\tilde{l}_{\theta}$.

\begin{thm} (Local Asymptotic Normality) Recall that we assumed  that $\kappa_n\rightarrow\kappa^\dag$ in $P_*$-probability. Assume that $\kappa^\dag$ is bounded away from zero and infinity in $P_*$-probability, and \ref{ass:noise}, \ref{ass:bddvar}, and \ref{ass:finite16} hold true. Then, for every compact
	set $K\subset \mathbb{R}$, $\tilde{l}$ satisfies
	\begin{align}\label{DEH0}
	\sup_{h\in K} \left| \frac{1}{\kappa_n}\left( \tilde{l}_{\theta^\dag+n^{-1/4}h} -  \tilde{l}_{\theta^\dag}\right) - \frac{1}{\kappa^\dag} h n^{\frac{1}{4}} I(\theta^\dag)(\tilde{\theta}_n-\theta^\dag) + \frac{1}{\kappa^\dag}\frac{1}{2}h^2 I(
	\theta^\dag)\right| = o_{P_*}(1),
	\end{align}
	where $I(\theta) = \frac{1}{8\theta^{3/2}\sigma_\varepsilon}$.
\end{thm}

\begin{proof}
	Rewrite the the left-hand side expression in (\ref{DEH0}) as
	\begin{align*}
	&\left| \frac{1}{\kappa_n}\left( \tilde{l}_{\theta^\dag+n^{-1/4}h} -  \tilde{l}_{\theta^\dag}\right) - \frac{1}{\kappa^\dag} h n^{\frac{1}{4}} I(\theta^\dag)(\tilde{\theta}_n-\theta^\dag) + \frac{1}{\kappa^\dag}\frac{1}{2}h^2 I(
	\theta^\dag)\right| \\
	& \quad\le \left| \left(\frac{1}{\kappa_n}-\frac{1}{\kappa^\dag}\right)\left( h n^{\frac{1}{4}} I(\theta^\dag)(\tilde{\theta}_n-\theta^\dag) - \frac{1}{2}h^2 I(
	\theta^\dag)\right)\right| \\
	& \qquad + \frac{1}{\kappa_n} \left| \tilde{l}_{\theta^\dag+n^{-1/4}h} - \tilde{l}_{\theta^\dag}- h n^{\frac{1}{4}} I(\theta^\dag)(\tilde{\theta}_n-\theta^\dag) + \frac{1}{2}h^2 I(
	\theta^\dag)\right|.
	\end{align*}
	By \eqref{CnstTildeTheta2} (i.e. $ \left| \tilde{\theta}_n - \theta^\dag \right| = O_{P_*}(n^{-1/4})$), the fact that $\kappa_n\rightarrow\kappa^\dag$ in $P_*$-probability, and Slutsky's Theorem, we only need to prove
	\begin{align*}
	\sup_{h\in K} \left| \tilde{l}_{\theta^\dag+n^{-1/4}h} -  \tilde{l}_{\theta^\dag} - h n^{\frac{1}{4}} I(\theta^\dag)(\tilde{\theta}_n-\theta^\dag) +\frac{1}{2}h^2 I(
	\theta^\dag)\right| = o_{P_*}(1).
	\end{align*}	
	Let us start by writing
	\begin{align} \label{eq:lantaylor}
	& \tilde{l}_{\theta^\dag+n^{-1/4}h} - \tilde{l}_{\theta^\dag} = h n^{-\frac{1}{4}}\dot{\tilde{l}}_{\theta^\dag} + \frac{1}{2}h^2 n^{-\frac{1}{2}} \ddot{\tilde{l}}_{\theta^\dag} + r_n.
	\end{align}
	We will prove that the first and second terms converge in probability to $h n^{\frac{1}{4}}  I(\theta^\dag)$ $(\tilde{\theta}_n-\theta^\dag)$ and $-\frac{1}{2}h^2 I(\theta^\dag)$, respectively, and the remainder $r_n$ goes to zero, uniformly in $h$.
	
	For the first term of \eqref{eq:lantaylor}, since $\dot{\tilde{l}}$ is continuous and differentiable w.r.t. $\theta$, by the mean value theorem, there exists $\ddot{\theta}_n$ lying on the segment which connects $\theta^\dag$ and $\tilde{\theta}_n$, such that
	\begin{align}\label{eq:LANfirstterm1}
	n^{-\frac{1}{4}}\dot{\tilde{l}}_{\theta^\dag} = n^{-\frac{1}{4}}\dot{\tilde{l}}_{\theta^\dag} - n^{-\frac{1}{4}}\dot{\tilde{l}}_{\tilde{\theta}_n}= n^{-\frac{1}{2}} \ddot{\tilde{l}}_{\ddot{\theta}_n} \cdot  n^{\frac{1}{4}}(\theta^\dag - \tilde{\theta}_n) .
	\end{align}
	The first equality holds because $\tilde{\theta}_n$ is the MLE of  $\tilde{l}_\theta$.
	Because $\tilde{\theta}_n$ converges to $\theta^\dag$ in probability and $\ddot{\theta}_n$ lies on the segment joining $\tilde{\theta}_n$ and $\theta^\dag$, we conclude that $\ddot{\theta}_n$ converges to $\theta^\dag$ in probability. 
	Then, applying Lemma \ref{lem:likelid2} with $\zeta = \ddot{\theta}_n$, $\zeta' = {\theta}^\dag$, and $\eta_n = |{\tilde{\theta}_n}- {\theta}^\dag|$,
	\begin{align}\label{eq:LANfirstterm2}
	n^{-\frac{1}{2}}\ddot{\tilde{l}}_{\ddot{\theta}_n} + I(\theta^\dag) = \left[ n^{-\frac{1}{2}}\ddot{\tilde{l}}_{\ddot{\theta}_n} -n^{-\frac{1}{2}}\ddot{\tilde{l}}_{{\theta}^\dag} \right] + \left[n^{-\frac{1}{2}}\ddot{\tilde{l}}_{{\theta}^\dag} + I\left( \theta^\dag \right) \right] = o_{P_*}(1).  
	\end{align}
	Thus, combining \eqref{eq:LANfirstterm1} and \eqref{eq:LANfirstterm2}, the first term of the right-hand side of \eqref{eq:lantaylor} can be written as,
	\begin{align*}
	hn^{-\frac{1}{4}}\dot{\tilde{l}}_{\theta^\dag} = - h  n^{\frac{1}{4}} I\left( \theta^\dag \right)(\theta^\dag - \tilde{\theta}_n)+o_{P_*}(1).
	\end{align*}
	{The} second term of \eqref{eq:lantaylor} converges to $-\frac{1}{2}h^2 I(\theta^\dag)$ by Lemma \ref{lem:likelid2}.
	
	For the reminder term of \eqref{eq:lantaylor}, $r_n$, note first that by Lemma \ref{lem:likeliep}, for all $\theta\in\Theta'$,
	$$\left| \frac{d^3 \tilde{l}_\theta}{d\theta^3} \right| \le \left|\frac{d^3 \tilde{l}_\theta}{d\theta^3}-\frac{d^3 l_\theta}{d\theta^3}\right| + \left|\frac{d^3 l_{\theta}}{d\theta^3}\right| = O_{P_*}(1) + \frac{(-1)^{k}}{2} \sum_{i=1}^{n} \left| \frac{(k-1)!}{n^k \lambda_i^k(\theta)} - \frac{k! nR_i^2}{n^{k+1} {\lambda}_i^{k+1}(\theta)}\right|.$$
	Applying \eqref{sumsqrtn} and \eqref{eq:ER2}, the last term is $O_{P_*}(\sqrt{n})$. The proof is the same as in \eqref{eq:ER2uni}. We use $E_*nR_j^2 = n\lambda_j(\theta^*) + O(1)$ and cancel out one $n\lambda_j(\theta)$ from the denominator (up to some constant) by \eqref{eq:lambda/lambda}, and then use \eqref{sumsqrtn} to obtain the rate $n^{1/2}$.
	Then, 
	\begin{align*}
	|r_n| \le \frac{1}{6}h^3 n^{-\frac{3}{4}} \sup_{\theta \in [\theta^\dag, \theta^\dag+n^{-1/4}h]} \left| \frac{d^3 \tilde{l}_\theta}{d\theta^3} \right| = o_{P_*}(1).
	\end{align*}
	Combine all three terms, we can obtain the LAN property of $\tilde{l}$.
\end{proof}

\section{Some properties of the transformation matrix} \label{app:pn}  
The orthogonality of the matrix $P_n$ defined in \S~2.1 tells us that
\begin{align}\label{eq:PropOfPn}
\sum_{j=1}^n p_{ij}^2 = \sum_{i=1}^n p_{ij}^2 = 1, \text{ and } \sum_{j=1}^n p_{ij}p_{kj} = 0 \text{ for any } i,k \in \{1,2,\ldots,n\} \text{ and } i\neq k. 
\end{align}
Using trigonometric identities, we also have that $$\sum_{j=1}^n \cos\frac{2ij\pi}{n+1}=\sum_{j=1}^n \cos\frac{4ij\pi}{n+1} = -1.$$ Then, for all $i,j,k,u,s,t \in \{1,2,\ldots,n\}$,
\begin{align}\label{eq:PropOfPn4}
\sum_{j=1}^n p_{ij}^4 &= \frac{1}{(n+1)^2}\sum_{j=1}^n\left( 1-\cos\frac{2ij\pi}{n+1}\right)^2 \nonumber\\
&=\frac{1}{(n+1)^2}\left[ n-2\sum_{j=1}^n\cos\frac{2ij\pi}{n+1} +\frac{1}{2}\sum_{j=1}^n \left( 1+\cos\frac{4ij\pi}{n+1}\right) \right]= \frac{3}{2(n+1)}.
\end{align}
\begin{align}\label{eq:PropOfPnPn}
\sum_{j=1}^n p_{ij}^2p_{kj}^2 &= \frac{1}{(n+1)^2}\sum_{j=1}^n\left( 1-\cos\frac{2ij\pi}{n+1}\right)\left( 1-\cos\frac{2kj\pi}{n+1}\right) = \frac{1}{n+1}.
\end{align}
Note that by \eqref{eq:PropOfPn}, $\sum_{i=1,i\neq u, i\neq k}^n p_{is}p_{it}= -p_{us}p_{ut}-p_{ks}p_{kt}$.
\begin{align*}
\left| \sum_{i\neq k\neq u} p_{is}p_{it}p_{ks}p_{ut}\right| &= \left| \sum_{u\neq k} p_{ks}^2p_{kt}p_{ut}+ p_{ut}^2p_{us}p_{ks} \right|\\
&= \left| \sum_{k=1}^n p_{ks}^2p_{kt}\sum_{u=1, u\neq k}^n p_{ut}+\sum_{u=1}^n p_{ut}^2p_{us}\sum_{k=1, k\neq u}^n p_{ks} \right|  \\
& \lesssim \sum_{k=1}^n p_{ks}^2(n-1)\frac{2}{n+1}+\sum_{u=1}^n p_{ut}^2(n-1)\frac{2}{n+1} \stackrel{\eqref{eq:PropOfPn}}{=} O(1). 
\end{align*}
The last inequality holds because $p_{kt}\le \sqrt{\frac{2}{n+1}}$.
Similarly,
\begin{align*}
\left| \sum_{i\neq k} p_{is}p_{it}p_{ks}p_{kt}\right| &= \left| \sum_{k=1}^n (-p_{ks}p_{kt}) p_{ks}p_{kt} \right|= \sum_{k=1}^n p_{ks}^2p_{kt}^2 \stackrel{\eqref{eq:PropOfPnPn}}{=} \frac{1}{n+1}.
\end{align*}

For the summation of $p_{ij}$ over $i$,
\begin{align*}
\sum_i p_{ij} & = \sqrt{\frac{2}{n+1}}\sum_{i=1}^n \frac{\cos\frac{(2i-1)j\pi}{2(n+1)}-\cos\frac{(2i+1)j\pi}{2(n+1)}}{2\sin\frac{ij\pi}{n+1}} \\
& = \left\{
\begin{array}{ll}
0, & \text{when $j$ is even,} \\
\sqrt{\frac{2}{n+1}}\cot \frac{j\pi}{2(n+1)}, & \text{when $j$ is odd.}
\end{array}
\right.
\end{align*}
Since $\sqrt{\frac{2}{n+1}}\cot \frac{j\pi}{2(n+1)}\lesssim \frac{\sqrt{n}}{j}$, we have
\begin{align}\label{eq:sumpij}
\left(\sum_{i=1}^n p_{ij}\right)^2 \lesssim \frac{n}{j}.
\end{align}

\section{Additional Proofs}\label{AddProfApend}

\subsection*{Proof of \eqref{eq:JPropkth}-\eqref{eq:JWconvmulti}}\label{app:jump}

We first derive the limiting behavior of the moments of  $\sum_{j=1}^{n} p_{ij} \Delta_j J$. For the second moment, since $E_*[\Delta_1 J] = 0$, it suffices to focus on the terms corresponding to  $(\Delta_j J)^2$:
\begin{align}\label{eq:JWconv}
E_*\left[ \left(\sum_{j=1}^{n} p_{ij} \Delta_j J\right)^2 -\frac{[J]}{n}\right] &= E_*\left[ \sum_{j=1}^{n} p_{ij}^2 (\Delta_j J)^2 \right] - \frac{1}{n} \int_{\mathbb{R}}x^2 v(dx) \nonumber\\
&= E_*\left[  (\Delta_1 J)^2 \right]\sum_{j=1}^{n} p_{ij}^2 - \frac{1}{n} \int_{\mathbb{R}}x^2 v(dx) \stackrel{\eqref{eq:PropOfPn}}{=} 0.
\end{align}
For the higher moments, i.e. $k = 3,4,\ldots, 16$,
\begin{align*}
E_*\left[ \left(\sum_{j=1}^{n} p_{ij} \Delta_j J\right)^k \right] &= E_*\left[ \sum_{b_1+\ldots+b_n=k} \binom{k}{b_1,b_2,\ldots,b_n} \prod_{j=1}^n p_{ij}^{b_j} (\Delta_j J)^{b_j} \right] \\
&= \sum_{b_1+\ldots+b_n=k} \binom{k}{b_1,b_2,\ldots,b_n} \prod_{j=1}^n p_{ij}^{b_j} E_*\left[(\Delta_1 J)^{b_j}\right].
\end{align*}
Similarly, we can assume $b_j\ge 2$, for all $j\in\{1,2,\ldots,n\}$. Otherwise, the terms inside the summation become zero since $E_*[\Delta_1 J] = 0$. We combine the terms of which the set $\{b_1,b_2,\ldots,b_n\}$ contains the same elements. The size of the combined group should be less than $n^a$, where $a$ is the number of the nonzero $b_j$s. Let $T$ denote the set of all possible $b_j$s such that $b_1+\ldots+b_n=k$ and
$b_1\le\ldots \le b_n$. Then,
\begin{align*}
E_*\left[ \left(\sum_{j=1}^{n} p_{ij} \Delta_j J\right)^k \right] &\le \sum_{T} \binom{k}{b_1,b_2,\ldots,b_n} n^a\prod_{j=1}^a \left|p_{ij}^{b_j} E_*\left[(\Delta_1 J)^{b_j}\right]\right|
\end{align*}
Since $b_j\ge 2$, $|p_{ij}|^{b_j} \le n^{-b_j/2} =O(n^{-1})$. By (\ref{eq:JPropkth}-i), for $k = 3,4,\ldots, 16,$
\begin{align}\label{eq:sumpjj}
E_*\left[ \left(\sum_{j=1}^{n} p_{ij} \Delta_j J\right)^k \right] &\lesssim \sum_{b_1+\ldots+b_n=k,b_1\le\ldots \le b_n} n^a\prod_{j=1}^a n^{-1} n^{-1} = O(n^{-1}).
\end{align}
We further prove a stronger result for the fourth moment:
\begin{align}
E_*\left[ \left(\sum_{j=1}^{n} p_{ij} \Delta_j J\right)^4 \right] &\lesssim \sum_{j=1}^{n} p_{ij}^4E_*\left[\Delta_j J^4\right] + \sum_{j=1}^{n}\sum_{k=1}^{n} p_{ij}^2p_{ik}^2E_*\left[\Delta_j J^2\right] E_*\left[\Delta_k J^2\right] \nonumber\\
& =E_*\left[\Delta_1 J^4\right]\sum_{j=1}^{n} p_{ij}^4 + E_*\left[\Delta_1 J^2\right] E_*\left[\Delta_2 J^2\right] \sum_{j=1}^{n} p_{ij}^2 \sum_{k=1}^{n} p_{ik}^2 \nonumber\\
&\stackrel{(\ref{eq:JPropkth}-i)\eqref{eq:PropOfPn4}}{\lesssim} n^{-1}n^{-1}+ n^{-1}n^{-1}\cdot 1 \cdot 1  = O(n^{-2}).
\end{align}

Then, by \eqref{eq:sumpjj}, and since $\mu/n$ is of order $n^{-1}$, for $k = 1,2,\ldots, 16$,
$$E_*\left[\left(\frac{\mu}{n}+ \sum_{j=1}^n p_{ij}\Delta_j J \right)^k\right] = O(n^{-1}).$$
For the fourth moment,
\begin{align*}
E_*\left[\left(\frac{\mu}{n}+ \sum_{j=1}^n p_{ij}\Delta_j J \right)^4\right] & \lesssim \frac{\mu^2}{n^4} + \frac{\mu^2}{n^2} E_*\left[\left(\sum_{j=1}^n p_{ij}\Delta_j J \right)^2\right]\\
&+ \frac{\mu}{n} E_*\left[\left(\sum_{j=1}^n p_{ij}\Delta_j J \right)^3\right] +  E_*\left[\left(\sum_{j=1}^n p_{ij}\Delta_j J \right)^4\right] \\
&= O(n^{-4}+n^{-3}+n^{-2}+n^{-2}) =  O(n^{-2}).
\end{align*}

\subsection*{Proof of \eqref{eq:unif}.}\label{app:unif}
Denote $A = n^{-\frac{1}{2}} \sum_{j=1}^n \left[ \frac{\lambda_j(\theta^\dag)}{\lambda_j(\zeta)}-1 -\log \frac{\lambda_j(\theta^\dag)}{\lambda_j(\zeta)}\right]$, $B = 2\sqrt{n}\log\frac{\sqrt{\theta^\dag} + \sqrt{\theta^\dag+4n\sigma_\varepsilon^2} }{\sqrt{\zeta} + \sqrt{\zeta+4n\sigma_\varepsilon^2}}$, and 
$C = \frac{\sqrt{n}(\theta^\dag-\zeta)}{\sqrt{\zeta(\zeta+4n\sigma_\varepsilon^2)}}$. By (4.3)-(4.5) in \cite{gloter_jacod_2001_est} with $a=\frac{\theta^\dag}{n\sigma_\varepsilon^2}$, $b=\frac{\zeta}{n\sigma_\varepsilon^2}$, $p=1$ and $k=n+1$, for all $n$, and $\zeta\in\Theta'$ with $|\zeta-\theta^\dag|\le \delta$,
\begin{align*}
\left| A - B + C \right| \le \frac{1}{\sqrt{n}}\left|\log\frac{\theta^\dag}{\zeta}\right| + \frac{\left|\theta^\dag - \zeta\right|}{\sqrt{n}\zeta} \le \frac{1}{\sqrt{n}} \left[ \max\left\{\log\frac{\theta^\dag}{\delta_0},\log\frac{\theta^\dag}{\theta^\dag+\delta_0}\right\} + \frac{\delta}{\delta_0}  \right].
\end{align*}
This implies that $|A-B+C|$ converges uniformly to $0$ in probability as $n\rightarrow\infty$. Then, it remains to prove that $|B-C-\Gamma(\theta^\dag,\zeta)|$ converges uniformly to $0$ in probability. Since for $B$, the logarithm can be approximated by Taylor expansion at $1$, and $ \Gamma(\theta^\dag,\zeta) = \frac{\sqrt{\zeta}-\sqrt{\theta^\dag}}{\sigma_\varepsilon} + \frac{\theta^\dag - \zeta}{2\sqrt{\zeta}\sigma_\varepsilon}$, we divide the proof into three parts:
\begin{align}\label{eq:firstorderunif}
&\sup_{\zeta\in\Theta':|\zeta-\theta^\dag|\le \delta}\left| 2\sqrt{n}\left[ \frac{\sqrt{\theta^\dag} + \sqrt{\theta^\dag+4n\sigma_\varepsilon^2} }{\sqrt{\zeta} + \sqrt{\zeta+4n\sigma_\varepsilon^2}} - 1\right] + \frac{\sqrt{\zeta}-\sqrt{\theta^\dag}}{\sigma_\varepsilon}  \right| \stackrel{P_*}{\rightarrow} 0,\\
&\sup_{\zeta\in\Theta':|\zeta-\theta^\dag|\le \delta}\left| 2\sqrt{n}\sum_{k=2}^\infty \left[ \frac{\sqrt{\theta^\dag} + \sqrt{\theta^\dag+4n\sigma_\varepsilon^2} }{\sqrt{\zeta} + \sqrt{\zeta+4n\sigma_\varepsilon^2}} - 1\right]^k\big/k \right| \stackrel{P_*}{\rightarrow} 0,\label{eq:secondorderunif}\\
& \sup_{\zeta\in\Theta':|\zeta-\theta^\dag|\le \delta}\left| \frac{\sqrt{n}(\theta^\dag-\zeta)}{\sqrt{\zeta(\zeta+4n\sigma_\varepsilon^2)}} -  \frac{\theta^\dag - \zeta}{2\sqrt{\zeta}\sigma_\varepsilon} \right| \stackrel{P_*}{\rightarrow} 0 \label{eq:restunif}
\end{align}

For \eqref{eq:firstorderunif}, we first simply the expression to
\begin{align*}
& \left| 2\sqrt{n}\left[ \frac{\sqrt{\theta^\dag} + \sqrt{\theta^\dag+4n\sigma_\varepsilon^2} }{\sqrt{\zeta} + \sqrt{\zeta+4n\sigma_\varepsilon^2}} - 1\right] + \frac{\sqrt{\zeta}-\sqrt{\theta^\dag}}{\sigma_\varepsilon}\right| \\
= & \left| (\sqrt{\theta^\dag}-\sqrt{\zeta}) \left(\frac{2\sqrt{n}}{\sqrt{\zeta}+\sqrt{\zeta+4n\sigma_\varepsilon^2}} - \frac{1}{\sigma_\varepsilon}\right) + \frac{\sqrt{\theta^\dag+4n\sigma_\varepsilon^2} - \sqrt{\zeta+4n\sigma_\varepsilon^2}}{\sqrt{\zeta}+\sqrt{\zeta+4n\sigma_\varepsilon^2}}\right| \\
= & \left| \frac{\theta^\dag-\zeta}{\sqrt{\theta^\dag}+\sqrt{\zeta}} \frac{2\sqrt{\zeta}}{\sqrt{\zeta}+\sqrt{\zeta+4n\sigma_\varepsilon^2}+\sqrt{4n\sigma_\varepsilon^2}} \right.\\
& \left. + \frac{\theta^\dag - \zeta}{(\sqrt{\zeta}+\sqrt{\zeta+4n\sigma_\varepsilon^2})(\sqrt{\theta^\dag+4n\sigma_\varepsilon^2} + \sqrt{\zeta+4n\sigma_\varepsilon^2})}\right|
\le  \frac{2|\theta^\dag-\zeta|}{2\sqrt{4n\sigma_\varepsilon^2}} + \frac{|\theta^\dag - \zeta|}{4n\sigma_\varepsilon^2}.
\end{align*}
Then, the left hand side of \eqref{eq:firstorderunif} can be bounded by $\frac{2\delta}{2\sqrt{4n\sigma_\varepsilon^2}} + \frac{\delta}{4n\sigma_\varepsilon^2}$, which goes to $0$ 
as $n\rightarrow \infty$. 

For \eqref{eq:secondorderunif}, note that 
\begin{align*}
\left| \frac{\sqrt{\theta^\dag} + \sqrt{\theta^\dag+4n\sigma_\varepsilon^2} }{\sqrt{\zeta} + \sqrt{\zeta+4n\sigma_\varepsilon^2}} - 1 \right| & = \frac{\left| \sqrt{\theta^\dag} + \sqrt{\theta^\dag+4n\sigma_\varepsilon^2} -\sqrt{\zeta} - \sqrt{\zeta+4n\sigma_\varepsilon^2}\right|}{\sqrt{\zeta} + \sqrt{\zeta+4n\sigma_\varepsilon^2}} \\
& \le \frac{\frac{\left|\theta^\dag-\zeta\right|}{\sqrt{\theta^\dag}+ \sqrt{\zeta}} + \frac{\left|\theta^\dag-\zeta\right|}{\sqrt{\theta^\dag+4n\sigma_\varepsilon^2}+ \sqrt{\zeta+4n\sigma_\varepsilon^2}}}{\sqrt{\zeta} + \sqrt{\zeta+4n\sigma_\varepsilon^2}} \le \frac{\delta/\sqrt{\delta_0}}{ \sqrt{n\sigma_\varepsilon^2}}.
\end{align*}
Then, 
\begin{align*}
\sup_{\zeta\in\Theta':|\zeta-\theta^\dag|\le \delta} & \left| 2\sqrt{n}\sum_{k=2}^\infty \left[ \frac{\sqrt{\theta^\dag} + \sqrt{\theta^\dag+4n\sigma_\varepsilon^2} }{\sqrt{\zeta} + \sqrt{\zeta+4n\sigma_\varepsilon^2}} - 1\right]^k\big/k \right|
\le   2\sqrt{n}\sum_{k=2}^\infty \left(\frac{\delta/\sqrt{\delta_0}}{ \sqrt{n\sigma_\varepsilon^2}} \right)^k \\
&= \frac{2\delta^2}{\delta_0\sigma_\varepsilon^2\sqrt{n}} \frac{1}{1- \frac{\delta/\sqrt{\delta_0}}{ \sqrt{n\sigma_\varepsilon^2}}},
\end{align*}
which goes to $0$ as $n\rightarrow \infty$.

For \eqref{eq:restunif}, we can obtain the uniformly convergence for $\zeta$:
\begin{align*}
& \sup_{\zeta\in\Theta':|\zeta-\theta^\dag|\le \delta}\left| \frac{\sqrt{n}(\theta^\dag-\zeta)}{\sqrt{\zeta(\zeta+4n\sigma_\varepsilon^2)}} -  \frac{\theta^\dag - \zeta}{2\sqrt{\zeta}\sigma_\varepsilon} \right|  \\
= & \sup_{\zeta\in\Theta':|\zeta-\theta^\dag|\le \delta}\left| \frac{\sqrt{\zeta}(\theta^\dag-\zeta)}{2\sqrt{\zeta+4n\sigma_\varepsilon^2}(\sqrt{\zeta+4n\sigma_\varepsilon^2} + 2\sqrt{n\sigma_\varepsilon^2})}  \right| \le \frac{\delta \sqrt{\theta^\dag+\delta}}{16n\sigma_\varepsilon^2}  \stackrel{P_*}{\rightarrow} 0.
\end{align*}

Combining all three parts completes the proof.

\end{document}